\documentclass[11pt,a4paper]{article}
\usepackage[utf8]{inputenc}

    \usepackage{fullpage}
    \usepackage[margin=1in]{geometry}       
    \usepackage{amsmath,bm,mathrsfs,mathtools,amssymb}        % great math stuff
    
    \usepackage[
      backend=bibtex,
      bibencoding=ascii,
      citestyle=numeric-comp,
      url=false,
      doi=false,
      eprint=false,
      isbn=false,
      giveninits=true % Abbreviate first names
    ]{biblatex}
    \renewbibmacro{in:}{}
    
    \AtEveryBibitem{
      \clearfield{doi}
      \clearfield{url}
      \clearfield{eprint}
      \clearfield{isbn}
      \clearfield{issn}
    }

    \usepackage{amsfonts}              % for blackboard bold, etc
    \usepackage{amsthm}                % better theorem environments
    \usepackage{bbold}
    \usepackage{amsmath}
    \usepackage{xcolor}
    \definecolor{myblue}{rgb}{0.0352,0.4981,0.6509}
    \usepackage{authblk}

    \newcommand*\diff{\mathop{}\!\mathrm{d}}
    
    \usepackage{cancel}
    \usepackage{bbm}
    \usepackage{bm}
    \usepackage{mathtools}
    \newcommand{\E}[0]{\mathbb{E}}

    \newcommand{\hatw}[0]{\widehat{W}}
    \newcommand{\R}[0]{\mathbb{R}}
    
    \newcommand{\by}{{\mbf{y}}}
    \renewcommand{\P}[0]{\mathcal{P}}
    \newcommand{\hatrho}[0]{\widehat{\rho}}
    \newcommand{\RRd}[0]{\mathbb{R}^d\times \mathbb{R}^d}
    \newcommand{\hatmu}[0]{\widehat{\mu}}
    \newcommand{\N}[0]{\mathbb{N}}
    \newcommand{\norm}[1]{\|#1\|}
    \newcommand{\supp}[0]{\text{supp }}
    
    \newcommand{\W}[0]{\mathcal{W}}
    
    \newcommand{\hatphi}[0]{\widehat{\Phi}}
    
    \DeclareMathOperator*{\argmin}{arg\,min}

    \newcommand{\dd}{\mathrm{d}}
    \usepackage{tcolorbox}
    \usepackage[customcolors,shade]{hf-tikz}
    \newtcolorbox{mybox}[1][]{
        width=\textwidth,
        boxsep=-0cm,
        toprule=0.1pt,
        leftrule=0pt,
        bottomrule=0.1pt,
        rightrule=0pt,
        fontupper=\fontsize{10pt}{10pt},
        nobeforeafter,
        opacityframe=0.3,
        opacityfill=0.15
    }
    \usepackage[unicode=true,
     bookmarks=true,bookmarksnumbered=true,bookmarksopen=true,bookmarksopenlevel=1,
     breaklinks=false,pdfborder={0 0 0},backref=false,colorlinks=true]
     {hyperref}
    \hypersetup{
     linkcolor=blue, urlcolor=marineblue, citecolor=blue, pdfstartview={FitH}, unicode=true}
    \usepackage{graphicx,subfigure,color}       % to include figures            
    \usepackage{epstopdf}
    \usepackage{tabularx}
    \usepackage{subfigure}
    \usepackage{caption}
    \usepackage{algorithm}
    \usepackage{algpseudocode} 
    \usepackage{enumerate} 			         % for enumerate with step
    \usepackage{bbm}             		         % for characteristic function 1
    \usepackage[normalem]{ulem}
     \usepackage{booktabs}
    \usepackage[section]{placeins}
    \usepackage[english]{babel}
    \usepackage{times,verbatim}
    
    \numberwithin{equation}{section}
    \newtheorem{theorem}{Theorem}[section]
    \newtheorem{remark}[theorem]{Remark}
    \newtheorem*{remark*}{Remark}
    
    \newtheorem{proposition}[theorem]{Proposition}
    \newtheorem{lemma}[theorem]{Lemma}

    \newtheorem{assumption}[theorem]{Assumption}

    \newcommand{\bx}{{\mbf{x}}}

    \usepackage{multirow}
    \usepackage{algpseudocode}

    \newcommand{\mbf}[1]{\boldsymbol{#1}}

    \usepackage{xcolor}
    
    \usepackage{soul}
    \makeatletter
    \@namedef{subjclassname@2020}{\textup{2020} Mathematics Subject Classification}
    \makeatother
    \usepackage{ dsfont }%
    
\title{Sparse identification of nonlocal interaction kernels in nonlinear gradient flow equations via partial inversion}

\author[1]{Jos\'e A. Carrillo}
\author[2]{Gissell Estrada-Rodriguez}
\author[1]{L\'aszl\'o Mikol\'as}
\author[3]{Sui Tang}

\affil[1]{Mathematical Institute, University of Oxford, Woodstock Road, Oxford, OX2 6GG, UK.}
\affil[2]{Department of Mathematics, Universitat Politecnica de Catalunya (UPC) Jordi Girona, 1-3, 08034, Barcelona, Spain}
\affil[3]{Department of Mathematics, University of California,
Santa Barbara, Isla Vista, CA 93117, USA.}

\begin{document}

\maketitle

\begin{abstract}
We address the inverse problem of identifying nonlocal interaction potentials in nonlinear aggre\-gation-diffusion equations from noisy discrete trajectory data.  Our approach involves formulating and solving a regularized variational problem, which requires minimizing a quadratic error functional across a set of hypothesis functions, further augmented by a sparsity-enhancing regularizer.  We employ a partial inversion algorithm, akin to the  
{CoSaMP and subspace pursuit algorithms,} to solve the Basis Pursuit problem. A key theoretical contribution is our novel stability estimate for the PDEs, validating the error functional ability in controlling the 2-Wasserstein distance between solutions generated using the true and estimated interaction potentials. Our work also includes an error analysis of estimators caused by discretization and observational errors in practical implementations. We demonstrate the effectiveness of the methods through various 1D and 2D examples showcasing collective behaviors.
\end{abstract}


% 35Q70 PDEs in connection with mechanics of particlesand systems of particles
% 70F17: inverse problem for particle systems
% 70-08: computation methods for particle systems
% 65F22: ill-posedness and regularization in numerical linear algebra

\section{Introduction}
In this work, we investigate the estimation of interaction potentials for a broad spectrum of nonlocal equations with gradient flow structure \cite{carrillo2015finite,carrillo2019aggregation}. These equations can be written as
\begin{equation}\label{nonlocal}
\begin{cases}
\partial_t\mu =\nabla \cdot [\mu \nabla(H'(\mu)+V(\mathbf{x})+W*\mu)]\ , \quad \mathbf{x} \in \mathbb{R}^d\ , t>0\ ,\\
\mu(\mathbf{x},0) =\mu_0(\mathbf{x})\ ,
\end{cases}
\end{equation}
where $\mu(t, \mathbf{x})\geq 0$ denotes a probability measure; $H(\mu)$ denotes the density of internal energy; $V(\mathbf{x})$ is a confinement potential, and $W(\mathbf{x})$ is an interaction potential governing the nonlocal interaction rules.

Equation \eqref{nonlocal} arises in many applications, from porous medium flows \cite{vazquez2007porous,carrillo2000asymptotic,otto2001geometry} to the study of cell populations \cite{bodnar2006integro,gueron1995dynamics,CS18,carrillo2019population} passing by swarming models for animal movement \cite{topaz2006nonlocal,kolokolnikov2013emergent}. Notably, in cases where diffusion is absent ($H\equiv 0$), Equation \eqref{nonlocal} models aggregation behaviors of large number of particles \cite{VbUKB12,BCLR13,ABCvB14}.  With linear diffusion, where $H= \kappa\mu( \log \mu-1)$ with $\kappa$ the diffusion constant, it transforms into a Fokker-Planck equation with applications in opinion formation \cite{To06,FPTT17,GPY17}, finance \cite{sornette2001fokker,nicolis2011dynamical}, wealth distribution \cite{DMT08}, synchronization \cite{CCHKK14,CGPS18} and many other applications in kinetic theory. With nonlinear diffusion, $H(\mu) = \frac{\kappa\mu^m}{m-1}$ for $m> 1$, it relates to Keller-Segel type models in chemotaxis \cite{keller1971model,BDP06} with volume exclusion \cite{CC06,BCL09,CHMV18,CHVY19}.

A central problem in the qualitative analysis of \eqref{nonlocal}, which has garnered significant attention, is determining the criteria for the interaction potential $W$ that result in solutions exhibiting spontaneous pattern formation or self-organization \cite{VbUKB12,ABCvB14,carrillo2019population}.   Recent research suggest that even simple forms of interaction potentials, such as radial potentials denoted by $W(\mathbf{x}):=\Phi(\lvert\mathbf{x}\rvert)$, are capable of inducing complex collective behaviors \cite{leverentz2009asymptotic,bernoff2011primer,BCLR13,carrillo2014explicit}. Examples of such potentials include polynomial forms such as \( W = \frac{\lvert \mathbf{x} \rvert^3}{3} \) and the Morse potential \( W = -C_A e^{-\lvert \bx \rvert/\ell_A} + C_R e^{-\lvert \bx \rvert/\ell_R} \), which are crucial in modeling attractive and repulsive interactions among large groups of particles. In these numerical and theoretical studies, where the goal is often to reproduce the observed dynamics qualitatively, the interaction potential is often predetermined in an empirical way. 

Advancements in data acquisition technologies, such as digital imaging \cite{EW21} and GPS tracking \cite{nagy2010hierarchical,lukeman2010inferring,tunstrom2013collective}, have made  possible to collect density evolution data for large ensembles of particles leading to important advances such as topological interactions \cite{ballerini2008interaction}. This leads to an intriguing question: is it possible to deduce the interaction rules from such data?  Effective algorithms aligning Equation \eqref{nonlocal} with this observational data are essential. This paper delves into addressing this problem, with the goal of bridging the gap between theoretical models and empirical data.
We propose a variational approach to estimate the interaction potential from observed solution data, that accounts for both discretization errors and observation errors, as described by
\begin{align}\label{ObsData}
\{\rho(t_\ell,\bx_m)+\epsilon_{m}^\ell\}_{m=-M, \ell=1}^{M,L}\ ,
\end{align} 
where $\rho$ is the smooth density of $\mu$ in the sense explained in Section \ref{sec 2}; $(t_\ell,\bx_m)$ represents a uniform mesh in the domain, and $\{\epsilon_{m}^\ell\}$ is the discrete added noise. Specifically, the solution is sought through solving a quadratic minimization problem:
\begin{equation*}
\widehat W \in \argmin_{\Psi \in \mathcal{H}}\tilde{\mathcal{E}}_{\infty}(\Psi)\ ,
\end{equation*} 
with
\begin{equation*}
\tilde{\mathcal{E}}_{\infty}(\Psi)= 
\frac{1}{T}\int_0^T\int_{\mathbb{R}^d} \|\nabla\Psi*\rho-\nabla W*\rho\|^2 \rho(t,\bx)\dd\bx \diff t\ .   
\end{equation*}
Here, $\mathcal{H}=\textnormal{span}\{\Psi_i \}_{i=1}^n$ represents a hypothesis function space and $\widehat{W}$ is the estimated potential by our method. Due to the ill-posedness of the inverse problem \cite{lang2021identifiability,lu2021learning,li2021identifiability,tang2023identifiablility}, the solutions may not be unique or can not be stably recovered given the perturbed data.  We propose to regularize the inverse problem by promoting sparsity, motivated by the insight that many interaction potentials are simple functions sparse with respect to certain basis functions.

From an algorithmic perspective, our variational functional is composed of two key elements: a quadratic data fidelity term, that performs interaction force matching, and a sparsity-promoting regularizer. This formulation aligns with addressing a Basis Pursuit (BP) problem \cite{wright2022high}, commonly encountered  in the realm of compressed sensing. While numerous state-of-the-art algorithms exist for solving BP problems, finding an algorithm that is specifically tailored and effective for a given setting remains a significant challenge. 

One of our main contributions in this work is that we propose the PartInv (Partial Inversion) algorithm to solve the BP problem arising in our context. This algorithm excels at handling highly coherent columns in the regression matrix, a phenomenon frequently observed empirically across numerous physical examples, and in particular in the ones considered in Section \ref{sec: numerics}. Its effectiveness is further enhanced by incorporating support pruning (see Section \ref{sec: support_pruning}), which integrates residual data loss with time evolution error analysis. We have intensively tested our algorithm on both one and two dimensional examples, and the results demonstrate its remarkable effectiveness and superiority over alternative methodologies. Our work builds upon and extends the findings of \cite{lang2020learning}, which primarily focused on aggregation equations with linear diffusion and noise-free solution data in one dimension. We have also made contributions by integrating a distinct regularization technique and by extending our study to more complex scenarios, including those involving nonlinear diffusion terms and noisy data.

On the other hand, we also establish new stability estimates for \eqref{nonlocal},
controlling the 2-Wasserstein distance between the solution generated using $\widehat W$ and the solution generated  with $W$ in \eqref{nonlocal} in terms of the error functional $\tilde{\mathcal{E}}_{\infty}$.  This analysis, which is new and not present in relevant papers \cite{lang2020learning,bongini2017inferring}, reinforces the theoretical interpretability of our estimators in reproducing training data.
This stands in contrast with other residuals used in partial differential equation (PDE) discovery, such as those based on the strong or weak form of the PDEs, where no such interpretability exists. In particular, in the case of no diffusion, we show that the target functional $\mathcal{\tilde{E}}_\infty$ can be interpreted as the $\Gamma$-limit of a sequence of analogous error functionals $\mathcal{\tilde{E}}_N$ (see \eqref{eq:particle_error_functional}) which depend on sequences of approximating particle systems. In doing so, we are able to sharpen \cite[Theorem 1.1]{bongini2017inferring} by showing that the minimizer of $\mathcal{\tilde{E}}_\infty$ arising as the limit of a sequence of minimizers to $\mathcal{\tilde{E}}_N$ is the interaction potential driving the dynamics of the particle system in the $N\to \infty$ limit.

Finally, we conduct a comprehensive error analysis for the estimators. This analysis builds upon and extends the methodologies outlined in \cite{lang2020learning}. Our extension applies these methods to scenarios that include nonlinear and noisy solution data. This broader approach enables a more versatile application of the estimators, catering to a wider range of real-world conditions where noise and nonlinearity are common challenges.

Our work can be recast in the nowadays surging mathematical field arising from the blending of machine learning tools and numerical PDEs for the data-driven discovery of partial differential equations. This trend has received considerable attention in recent years aiming to autonomously decipher underlying dynamics from available data. This pursuit introduces a challenging inverse problem, where sparsity-promoting techniques have proven to be a potent means of uncovering robust estimators. Pioneering efforts, including the Sequentially Thresholded Least Squares (SINDy) \cite{rudy2017data} and variants of LASSO algorithms \cite{kang2021ident,rudy2019data}, as well as iterative greedy algorithms such as subspace pursuit \cite{he2020robust,he2023group} and advanced gradient descent algorithms solving 
$L^1$
  minimization \cite{schaeffer2017learning}, typically address the inverse problem by posing it as an optimization problem. Frequently, the strong form of a PDE is employed as the data fidelity term in the loss functional within these works. A novel approach utilizing the weak form of the PDE  \cite{messenger2021weak} has exhibited superior robustness to noise, mitigating its impact when approximating derivatives. Nevertheless, a drawback lies in the often problem-dependent theoretical foundation of these methods, lacking a comprehensive connection to the differential equation itself. While these methodologies possess a general applicability, when applied to specific types of differential equations, a nontrivial effort is needed to devise effective algorithms tailored to those particular equations.

A notably active research vein is the data-driven discovery within particle-based systems. For instance, \cite{he2022numerical} explored the identification of non-local potentials in aggregation equations by addressing a regularized $L^1$ minimization problem through PDE residuals, employing operator splitting techniques. Despite showcasing superior empirical performance, a theoretical understanding remains elusive. { In \cite{bunne2022proximal}, the authors proposed a method to reconstruct particle trajectories from snapshots, interpreting them as collective realizations of a causal JKO scheme \cite{jordan1998variational}. A similar idea was adopted in \cite{terpin2024learning} to learn diffusion terms from observational data, and in \cite{pietschmann2022data} for variational data assimilation for gradient flows.} In another instance, \cite{messenger2022learning} employed a weak SINDy approach to discern mean-field overdamped equations from particle-level data. This method contemplates input training data simulated from microscopic SDEs without external noise or microscopic ODEs  with external noise. The potential effectiveness of these approaches within our problem context is promising.

The work most closely aligned with ours is presented in \cite{lang2020learning}, which focuses on the nonparametric inference of non-local interaction potentials in aggregation equations with linear diffusion. It generalized the previous work on learning interaction kernels on  microscopic ODEs \cite{lu2019nonparametric,lu2020learning,lu2021learning, miller2023learning} and SDEs \cite{lu2020learning} to PDEs.   While using the same data-fidelity term in the loss functional, they employed Tikhonov regularization. They demonstrated that such data-fidelity term is, in fact, the maximum likelihood by looking at the connections with the microscopic SDE counterpart. Further, they show that the kernel identification in the mean-field equations is ill-posed \cite{lang2021identifiability}, requiring effective regularization techniques. 

Finally, we note that our identification problem bears significant resemblance to the deconvolution problem \cite{bigot2019estimation} encountered in image processing. In the latter, the objective is to recover the image from corrupted data samples, which are the result of convolving the image with a known kernel. In our context, we are concerned with solving a deconvolution problem constrained by a PDE, which introduces unique challenges. For example, the unknown coefficient is nonlinearly dependent on the observational data.  Consequently, traditional algorithms from image processing cannot be directly applied, demanding innovative approaches to navigate the complexities introduced by the PDE constraints.

The rest of this paper is organized as follows. In Section \ref{sec 2}, we introduce the notation and the mathematical set up of the inverse problem considered. In Section \ref{sec: stability estimates}, we present the stability estimates in terms of the error functional $\mathcal{\tilde{E}}_\infty$ as well as the $\Gamma$-limit result in the case of no-diffusion. In Section \ref{sec: algorithms}, we present the bounds on the numerical discretization errors incurred in the implementations of the solution method. We present numerical examples illustrating our results in Section \ref{sec: numerics}. In Section \ref{sec:conclusion_and_future_work} we present some conclusions and future perspectives. We include most proofs and auxiliary results in the appendices.

\section{A regularised variational approach via basis pursuit}\label{sec 2}

In this section, we describe the proposed method to identify the interaction potential from a single set of continuous-time trajectory data.  This approach entails addressing a variational problem which is comprised of a data-fidelity term for interaction force approximation and an $\ell^1$ 
  regularization term to promote sparsity. While acquiring continuous-time observational data is not feasible in real-world scenarios, the theoretical framework provided here forms the cornerstone for the computational estimators we later propose for discrete data.

  \subsection{Notation}\label{sec Notation}
 In what follows, unless specified otherwise, we use $\norm{\cdot}$ to denote the Euclidean norm in $\R^d$ or the Frobenius norm when treating matrices. In addition, $\|\cdot\|_p$ denotes the $p$ norm for a vector. When $p=0$, it means the number of nonzero entries in a vector. The complex transpose of $\mathbf{B}$ is denoted by $\mathbf{B}^*$, and its transpose by $\mathbf{B}^{\top}$.  {We use $\sigma_{min}(\mathbf{B})$ to denote the minimal singular value of $\mathbf{B}$}. The pseudo-inverse of $\mathbf{B}$ is represented as $\mathbf{B}^{+}$. For an index set $I \subset \{1, \cdots, p\}$, the submatrix of $\mathbf{B}$ formed by selecting row indices in $I$ is denoted by $\mathbf{B}_I$ and belongs to $\mathbb{R}^{|I| \times q}$, where $|I|$ represents the cardinality of the set $I$. We will denote by $\widetilde{I}$ the complement of the index set $I$, i.e. if $I \subset \{1,\ldots,p\}$, then $\widetilde{I} = \{1,\ldots,p\} \backslash I $. 
Let $\mathbf{B} \in \mathbb{R}^{p \times q}$ be a matrix. Let $\mathbf{c} \in \mathbb{R}^p$ denote a vector, then $\mathbf{c}(I) \in \mathbb{R}^{|I|}$ is the restriction of $\mathbf{c}$ on $I$. For integers $m,n,p$, we use a Matlab notation $m:p:n$ to represent the array with values starting at $m$, augmenting by $p$, and ending at or before $n$.

Other relevant notation used in this paper is summarized in Table \ref{tab: notation}.

\begin{table}[h!]
\begin{center}
\begin{tabular}{l l} 
\hline
Notation &  Description\\ [0.5ex] 
 \hline\specialrule{0.1em}{0.1em}{0em}
 $W(\bx):=\Phi(|\bx|)$ & Interaction potential \\ 
 \hline
 $\mathcal{W}^{k,p}(\R^d)$ & Sobolev space with $k$ derivatives in $L^p(\R^d)$ \\ 
 \hline
 $\rho(t,x)\diff \bx=\diff \mu$ & Solution of the PDE and its density \\ 
  \hline
 $\phi,\ \Phi$ & True interaction kernel and potential \\ 
  \hline
  $\psi,\ \Psi$ & Estimated interaction kernel and potential \\
 \hline
 $\mathcal{E}_\infty(\Psi)$ & Error functional, see \eqref{eq: errorfunctional} \\
 \hline
  $\mathcal{E}_{n,M,L}(\Psi)$ & Discretized error functional\\ 
\hline
{$F(\rho,\mathbf{x}):=\rho \nabla(H'(\rho)+V(\bx) )$}
  & {Local part of the flux}\\
  \hline
\end{tabular}
\caption{A first glance to the most important notations.}\label{tab: notation}
\end{center}
\end{table}
{ Finally, we note that for a curve $\gamma \in C([0,T],X)$ for any metric space $X$,  we will denote the evaluation of the curve at some time $t \in [0,T]$ as {$\gamma_t$ throughout this paper.} 
If $\gamma \in C([0,T],\P^2(\R^d))$,  where $ \P^2(\R^d)$ denotes the space of probability measures with finite second moments for any $t \in [0,T]$, we denote the $L^2$-norm with respect to this curve as 
\begin{equation*}
    \norm{f}^2_{L^2(\gamma_t)}:= \int_{\R^d}|f(\bx)|^2\dd \gamma_t(\bx) \ ,
\end{equation*}
and by $d_2:\P^2(\R^d) \times \P^2(\R^d) \to \R^d$ the 2-Wasserstein distance defined as 
\[
d_2(\mu,\xi):=\min \left\{\int_{\R^d \times \R^d}|\bx-\by|^2 \dd \gamma(\bx,\by) : \gamma \in \Pi(\mu, \xi)\right\}^{^{\frac{1}{2}}} \ ,
\]
where $\mu, \xi \in \P(\R^d)$, $\Pi(\mu,\xi):= \{\gamma \in \P^2(\R^d \times \R^d) \ | \ (\pi_{\bx})\#\gamma = \mu, \ (\pi_{\by})\#\gamma = \xi\}$ is the set of transport plans between the measures $\mu$ and $\xi$ and, for any measure $\nu \in \P(\R^d)$, and measurable set $A\subseteq \R^d$ $\pi_x{\#\nu}:=\nu(\pi_x^{-1}(A))$ is the push-forward measure by the projection map to the first coordinate given by $\pi_x(x,y) = x$ and analogously for $y$.}

\subsection{The error functional}
Let $\mu:[0,T] \times \mathbb{R}^d \rightarrow \mathbb{R}$ be a solution of the PDE \eqref{nonlocal} in which $W$ is the target interaction potential to be learned. For any $t\in [0,T]$, assume $\mu(t,\bx)=\mu_t(\bx)$ has a smooth density $\rho:[0,T] \times \R^d \to \R$ with respect to the Lebesgue measure, i.e. $\dd \mu_t = \rho(t,\bx) \dd \bx$, and it decays fast enough as $|x|\to\infty$ for all $t\in [0,T]$.  As the Equation \eqref{nonlocal} is linear in $W$, it can be written as 
\begin{equation}
    \partial_t \rho = \nabla\cdot (\rho \ L_{\rho}W + F(\rho,\mathbf{x}))\ ,
\end{equation} 
where $L_{\rho} W:=\nabla W*\rho$ and 
$F(\rho,\mathbf{x}):=\rho \nabla(H'(\rho)+V(\bx) )$.
Let us assume for simplicity that $W\in\mathcal{W}^{2,\infty}(\R^d):=\{b:\R^d \to \R \ | \ \norm{b}_\infty+\norm{\nabla b}_\infty + \norm{\nabla^2 b}_\infty < \infty \}$. It is obvious that the target interaction potential satisfies
\begin{equation}\label{eq: errorfunctionalidea} 
W\in\argmin_{\Psi \in \mathcal{W}^{2,\infty}(\R^d)}\tilde{\mathcal{E}}_{\infty}(\Psi)\ ,
\end{equation} 
with
\begin{equation}\label{eq: error_functional_tilde}
\tilde{\mathcal{E}}_{\infty}(\Psi)= 
\frac{1}{T}\int_0^T\int_{\mathbb{R}^d} \|L_{\rho}\Psi-L_{\rho}W\|^2 \rho(t,\bx)\dd\bx \diff t\ .   
\end{equation}
{ Using the weak formulation of the PDE \eqref{nonlocal} with the test function $ \Psi * \rho$, for $\Psi$ sufficiently smooth, we deduce
\begin{align*}
\tilde{\mathcal{E}}_{\infty}(\Psi)= &\, \frac{1}{T} \int_{0}^{T}\int_{\mathbb{R}^d}\left[\|L_\rho \Psi\|^2+\|L_\rho W\|^2 -2\langle L_\rho\Psi, L_\rho W\rangle\right]{ \rho(t,\bx)} \dd\bx \diff t\\
 =&\, \mathcal{E}_{\infty}(\Psi) + \frac{1}{T} \int_{0}^{T}\int_{\mathbb{R}^d}\|L_\rho W\|^2{ \rho(t,\bx)}\dd\bx \diff t\ ,
\end{align*}
} where
{
\begin{align} \mathcal{E}_{\infty}(\Psi):=
&\,
\, \frac{1}{T}\int_0^T\int_{\mathbb{R}^d}\Bigl[ \|L_{\rho}\Psi\|^2\rho(t,\bx) +2 \Psi*\rho\, \partial_t\rho+{ 2}\nabla\Psi*\rho\cdot F(\rho,\mathbf{x})\Bigr]\dd \bx \diff t . 
\end{align}
Notice that we prefer to use the weak solution concept of the PDE \eqref{nonlocal} to avoid the potential loss of regularity that happens for nonlinear degenerate diffusions at the tip of their supports, \cite{carrillo2019aggregation,vazquez2007porous}.} 
 
Given a finite dimensional subspace $\mathcal{H}=\textnormal{span}\{\Psi_i\}_{i=1}^n\subset \mathcal{W}^{2,\infty}(\R^d)$, we propose to approximate $W$ by minimizers of the following functional:
{
\begin{align}\label{eq: errorfunctional}  \widehat W\in &\argmin_{\Psi \in \mathcal{H}}\mathcal{E}_{\infty}(\Psi) \ , 
\\ 
 \mathcal{E}_{\infty}(\Psi):&={{\frac{1}{T}\int_0^T\int_{\mathbb{R}^d}\Bigl[ \|L_{\rho}\Psi\|^2\rho(t,\bx) +2 \Psi*\rho\, \partial_t\rho+{ 2}\nabla\Psi*\rho\cdot F(\rho,\mathbf{x})\Bigr]\dd \bx \diff t} \ , \nonumber
}\end{align}
}where $\widehat W$ is the identified potential by our method. We note that the error functional $\mathcal{E}_\infty$ promotes the matching of the interaction force with the ground truth. In fact, the previous computation shows that 
\begin{equation}\label{eq: errorfunctional1} \argmin_{\Psi \in \mathcal{H}}\mathcal{E}_{\infty}(\Psi)= \argmin_{\Psi \in \mathcal{H}}\tilde{\mathcal{E}}_{\infty}(\Psi) \ .
\end{equation} 
 In Section \ref{sec: stability estimates} we present stability estimates showing that the 2-Wasserstein distance between solutions of \eqref{nonlocal} corresponding to the ground truth interaction potential and the learned one can be controlled by the functional $\mathcal{\widetilde{E}}_\infty$. Namely, we  present results of the following type.
 \begin{proposition}\label{prop: informal_stability_estimate}
    Let $\mu, \widehat{\mu} \in C([0,T],\mathcal{P}^2(\R^d))$ be solutions of $\eqref{nonlocal}$ with the interaction potential and external potential $(W,V), (\widehat{W},\widehat{V})$ respectively. Then, under suitable regularity conditions on the velocity fields of $\mu$ and $\widehat{\mu}$, we have the following stability estimate 
    \begin{equation}\label{eq:stab_estimate_1}
    d^2_2(\mu(t),\hatmu(t)) \leq C \left(\tilde{\mathcal{E}}_\infty(\widehat{W}) + \int_0^t\norm{\nabla V - \nabla \widehat{V}}^2_{L^2(\mu_s)} ds+ d^2_2(\mu(0),\hatmu(0)) \right)\ ,
    \end{equation}
    where $C>0$ is a suitable constant depending only on the final time $T$ and the Lipschitz constant of the functions involved.
\end{proposition}

{
Note that we can write the stability estimate in \eqref{eq:stab_estimate_1} in terms of $\tilde{\mathcal{E}}_\infty$ or $\mathcal{E}_\infty$ because of \eqref{eq: errorfunctional1}.} In our learning problem, we have $\widehat{V}=V$ and assume $W$ is the only unknown. The statement in Proposition \ref{prop: informal_stability_estimate} shows that the minimization of $\mathcal{\widetilde{E}}_\infty$ by \eqref{eq: errorfunctional1}  leads to minimization in the difference between trajectories of solution to \eqref{nonlocal} corresponding to the ground truth interaction potential $W$ and the estimated interaction potential $\widehat{W}$. In particular, whenever we have an estimator such that $\tilde{\mathcal{E}}_\infty(\widehat W)=0$, then $\hat \mu =\mu$ on $[0,T]$.  However, the error functional  $\tilde{\mathcal{E}}_\infty$ depends on the \textit{unknown} ground truth through the term $L_\rho W$, so it is not feasible in any practical computational scheme. Minimizing ${\mathcal{E}}_\infty$ enables practical implementation using only the data $\rho(t,\bx)$.  

\subsection{The vanilla least squares solutions}
In this study, our focus will be on identifying radial interaction potentials  $W$, but our computational framework can be extended to general potentials as well. From now on, we will always restrict to this class of potentials.

Let us assume that $\mathcal{H}$ is a linear subspace generated by the basis of radial functions $\textnormal{span}\{\Psi_i\}_{i=1}^n$. Let us introduce the notation $\nabla{\Psi_i}(\bx)=\psi_i(|\bx|)\frac{\bx}{|\bx|}$. 
Since $\mathcal{E}_{\infty}(\cdot)$ is a quadratic functional, then we can rewrite the minimization problem by means of a simple matrix representation. To simplify the notation, we omit the time dependence of the solution $\rho$ in the rest of the section. We  first introduce the following bilinear form 
\begin{align}\label{eq: computation of G}
    \langle\Psi_i,\Psi_j \rangle_G & =\frac{1}{T}\int_0^T\int_{\mathbb{R}^d}\langle L_{\rho}\Psi_i, L_{\rho}\Psi_j\rangle \rho\diff \bx\diff t\nonumber\\ &=\frac{1}{T}\int_0^T\int_{\mathbb{R}^d}\Bigl[\int_{\mathbb{R}^d}\int_{\mathbb{R}^d} \nabla{\Psi_i}(\textbf{y})\cdot\nabla{\Psi_j}(\textbf{z})\rho(t,\bx-\textbf{y})\rho(t,\bx-\textbf{z})\diff \textbf{y}\diff \textbf{z}\Bigr] \rho\diff \bx\diff t\nonumber\\ &=\frac{1}{T}\int_{\mathbb{R}^d}\int_{\mathbb{R}^d}\nabla{\Psi_i}(\textbf{y})\cdot\nabla{\Psi_j}(\textbf{z})\Bigl[\int_0^T\int_{\mathbb{R}^d} \rho(t,\bx-\textbf{y})\rho(t,\bx-\textbf{z})\rho(t,\bx)\diff \bx\diff t\Bigr]\diff \textbf{y}\diff \textbf{z}\nonumber\\ &=\frac{1}{T}\int_{\mathbb{R}^d}\int_{\mathbb{R}^d}\psi_i(|\textbf{y}|){\psi_j(|\textbf{z}|)}I_G(\textbf{y},\textbf{z})G(\textbf{y},\textbf{z})\diff \textbf{y}\diff \textbf{z}\ ,
\end{align} 
where $I_G(\textbf{y},\textbf{z})=\frac{\textbf{y}}{|\textbf{y}|}\cdot\frac{\textbf{z}}{|\textbf{z}|}$ with
\begin{equation}\label{ggg}
    G(\textbf{y},\textbf{z})=\int_0^T\int_{\mathbb{R}^d} \rho(t,\bx-\textbf{y})\rho(t,\bx-\textbf{z})\rho(t,\bx)\diff \bx\diff t\ ,
\end{equation}
for all $\textbf{y},\textbf{z}\in \mathbb{R}^d$.
Let $ \Psi=\sum_{i=1}^n {c}_i\Psi_i$, for $i=1,\cdots,n$\ ,\  and define
\begin{align}
    {A}_{ij}&=\langle\Psi_i,\Psi_j \rangle_G\ ,\label{eq: A new}
    \\
    b_i&={ -{\frac{1}{T}\int_0^T\int_{\mathbb{R}^d}\Bigl[ \partial_t\rho(\Psi_i*\rho)+(\nabla\Psi_i*\rho)\cdot{F}(\rho,\bx)\Bigr]\diff \bx\diff t}\ .\label{eq: b new}}
    \end{align}
    Notice that $b_i = \langle \Psi_i, W\rangle_G$ when we assume that $\rho $ is the exact solution to \eqref{nonlocal}. Then we can write the error functional as
    \begin{equation}
    \mathcal{E}_{\infty}(\mathbf{c})=\mathbf{c}^T\mathbf{A}\mathbf{c}-2\mathbf{b}^{T}\mathbf{c}\ .\label{eq: estimator matrix form}
    \end{equation}
    By first-order optimality, the optimal solutions satisfy the normal equation 
    \begin{align}\label{normalequation}
        \mathbf{A}\mathbf{c}=\mathbf{b}\ .
    \end{align}
In the context of inverse problems, the system introduced in \eqref{normalequation} often displays inherent challenges due to its ill-posed nature. Specifically, even in the scenario where \( W \) is an element of the Hilbert space \( \mathcal{H} \), uniqueness of the solution is not guaranteed. Moreover, the task of matrix inversion, particularly of \( \mathbf{A} \), is fraught with numerical instabilities, often enhanced by perturbations such as discrete-time data and observational noise. For an extended discussion on this topic, the reader may consult \cite{lang2021identifiability}.

Identifying effective regularization methods to stabilize the recovery process remains a critical challenge. Our numerical studies reveal that the regularized least squares estimator, derived using the pseudoinverse, did not perform satisfactorily. Lang et al. \cite{lang2020learning} explored a Tikhonov regularization for aggregation equations with linear diffusion. In this case, the standard least squares estimators, as per Equation \eqref{normalequation}, works as maximum likelihood estimators. However, this property does not extend to cases of nonlinear diffusion, necessitating the exploration of alternative regularization approaches.

\subsection{$\ell^1$ regularization via Basis Pursuit} \label{subsec: bp_intro}

We note that in many prototypical examples, the true interaction kernels are typically simple functions and are often sparse with respect to a set of given basis functions, such as polynomials. Building upon this prior knowledge, we propose estimating the interaction kernel by tackling the following BP problem:
\begin{equation}
\begin{aligned}\label{orignalbp}
    \text{minimize }_{\mathbf{c}\in \mathbb{R}^n} \|\mathbf{c}\|_1\ , \\
    \text{subject to } \mathbf{A}\mathbf{c} = \mathbf{b}\ .
\end{aligned}
\end{equation}
This approach seeks to minimize the \( \ell_1 \)-norm of the coefficient vector \( \mathbf{c} \) within the real vector space \( \mathbb{R}^n \), subject to the constraint that the product of matrix \( \mathbf{A} \) and vector \( \mathbf{c} \) equals the vector \( \mathbf{b} \).

BP problems have been actively studied in the area of compressed sensing \cite{wright2022high}. Many state-of-the-art algorithms  such as the CoSaMP algorithm \cite{needell2009cosamp} and the closely related subspace pursuit algorithm \cite{dai2009subspace} 
are designed to address the BP problem. These iterative {greedy} algorithms are particularly acclaimed for their superior  recovery in the noisy data regime, combined with their rapid computational efficiency. 
  However, their guaranteed performance often depends on certain properties of the sensing matrix \( \mathbf{A} \). Specifically, these properties include: 
  \begin{itemize}
      \item Coherence of \( \mathbf{A} \) should be small enough  (ideally, smaller than $\frac{1}{2\|\mathbf{c}\|_0}\leq \frac{1}{2}$), meaning that the maximal correlation between the normalized columns of \( \mathbf{A} \) is small (see Proposition 3.2 in \cite{wright2022high});
      \item \( \mathbf{A} \) should act almost as an isometry on the set of sparse vectors, a property known as the Restricted Isometry Property (RIP) \cite{candes2006stable, candes2006near}.
  \end{itemize}
  
Such characteristics are typically satisfied by random matrices. In our context, the sensing matrix \( \mathbf{A} \) is determined by the inherent physics of the problem and, consequently, is non-random. Our numerical experiments revealed that in all examples, the matrix \( \mathbf{A} \) possesses highly coherent columns {(where the coherence of a matrix is the maximum absolute correlation of its columns)}, so that the incoherence parameter is very close to 1 (see Figure \ref{fig:1d_incoherence}). Consequently, it fails to satisfy the desired RIP. Interestingly, similar challenges with the sensing matrix have been observed in super-resolution problems in imaging. How to perform sparse recovery with a coherent sensing matrix is still an on-going challenge in the signal and image processing community. Only a few works focus on addressing this issue and the algorithms are  heuristic  and short on theoretical justifications \cite{candes2011compressed,fannjiang2012coherence,chen2014guaranteed}.

We propose to use the PartInv Algorithm  \ref{alg:euclid}, a modification of the CoSaMP algorithm \cite{needell2009cosamp} to solve the BP with a coherent sensing matrix. PartInv was originally proposed in \cite{chen2014guaranteed} and showed better performance than existing greedy methods for random matrices, and is especially suitable for matrices that have subsets of highly correlated columns.
 Compared with CoSaMP, the only difference lies in  line 3 of Algorithm \ref{alg:euclid}, where $A_{I^{(k)}}^*$ is replaced by the pseudo-inverse $A_{I^{(k)}}^{+}$. This step can reduce the error propagation due to the coherent columns and one can refer to \cite{chen2014guaranteed} for more details.  Moreover, it enjoys partial theoretical justification. More precisely, \cite[Theorem 3.1]{chen2014guaranteed} provides a sufficient condition that uses a weaker condition than RIP and incoherence bounds to prove the success of the algorithm on sparse recovery. { Notice that, in the following algorithm, knowing the exact sparsity is not needed and we only require an upper bound denoted by $K$. We recall that for an index set $I\subset\{1,\ldots,n\}$, we denote $\widetilde{I}=\{1,\ldots,n\}\backslash I$.}

\begin{algorithm}
\caption{Given $\mathbf{Ac} = \mathbf{b}$ where the ground truth is $s$-sparse,  return  the best $K$-sparse approximation $\mathbf{\hat c} $ (see Section \ref{sec Notation} for the notation).}\label{alg:euclid}
\begin{algorithmic}[1]
\Require $\mathbf{A, b}, K$ (an upper bound on sparsity $s$)
\State $\mathbf{\tilde c} \leftarrow \mathbf{A}^* \mathbf{b}$; $I^{(0)} \leftarrow $ indices of the $K$-largest magnitudes of $\mathbf{\tilde c}$; $k \leftarrow 0$
\While{Stopping condition not met}
    \State $\mathbf{\tilde c}_{I^{(k)}} \leftarrow \mathbf{A}^{+}_{I^{(k)}} \mathbf{b}$
    \State $\mathbf{r} \leftarrow \mathbf{b} - \mathbf{A}_{I^{(k)}} \mathbf{\tilde c}_{I^{(k)}}$
    \State $J^{(k)} \leftarrow \widetilde{I}^{(k)}$  
    \State $\mathbf{\tilde c}_{J^{(k)}} \leftarrow \mathbf{A}^*_{J^{(k)}} \mathbf{r}$
    \State $I^{(k+1)} \leftarrow $ indices of $K$-largest magnitude components of $\mathbf{\tilde c}$
    \State $k \leftarrow k+1$ 
\EndWhile
\State \textbf{Return} $\mathbf{\hat c}=\mathbf{A}_{I^{(k)}}^{+}\mathbf{b}$
\end{algorithmic}
\end{algorithm}

\section{Stability estimates and $\Gamma$-convergence}\label{sec: stability estimates}

As we anticipated in the Introduction, in this section we present stability estimates for the 2-Wasserstein distance between solutions of \eqref{nonlocal} depending on the ground truth interaction potential and an interaction potential estimated with techniques such as Basis Pursuit (see Section \ref{sec 2}).  

 A focal point of interest is comparing solutions that arise from the ground truth interaction potential $W$ with those derived from the learned interaction potential $\widehat{W}$.  We present results for the nonlinear diffusion, which is the focus of this paper, as well as for the cases of no diffusion and linear diffusion. Although the results are analogous, the techniques involved in the proofs differ in each case. To improve the readability of the paper, we start with the simpler case of no diffusion and progressively increase the complexity, concluding with the nonlinear diffusion case.

\subsection{No diffusion, the aggregation equation case}\label{subsec:aggregation_eq}

Let us begin by considering the following interacting particle system 
\begin{equation}\label{eq:IPS}
   \dot \bx_{i}(t) = -\frac{1}{N}\sum_{i \neq j} \nabla W(\bx_{i}(t)-\bx_{j}(t)) - \nabla V(\bx_{i})\ , \quad  i = 1,\ldots,N \ ,
\end{equation}
for particles $(\bx_i)_{i=1}^N \in C([0,T],\R^d)$, an interaction potential $W \in \W^{2,\infty}(\R^d)$  and $V \in \W^{2, \infty}(\R^d)$, a confinement potential with $\supp W \subset \Omega$ and $\supp V \subset \Omega$ for a compact set $\Omega \subset \R^d$. Note that \eqref{eq:IPS} is an ODE system driven by a velocity field analogous to the one in \eqref{nonlocal}, where we have set $H \equiv 0$, i.e. there is no diffusion. 

Under our assumptions on the interaction potential $W$ and the confinement potential $V$, the system \eqref{eq:IPS} is well posed by traditional Cauchy-Lipschitz results for ODEs.
It can be shown that, as the number of particles $N\to \infty$, the sequence of empirical measures $\mu^N_t:= \frac{1}{N}\sum_{i=1}^N\delta_{\bx_i(t)}$ of the solutions to \eqref{eq:IPS} converges in the 2-Wasserstein distance to a probability measure $\mu_t \in \P^2(\R^d)$, where $ \P^2(\R^d)$ denotes the space of probability measures with finite second moments for any $t \in [0,T]$. In turn, the curve $\mu \in C([0,T], \P^2(\R^d))$ solves the following PDE giving a continuum description of the system \eqref{eq:IPS}
\begin{equation}\label{eq:aggregation_eq}
    \partial_t \mu =  \nabla \cdot ( \mu (\nabla W * \mu +  \nabla V))\ .
\end{equation}
 We refer the reader to \cite{golse2016dynamics} for more details. This is the mean-field PDE associated to the dynamical system \eqref{eq:IPS} and it is also referred to as the \textit{aggregation equation}.

In what follows, we will consider $W$ to be the ground truth interaction potential and we will write the error functional $\mathcal{\tilde{E}}_\infty$ in \eqref{eq: errorfunctional1} in terms of a general curve of measures $\mu\in C([0,T],\P^2(\R^d))$ solving \eqref{eq:aggregation_eq} as
\begin{equation}\label{eq:error_functional_mu}
    \mathcal{\tilde{E}}_\infty(\widehat{W})=\frac{1}{T}\int_0^T\int_{\mathbb{R}^d} |\nabla W*{ \mu_t}-\nabla \widehat{W}*\mu_t|^2 \dd \mu_t(\bx) \diff t \ ,
\end{equation}
for any ${\widehat{W}} \in \W^{2,\infty}(\R^d)$. Now we are ready to present our first Dobrushin-type stability result.

\begin{proposition}\label{prop:Dobrushin_}
Let $W,\hatw$, and ${V}$ belong to $\W^{2,\infty}(\R^d)$ with $\supp W,\  \supp V \subseteq \Omega$, where $\Omega\subset \R^d$ is a compact set. For initial data  $\mu_0,\ \widehat{\mu}_0\in \P^2(\R^d)$, let $\mu,\widehat{\mu} \in C([0,T],\P^2(\R^d))$ be solutions to the aggregation equation \eqref{eq:aggregation_eq} with velocity fields { $\nabla W* \mu + \nabla V, \ \nabla \hatw* \widehat{\mu} + \nabla {V}$}, respectively. Then, for any $t \in [0,T]$, we have
\[
d^2_2(\mu_t, \hatmu_t) \leq  C_1 \mathcal{\tilde{E}}_{\infty}( \hatw)   + C_2d_2^2(\mu_0, \widehat{\mu}_0)\ ,
\]
where $C_1$,  and $C_2$ are non-negative constants depending on $T$ and $L_W,L_{\widehat{W}},L_V$, the Lipschitz constants of $\nabla W$, $\nabla\hatw$ and $\nabla V$, respectively.
\end{proposition}

\begin{proof}
    We refer the reader to Appendix \ref{subsec:proof_stability_agg_eq} for the proof of Proposition \ref{prop:Dobrushin_}. 
\end{proof}

\begin{remark}
    Note that the estimate derived in the Appendix \ref{subsec:proof_stability_agg_eq} allows both the interaction and the confinement potentials to differ in \eqref{eq:aggregation_eq},
        \[
d^2_2(\mu_t, \hatmu_t) \leq  C_1 \mathcal{\tilde{E}}_{\infty}( \hatw)+ C_3 \int_0^t\norm{\nabla V - \nabla \widehat{V}}^2_{L^2({\mu}_s)} \dd s   + C_2d_2^2(\mu_0, \widehat{\mu}_0)\ ,
\]
where $C_3$ also depends on $L_{\widehat{V}}$. Thus, with minor modifications to the error functional $\tilde{\mathcal{E}}_{\infty}$ to account for the difference between $V$ and $\widehat{V}$, we could control the 2-Wasserstein distance between $\mu$ and $\hatmu$ in terms of {that} new error functional and the difference in the initial data. This would be particularly relevant if the goal was to infer both the interaction as well as the confinement potential from trajectories of the PDE. We leave this for future work and, in what follows, the reader can set $V=\widehat{V}$ in the result above.

\end{remark}

\begin{remark}\label{remark:use_of_true_or_estimated_measure}
    We note furthermore that an analogous result holds in which the $L^2$ norms depend on the estimated solution $\hatmu$, instead of the ground truth solution $\mu$. Thus, we could rewrite the estimate in Proposition \ref{prop:Dobrushin_} to depend on the minimum of the two norms in each case, but opted for the present specification for the sake of clarity.   
\end{remark}

{ In the following subsections, we provide stability estimates that extend beyond merely estimating the potential $W$. While the numerical section focuses solely on estimating $W$, these broader results in Propositions \ref{prop:agg_diff_estimate} and \ref{prop:non_linear_stability}  pave the way for future research.}

\subsubsection{Mean-field dynamics and $\Gamma$-convergence}
In what follows, without loss of generality, we will assume $V\equiv0$.
Using the stability estimate from Proposition \ref{prop:Dobrushin_}, in this section we present a $\Gamma${-convergence result, i.e. establishing the minimizer of a functional as the limit of  minimizers of a sequence of functionals}, as well as the sharpening of \cite[Theorem 1.1]{bongini2017inferring}. We remark that in \cite{bongini2017inferring} the setting is slightly different to ours, since the error functional considered by the authors depends on the interaction kernel $\nabla W$, instead of the potential $W$. For the reader's convenience, we begin by recalling the notation in \cite{bongini2017inferring} adapted to our setting.  For a compact set $K \subset \R^d$ let 
\[
X_{M,K}:=\{b\in \W^{2,\infty}(K): { \norm{b}_\infty + \norm{\nabla b}_\infty  + \norm{\nabla^2 b}_\infty\leq M} \}\ ,
\]
and $(A^N)_{N\in\N}$ be a family of closed subsets of $X_{M,K}$ with the uniform approximating property { in $L^\infty(K)$}, i.e. for any $b \in X_{M,K}$ there exists a sequence $(b^N)_{N\in\N}$ converging uniformly to $b$ on $K$, such that $b^N \in A^N$ for every $N \in \N$. The authors considered a sequence $(\widehat{W}^N)_{N \in \mathbb{N}}\in A^N$ of minimizers of the following functional
\begin{equation}\label{eq:particle_error_functional}
\mathcal{\tilde{E}}_N({\widetilde{W}}):=\frac{1}{T}\int_0^T\int_{\R^d}\bigg|(\nabla\widetilde{W} -\nabla W)*\mu^N_s(\bx)\bigg|^2\dd \mu^N_s( \bx) \dd s \ , 
\end{equation}
where, as before, $ W \in X_{M,K}$ is the true interaction potential of a system like \eqref{eq:IPS} and $\mu^N(t)=\frac{1}{N}\sum_{i=1}^N\delta_{\bx_i(t)}$ is the empirical measure associated to \eqref{eq:IPS} with estimated interaction potential $\hatw^N$. In \cite[Theorem 1.1]{bongini2017inferring}, it is shown that if $( \widehat{W}^N)_{N\in\N} \in A^N$ is a sequence of minimizers of $\mathcal{\tilde{E}}_N$, this sequence has a uniformly converging subsequence to a function $\overline{W} \in X_{M,K}$. Furthermore, $ \overline{W}$ is a minimizer of the limiting functional $\mathcal{\tilde{E}}_\infty( \widetilde{W})$ in \eqref{eq:error_functional_mu} where $\mu_t$ is the solution of the mean-field PDE arising as the limit of the sequence $(\mu_t^N)_{N\in\N}$ in the $d_2$ distance. { We note that although \cite[Theorem 1.1]{bongini2017inferring} presents the $\Gamma$-convergence result, it does not show that, as one would formally expect, 
\[
\partial_t \mu^N = \nabla \cdot(\mu^N(\widehat{W}^N * \mu^N)) \overset{N\to\infty}{\longrightarrow} \partial_t \hatmu = \nabla \cdot(\hatmu(\overline{W} * \hatmu)) \ .
\]}
Using the estimate from Proposition \ref{prop:Dobrushin_}, we bridge the aforementioned gap by additionally showing that the mean-field limit of the interacting particle system 
\begin{equation}\label{eq:IPS_no_confinment}
        \dot \bx_i(t) = -\frac{1}{N}\sum_{i \neq j} \nabla\hatw^N(\bx_i(t)-\bx_j(t))\ ,
\end{equation}
is given by 
\[
\partial_t \hatmu = \nabla \cdot ( \hatmu (\nabla \overline{W} * \hatmu)) \ ,
\]
i.e. the velocity field of the limiting PDE depends on $\overline{W}$, the limit of the minimizing sequence of the functional \eqref{eq:particle_error_functional}. { This result is relevant because it confirms that the learned interaction potential will be the same regardless of whether one uses data from solutions to \eqref{eq:IPS} or \eqref{eq:aggregation_eq} when there is no diffusion.} { Furthermore, as in \cite[Theorem 1.1]{bongini2017inferring}}, under the additional \textit{coercitivity condition}, i.e. that there exists a constant $c_T>0$ such that  
\[
    { c_T \frac{1}{T}\int_0^T\norm{(\nabla\overline{W} - \nabla W)*\mu_t( x)}_{L^2(\mu_t)}^2 \dd t\leq  \mathcal{\tilde{E}}_\infty(\overline{W})}
\]
we get that $W = \overline{W}$ in $L^2(\mu)$ and thus $\hatmu = \mu$ in $(\P^2(\R^d), d_2)$, for any $t \in [0,T]$. We make these remarks precise in the following proposition. 

\begin{proposition}
    Let $ W\in X_{M,K}$ be the true interaction potential governing the particle system \eqref{eq:IPS} and $(A^N)_{N\in\N}\subset X_{M,K}$ be a family with the uniform approximating property. Consider a sequence of minimizers  $( \hatw^N)_{N \in \N}\in (A^N)_{N \in N}$ of the functional \eqref{eq:particle_error_functional} with limit $ \overline{W} \in X_{M,K}$. Let $\hatmu(0) \in \P^2(\R^d) $ with compact support be given, and $(\hatmu^N(0))_{N\in\N}$ be a
    sequence of empirical measures 
    \begin{equation*}
        \hatmu^N(0) = \frac{1}{N}\sum_{i=1}^N \delta_{\bx_{0,i}} \ , i =1,\ldots, N\ ,
    \end{equation*}
    such that $\lim_{N \to \infty}d_2(\hatmu^N(0),\hatmu(0)) = 0$. Let $(\bx_i(t))_{i=1}^N\in\R^d$ be the unique solution to the particle system   
    \begin{equation}\label{eq:estimated_part_syst}
   \dot \bx_i(t) = -\frac{1}{N}\sum_{i \neq j} \nabla\hatw^N(\bx_i(t)-\bx_j(t)), \quad  \bx_i(0) = \bx_{0,i} \quad \text{ for } i=1,\ldots,N \ . 
    \end{equation} 
     Then, the mean-field limit of the system \eqref{eq:estimated_part_syst} is given by 
     \[
         \partial_t \hatmu =  \nabla \cdot ( \hatmu (\nabla \overline{W} * \hatmu))\ .
     \]
 Furthermore, if the coercitivy condition holds and the true system has initial condition $\mu(0)= \hatmu(0) \in \P^2(\R^d)$, we have that $W = \overline{W}$ in $L^2(\mu)$.
\end{proposition}
\begin{proof}
    By Remark \ref{remark:use_of_true_or_estimated_measure}, we can apply our stability estimate from Proposition \ref{prop:Dobrushin_} with each norm depending on $\hatmu^N$, thus obtaining 
    \begin{align*}
        d^2_2(\hatmu^N_t,\hatmu_t) \leq C_1\mathcal{\tilde{E}}_{N}({\hatw}^N)  + C_2 d_2^2(\widehat{\mu}^N_0, \hatmu_0)\ ,
    \end{align*}
where 
 \[
         \partial_t \hatmu =  \nabla \cdot ( \hatmu (\nabla \overline{W} * \hatmu))\ .
 \]
Note that we have the following bound for the first term for any $t \in [0,T]$
\[
\mathcal{\tilde{E}}_{N}({\hatw}^N) = \|\nabla \overline{W}*\hatmu^N - \nabla \hatw^N*\hatmu^N \| ^2 _{L^2(\widehat{\mu}^N_t)} \leq \norm{\nabla \overline{W} - \nabla \hatw^N}^2_{\infty}\ .
\]
Then, by our assumption on the initial conditions and \cite[Theorem 1.1]{bongini2017inferring} we have that 
\[
\lim_{N\to\infty}d^2_2(\hatmu^N_t,\hatmu_t) = 0 \ ,
\]
which gives the first part of our statement. Finally, by the coercitivity condition we have that 
    \[
{ c_T \frac{1}{T}\int_0^T\norm{(\nabla\overline{W}- \nabla W)*\mu_t( x)}^2_{L^2(\mu_t)}\dd t \leq  \mathcal{\tilde{E}}_\infty( \overline{W})\ ,}
    \]
and, since $\overline{W}$ is a minimizer of $\mathcal{\tilde{E}}_\infty$, we can conclude that $W = \overline{W}$ in $L^2(\mu_t)$. 
\end{proof}

\begin{remark}
    The result is not generalized to the cases with diffusion because the method of proof of \cite[Theorem 1.1]{bongini2017inferring} requires the solution of \eqref{nonlocal} to have compact support. This is not guaranteed if diffusion is present. The extension for the cases of linear or nonlinear diffusion are left for future work. 
\end{remark}

\subsection{Linear diffusion}\label{subsec:aggregation_diff_eq}

Next, we consider an extension of the stability estimate of Proposition \ref{prop:Dobrushin_} for the aggregation-diffusion equation with linear diffusion. Namely, let $\mu \in C([0,T],\P^2(\R^d))$ be the weak solution of the following equation
\begin{align}
    \partial_t \mu &= \nabla \cdot ( \mu (\nabla W * \mu)) + \sum_{i,j=1}^d\partial^2_{x_i,x_j}[\sigma(K*\mu)^{\top}\sigma(K*\mu) \mu]\ , \label{eq:true_agg_diff}
\end{align}
 where, as before, $W$ is the interaction potential, $\sigma$ is the diffusion coefficient which is allowed to depend on the solution $\mu$ through its convolution with a kernel $K$, and the superscript $\top$ denotes the transpose of a matrix as before. We recall that \eqref{eq:true_agg_diff} can be interpreted as the evolution of the law of the solution of the following stochastic differential equation (SDE) \cite{carmona2016lectures,chapter1991sznitman}
\begin{align}\label{eq:non_linear_SDE}
    d X_t &=\nabla W * \mu_t(X_t) d t + \sqrt{2}\sigma(K*\mu_t(X_t)) d B_t \ , \nonumber
    \\
    X_0 &= X^0 \in L^2 \text{ independent of } (B_t)_{t \in[0,T]}\ ,
\end{align}
where $(B_t)_{t\in[0,T]} \in \R^d$ is a Brownian motion. 
We present now our stability estimate for this case.

\begin{proposition}
\label{prop:agg_diff_estimate} 
Let $\mu, \widehat{\mu} \in C([0,T], \P^2(\R^d))$ be weak solutions to \eqref{eq:true_agg_diff} with coefficients $\nabla W, \sigma(K)$ and $\nabla\hatw, \widehat{\sigma}(\widehat{K})$, respectively, where all the functions satisfy our Assumption \ref{assumptions:aggregation-diffusion} in Appendix \ref{subsec:stability_estimate_agg_diff_eq}. Let $X_t$ denote the solution to \eqref{eq:non_linear_SDE} and $\widehat{X}_t$ denote the solution to an analogous SDE with coefficients, {$ \hatw$, $\widehat{\sigma}$ and $\widehat{K}$}.
Furthermore, assume that the initial data $X_0$ and $\widehat{X_0}$ are chosen such that $d^2_2(\mu(0),\widehat{\mu}(0)) = \E|X_0 - \widehat{X}_0|^2$. Then we have the following stability estimate
    \[
d_2^2(\mu_t, \widehat{\mu}_t) \leq C(T) \left({ d_2^2(\mu_0,\widehat{\mu}_0)} +  \mathcal{\tilde{E}}_\infty({ \widehat{W}})  +\norm{\sigma - \widehat{\sigma}}^2_\infty + \int_0^t \norm{(K - \widehat{K})*\mu_s}^2_{L^2({\mu}_s)} \dd s\right) \ ,
\]
where $C(T)$ is a non-negative constant depending on $T$ and the Lipschitz constants of $K,\widehat{K}, \sigma, \widehat{\sigma}$ and $W$.
\end{proposition} 
\begin{proof}
    We refer the reader to Appendix \ref{subsec:stability_estimate_agg_diff_eq} for a proof of this proposition. 
\end{proof}
\begin{remark}
   In a similar way to Proposition \ref{prop:Dobrushin_}, we note that our estimate in Proposition \ref{prop:agg_diff_estimate} allows the interaction potential, the diffusion coefficient and the kernels $K$ and $\widehat{K}$ to differ between the equations being compared. As before, this would be particularly relevant in a situation where not only the interaction potential, but also the diffusion coefficient as well as the kernel $K$ have to be inferred. 
    Note that the bound in the previous proposition depends on the uniform norm of the difference between $\sigma$ and $\widehat{\sigma}$ suggesting that deeper modifications of the error functional would be required to allow for inference of these functions in an $L^2$ framework. In the numerical section we consider $K= \widehat{K}$ and $\sigma = \widehat{\sigma}$.
\end{remark}

\subsection{Nonlinear diffusion}
\label{subsec:nonlinear_aggregation_diff_eq}

In this section, we obtain a similar type of stability estimate for an aggregation-diffusion equation with nonlinear diffusion coefficient. This now corresponds to the full equation \eqref{nonlocal}. Throughout this section, we will assume $\mu(t,\bx) = \rho(t,\bx) \dd \bx$. Thus, we will consider the following Cauchy-problem for a curve of probability densities $\rho \in C([0,T],\P^2(\R^d))$
\begin{align}\label{eq:non_linear_diffusion}
    \partial_t {\rho} + \nabla\cdot(\rho v (\rho)) &= 0\ ,
    \\
    \rho(0) & = \rho^0 \in \P^2(\R^d)\ ,\nonumber
\end{align}
where $v: [0,T]\times \R^d \to \R$ is the following velocity field 
\begin{equation}\label{eq:non_linear_velocity_field}
    v(\rho) := - \nabla (H'(\rho) +  W*\rho +   V) \ . 
\end{equation}

Here, $H:[0,+\infty] \to \R$ is the internal energy density given by $H(z) = \kappa\frac{z^m}{m-1}$ where $m\neq 1$, $m\geq 1- \frac{1}{d}$ and $m > \frac{d}{d+2}$, $W \in \W^{2, \infty}(\R^d)$ is an interaction potential and $V\in \W^{2,\infty}(\R^d)$ is a confinement potential.  Following \cite[Proposition 1]{otto2001geometry}, for some $\Omega \subset \R^d$ convex with $\partial \Omega$ smooth we will consider two smooth, non-negative solutions $\rho,\hatrho:[0,T]\times \Omega \to \R^d $, to the following problem
\begin{align}\label{eq:non_linear_smooth_continuity_eq}
\frac{\partial \rho}{\partial \tau}+\nabla \cdot\left(\rho v(\rho)\right) & =0 \  \text { in }[0,T] \times \Omega 
\\
\rho v(\rho) \cdot \nu & =0 \  \text { on } [0,T] \times \partial \Omega   \nonumber 
\end{align}
where $v: \R^d \to \R$ is the velocity field in \eqref{eq:non_linear_velocity_field} and $\hatrho$ satisfies an analogous problem for the velocity field $\widehat{v}(\hatrho)= -  \nabla (H'(\hatrho) + \hatw*\hatrho +  \widehat{V})$. As before, $W$ is the ground truth interaction potential and $\hatw$ the learned interaction potential.  We are now ready to present our stability estimate for the nonlinear diffusion case
\begin{proposition}\label{prop:non_linear_stability}
Let $\rho$ be a smooth solution to \eqref{eq:non_linear_smooth_continuity_eq} with velocity field \eqref{eq:non_linear_velocity_field}, and $\hatrho$ be another solution of the analogous equation driven by the velocity field $\widehat{v}$. Then, if the conditions of Lemma \ref{lemma:internal_energy_estimate} in Appendix \ref{subsec:proof_stability_non_linear_diff} are satisfied, we have the following stability estimate for any $t \in [0,T]$
\begin{align*}
    d^2_2(\rho_t,\hatrho_t) \leq &\exp\{2(1 +L^2_{\widehat{V}} +L^2_{\hatw}) t\}
    \\
   & \qquad \times \left(d^2_2(\hatrho_0,\rho_0)+ 2T \mathcal{\tilde{E}}_\infty(\hatw)+  2\int_0^t \norm{\nabla V - \nabla \widehat{V}}^2_{L^2(\rho_{s})} \dd s\right) \ .
\end{align*}
\end{proposition}
\begin{proof}
    We refer the reader to Appendix \ref{subsec:proof_stability_non_linear_diff} for the proof of this proposition as well as the statement and proof of Lemma \ref{lemma:internal_energy_estimate}. 
\end{proof}
 \begin{remark}
Note that in Proposition \ref{prop:non_linear_stability} we assumed that the solutions $\rho, \hatrho$ are smooth. By well known properties of the porous medium equation, this implies that the solutions are bounded away from 0. However, this restriction can be removed by following the approximation arguments of the proof of Theorem 1 in \cite{otto2001geometry}. This yields weak solutions in $L^1(\Omega)$ by approximating with smooth solutions as the one considered in Proposition \ref{prop:non_linear_stability}. We note that here we consider a slightly more general energy functional than in \cite{otto2001geometry}, since our case includes an interaction term. However, the approximation argument goes through with minor modifications and we omit it it here for the sake of brevity. 
\end{remark}

{
\begin{remark}
Since the numerical section focuses exclusively on estimating the potential $W$, the result in Proposition \ref{prop:non_linear_stability} takes the following form:
\begin{align*}
    d^2_2(\rho_t,\hatrho_t) \leq &\exp\{2(1 +L^2_{\hatw}) t\}
 \left(d^2_2(\hatrho_0,\rho_0)+ 2 \mathcal{\tilde{E}}_\infty(\hatw)\right) \ .
\end{align*}
\end{remark}
}

\section{Numerical Schemes}\label{sec: algorithms} 
In practical scenarios, our access is limited to discrete-time data. Consequently, this section outlines the numerical discretization of the error functional \eqref{eq: estimator matrix form}
 and of the Basis Pursuit method introduced in Section \ref{sec 2}. We approximate all the integrals by a numerical quadrature rule and all the computations are carried out on a regular mesh. We present the fully discretized estimator for the 1D case. The generalisation to higher dimensions can be derived analogously. We consider a computational domain $[0,T]\times\Omega$, with $T>0$, $\Omega=[-R,R]$ for $R>0$ and $R$ chosen large enough such that the essential support of $\rho$ is contained in $[0,T]\times\Omega$. We will not estimate the cut-off error produced by this assumption as it is zero if the solution is compactly supported for all times with support in $\Omega$. This is the case for all PDEs \eqref{nonlocal} if the diffusion is degenerate at zero.

\subsection{Discrete error functional}\label{sec: numerical discretisation}

Let us take a space-time mesh size of $(\Delta x, \Delta t)$ and denote $t_\ell=\ell\Delta t$ where $\ell$ ranges from 0 to $\lceil{T/\Delta t}\rceil$, and $x_m=m\Delta x$ where $m$ spans from $-M$ to $M$, with $M$ defined as $\lceil 2R/\Delta x\rceil$.  For any function $v(t,x)$, denote $v_{m}^{\ell}\approx v(t_\ell,x_m)$, $v^{\ell}\approx v(t_\ell,x)$ and $v_m \approx v(t,x_m)$. The given discrete input data is $\{\rho_m^\ell:=\rho(t_\ell,x_m)\}_{m=-M,\ell=1}^{M,L\textcolor{black}{+1}}$.

 \noindent We define the standard \textcolor{black}{(forward (+))} finite difference operators $\delta_t^+$ and $\delta_x^+$ to approximate $\partial_tv$ and $\partial_x v$, respectively, 
\begin{equation}
  (\delta_t^+v)_m^{\ell}=  \frac{v_{m}^{\ell+1}-v_m^\ell}{\Delta t}\ ,\qquad   (\delta_x^+v)_m^\ell =\begin{cases} -v_m^\ell/\Delta x, \text { if } m =M\ ,\\[2mm] \frac{v_{m+1}^\ell-v_m^\ell}{\Delta x}, \text{ if } m<M\ .\end{cases}\label{finite difference op}
\end{equation}
For simplicity of the notation, we omit the parenthesis in the previous definitions and we write  $\delta_x^+ v_{m}^{\ell}$ and $\delta_t^+ v_{m}^{\ell}$. 

We will employ numerical quadratures utilizing discrete-time data to approximate the continuous integrals necessary for calculating $\mathbf{A}$ and $\mathbf{b}$, as detailed in \eqref{eq: A new} and \eqref{eq: b new}. This approach leads to a discretized version of $\mathbf{A}$ and $\mathbf{b}$, denoted by $\mathbf{A}_{n,M,L}$ and $\mathbf{b}_{n,M,L}$, respectively, satisfying the approximations:
$$\mathbf{A} \approx \mathbf{A}_{n,M,L}\ ,\qquad \mathbf{b} \approx \mathbf{b}_{n,M,L}\ .$$
Then we can write the discrete error functional, similar to \eqref{eq: estimator matrix form}, as
    \begin{equation}\label{eq: estimator matrix form2}
    \mathcal{E}_{\infty}^{n,M,L}(\mathbf{c})=\mathbf{c}^T\mathbf{A}_{n,M,L}\mathbf{c}-2\mathbf{b}_{n,M,L}^{T}\mathbf{c}\ .
    \end{equation}
In principle, choosing quadrature methods that correspond to the smoothness of the integrands is crucial for effective computation. Following the approach in \cite{lang2020learning}, we use a straightforward first-order forward Euler scheme, which makes minimal assumptions about the smoothness of the integrands. We introduce a series of functionals  that are useful in defining our numerical scheme.  For $t\in [0,T],\  x \in \Omega $  and $\nabla{\Psi_i}(x)=\psi_i(|x|)\frac{x}{|x|}$ the $i$-th basis function for the potential and its derivative, respectively, with $\psi_i(r)=\Psi_i'(r)$ for $r\in \mathbb{R}^+$, we define:
\begin{align}
     R_{n,M,L}^{i}(t,x)&\coloneqq  \sum_{m=-M}^{M}\Psi_i(x-x_m)\rho(t,x_m)\Delta x \approx \int_{\mathbb{R}} \Psi_i(x-y)\rho(t,y)dy = \Psi_i*\rho\ , \label{eq:R discrete}
     \\
  C_{n,M,L}^{i}(t,x) &\coloneqq\sum_{m=-M}^{M} \nabla \Psi_i(x-x_m)\rho(t,x_m)\Delta x \approx \int_{\mathbb{R}}\nabla \Psi_i(x-y)\rho(t,y)dy =  \nabla\Psi_i*\rho\ , \label{eq:C discrete}
  \\
  \widehat {\partial_t}\rho(t,x)&\coloneqq \sum_{\ell=1}^{L}  {\delta_t^+\rho(t^\ell,x)}\ \mathbb{1}_{[t_{\ell},t_{\ell+1})}(t)\ , \label{eq: rho time derivative discrete}
  \\
    F_{M,L}(t,x)&\coloneqq  \rho \bigg(\sum_{m=-M}^{M}{\delta_x^{+}H'(\rho(t,x_m)) }\mathbb{1}_{[x_{m},x_{m+1})}(x)+\partial_x V\bigg)\approx \rho\ \partial_x(H'(\rho)+V)\ . \label{eq: F discrete}
\end{align}
In \eqref{eq: F discrete}, the notation $\partial_x V(x)$ indicates that we compute this term analytically since the confinement potential $V(x)$ is a known function in our setting. Then  we get a discretization of $\mathbf{A}$ and $\mathbf{b}$ in \eqref{eq: A new} and \eqref{eq: b new} such that for $i,j =1,\dots,n$ we write
\begin{align}
\mathbf{A}(i,j) \approx \mathbf{A}_{n,M,L}(i,j) :&=\frac{1}{T} \sum_{m=-M,\ell=1}^{M,L} (C_{n,M,L}^{i})_m^{\ell} (C_{n,M,L}^{j})_m^{\ell}\rho_m^{\ell}\Delta x \Delta t\ ,\label{eq: discretised A}
\\
\mathbf{b}(i)\approx \mathbf{b}_{n,M,L}(i):&=  {-\frac{1}{T}\sum_{m=1,\ell=1}^{M,L} \bigg( (\widehat {\partial_t}\rho R_{n,M,L}^{i})_m^{\ell}+ (C_{n,M,L}^{i}F_{M,L})_m^{\ell} \bigg)\Delta x \Delta t}\label{eq: b with integration by parts}\ .
\end{align}
Note that in our numerical examples,  we have the solution $\rho$ (approximately) compactly supported on $[-R,R]$ and therefore  the integral kernel $G$ defined in \eqref{eq: computation of G}  is (approximately) supported on {$[-2R,2R]\times [-2R,2R]$}.
In practice, we rewrite the previous approximations to compute $\mathbf{A}(i,j)$ using the formulas in \eqref{eq: A new} as follows. We  approximate $G$ on the extension of the solution mesh: for $y,z \in \mathbb{R}, \ i,j =1,\cdots,n$, 
\begin{align}
G(y,z) \approx G_{M,L}(y,z) &= \sum_{\ell=1}^{L} \sum_{m=-M}^{M} \rho(t_{\ell}, x_m-y)\rho(t_{\ell},x_m-z)\rho(t_{\ell}, x_m)\Delta x \Delta t\ .
\end{align}
Hence, we can rewrite Equation \eqref{eq: discretised A} as 
\begin{align}\label{rewriteG}
\mathbf{A}_{n,M,L}(i,j) &= \frac{1}{T}\sum_{\ell=1}^{L} \sum_{m, m'=-2M}^{2M} \nabla \Psi_i(x_m)\nabla \Psi_j(x_{m'}) G_{M,L}(x_m,x_{m'})(\Delta x)^2\ .
\end{align}
These empirical quantities give rise to a linear system 
$\mathbf{A}_{n,M,L} \mathbf{c} =\mathbf{b}_{n,M,L}+\mathbf{e}$ where
\begin{equation}\label{ee}
\mathbf{e} = (\mathbf{A}_{n,M,L}-\mathbf{A})\mathbf{c} + (\mathbf{b}-\mathbf{b}_{n,M,L})\ . 
\end{equation}
We then solve the following  basis pursuit problem 
\begin{align}\label{pbp}
    \text{minimize }_{\hat{\mathbf{c}}\in \mathbb{R}^n} \|\hat{\mathbf{c}}\|_1\ , \nonumber\\
    \text{such that } \mathbf{A}_{n,M,L}\hat{\mathbf{c}} = \mathbf{b}_{n,M,L}\ ,
\end{align} which is a perturbed version of \eqref{orignalbp}.  Finally, we write the discrete estimator as 
$\hat{\Psi}_{n,M,L}=\sum_{i=1}^n\hat{\mathbf{c}}_i\Psi_i$.

\subsection{Error bounds}
 It is expected that  $\mathbf{A}_{n,M,L}$ and $\mathbf{b}_{n,M,L}$ will converge to $\mathbf{A}$ and $\mathbf{b}$ as $\Delta x,\ \Delta t \rightarrow 0$ and the convergence rate depends on the regularity of the solutions and the basis functions.  Therefore we first introduce some preliminary assumptions on $\rho$ and the basis functions of $\mathcal{H}$.

\begin{assumption}\label{assumpsolution}
Assume that $\rho \in { \mathcal{W}}^{2,\infty}( [0,T]\times \Omega)$ and $\mathcal{J}=H'(\rho)+V\in\mathcal{W}^{2,\infty}([0,T]\times \Omega)$.
\end{assumption}

\begin{assumption}\label{assump}
Assume $\mathcal{H}=\mathrm{span} \{\Psi_i\}_{i=1}^n$ consists of radial functions, $\nabla{\Psi_i}(\bx)=\psi_i(|\bx|)\frac{\bx}{|\bx|}$, with $\psi_i \in \mathcal{W}^{2,\infty}(\bar{\Omega})$ for $\bar{\Omega}=[-2R,2R]$. 
\end{assumption}

The  convergence analysis  is addressed in \cite{lang2020learning} and our two assumptions above are based on \cite[Assumption 3.1]{lang2020learning} and \cite[Assumption 3.2]{lang2020learning}, respectively. However, two differences are present in our approach. In \cite{lang2020learning}, $\mathcal{J}$ is the linear diffusion term whose regularity is determined by $\rho$, whereas our work extends regularity assumptions to a more general form of $\mathcal{J}$ that includes nonlinear diffusion. Furthermore, \cite{lang2020learning}  assumes that the basis functions are compactly supported, motivated by the use of a local spline basis. However, in our context, neither the basis functions nor the external potential functions need to be compactly supported. Instead, our approach involves considering their restrictions within a bounded domain. This is evident from  \eqref{rewriteG} and the compact support property of $\rho$, where the numerical error analysis only needs to be applied to functions  defined over  $[-2R,2R]$ or $[-R,R]$.

{ The regularity of the flux $\mathcal{J}$ in Assumption \ref{assumpsolution} is reasonable for solutions with no diffusion, linear diffusion or nonlinear nondegenerate diffusions. Furthermore, we note that, in the case of nonlinear degenerate diffusion, this assumption is satisfied for solutions that are bounded away from 0 in the domain $\Omega$.}

\begin{proposition}\label{prop: error A and b} 
Under the Assumptions \ref{assumpsolution} and \ref{assump}, the discretization error of $\mathbf{A}_{n,M,L}$ and $\mathbf{b}_{n,M,L}$ in \eqref{eq: discretised A} and \eqref{eq: b with integration by parts} are bounded by
    \begin{align}
        &| \mathbf{A}(i,j)-\mathbf{A}_{n,M,L}(i,j)|\leq \alpha(\Delta t+\Delta x)\label{eq: bound in A}\ ,\\
        &|\mathbf{b}(i)-\mathbf{b}_{n,M,L}(i)|\leq  \beta  (\Delta t+ \Delta x)\ ,\label{eq: bound in b}
    \end{align}
where $\alpha$ is a constant depending on $R, \|\rho\|_{1,\infty}$ and the bounds of the basis functions $\|\psi_i\|_{1,\infty}$, $i=1,\dots, n$, and 
 $\beta$ depends on $R$, $\| {\rho}\|_{2,\infty}$, $\| H'(\rho)+V\|_{2,\infty}$, and the bounds of the basis functions $\|\psi_i\|_{2,\infty}$, $i=1,\dots, n$.
\end{proposition}

\begin{proof}

The proof of \eqref{eq: bound in A} is identical to the one presented in \cite{lang2020learning}. For \eqref{eq: bound in b}, the only difference lies in estimating  $\|F-F_{M,L} \|_{\infty}$ where $F =\rho\partial_xJ$ and its quadrature is defined in \eqref{eq: F discrete}. Note that 
 \[
 \|F-F_{M,L} \|_\infty\leq \|\rho \|_\infty\Bigl|\partial_x (H'(\rho)+V)-\sum_{m=-M}^{M}(\delta_x^{+} (H'(\rho)+V))_m \mathbb{1}_{[x_{m},x_{m+1})}\Bigr|_{\infty}\leq\|\rho \|_\infty C\Delta x\,,
 \]
 with $C=\|H'(\rho)+V\|_{2,\infty}$.
 So the above estimate slightly generalizes \cite{lang2020learning} by considering a general form of diffusion that satisfies the same smoothness assumption as $\rho$. 
\end{proof}

\begin{remark} 
In Proposition \ref{prop: error A and b} we assume the solution data is exact and there is no forward error from the numerical solver. The error committed in the approximation of $\mathbf{A}_{n,M,L}$ is only due to the numerical integration. Since our quadrature rule is the middle point formula and there is no derivative involved in the expression of $\mathbf{A}$, \eqref{eq: bound in A} can be improved to spatial accuracy $(\Delta x)^2$.
Note that one can use centered finite difference to approximate the spatial derivative, and the result can be improved from $\Delta x$ to $(\Delta x)^2$ in \eqref{eq: bound in b}.  
If a higher order quadrature rule is used in time, we expect analogous improvements in the approximation with respect to time for \eqref{eq: bound in A}. However, we cannot expect the improvement on \eqref{eq: bound in b} as we need to perform numerical quadrature on 
$$\partial_t\rho \Psi_i*\rho \in W^{1,\infty}([0,T]\times \Omega)\ ,$$ 
in approximating $\mathbf{b}$ where $\mathcal{O}(\Delta t)$ is already optimal.   
\end{remark}

\paragraph{Implications for Optimal Estimation Accuracy:}

{ Consider the true support of the coefficient vector $\mathbf{c}$, denoted by $\mathcal{I} \subseteq \{1,\ldots,n\}$. If PartInv accurately identifies $\mathcal{I}$, we define our estimator as $\hat{\mathbf{c}}(\mathcal{I}) = \big(({\mathbf{A}}_{n,M,L})_{\mathcal{I}}\big)^+ {\mathbf{{b}}}_{n,M,L}$ and set entries in $\mathcal{I}^{c}$ as zero. Denoting the smallest eigenvalue of the matrix $\mathbf{A}_{\mathcal{I}}$ by $\sigma_{min}(\mathbf{A}_{\mathcal{I}})$, despite the potential ill-conditioning of $\mathbf{A}$, it is plausible to assume that 
\begin{align}\label{eq:assumption1}\|(\mathbf{A}_{\mathcal{I}})^+\|=\frac{1}{\sigma_{min}(\mathbf{A}_{\mathcal{I}})}\leq C
\end{align}
for some constant $C$, especially when $\mathcal{I}$ is a relatively small set. Assuming $\Delta x$ and $\Delta t$ are small enough, by  Weyl's inequality, it is possible to make 
\begin{align}\label{thm:singularvalue}
\sigma_{min} (({\mathbf{A}}_{n,M,L})_{\mathcal{I}}) \geq \frac{ \sigma_{min}(\mathbf{A}_{\mathcal{I}})}{2}.
\end{align}
Combining all bounds in previous section, we obtain the following error estimate showing the convergence order  of our estimator on $\Delta x$ and $\Delta t$:}
{
\begin{theorem} Suppose \eqref{eq:assumption1} and \(\Delta x\) and \(\Delta t\) are sufficiently small such that \eqref{thm:singularvalue} is also satisfied. Then, the estimation error satisfies the bound:
\[
\|\mathbf{\hat c} - \mathbf{c}\| \lesssim \frac{\sqrt{n|\mathcal{I}|}}{ \sigma_{min}^2(\mathbf{A}_{\mathcal{I}})}(\Delta x + \Delta t),
\]
where \(\lesssim\) indicates that there is a constant independent of \(\Delta x\) and \(\Delta t\).
\end{theorem}}

\begin{proof}
{ 
We bound our estimation error as
\begin{align*}
\|\mathbf{\hat c}-\mathbf{c}\| &= \|[(\mathbf{A}_{n,M,L})_{\mathcal{I}}]^+ \mathbf{b}_{n,M,L}-\mathbf{A}_{\mathcal{I}}^+ \mathbf{b}\|\\&=\|([(\mathbf{A}_{n,M,L})_{\mathcal{I}}]^+ -\mathbf{A}_{\mathcal{I}}^+)\mathbf{b}_{n,M,L}+\mathbf{A}_{\mathcal{I}}^+(\mathbf{b}_{n,M,L}-\mathbf{b})\|
\\&\ \leq 2\| \mathbf{A}_{\mathcal{I}} -(\mathbf{{A}}_{n,M,L})_{\mathcal{I}}\|((\mathbf{{A}}_{n,M,L})_{\mathcal{I}})^+\|\mathbf{A}_{\mathcal{I}}^+\| \|\mathbf{b}_{n,M,L}\|+\|\mathbf{A}_{\mathcal{I}}^+\| \|\mathbf{b}_{n,M,L}-\mathbf{b}\|
\\&=\| \mathbf{A}_{\mathcal{I}} -(\mathbf{{A}}_{n,M,L})_{\mathcal{I}}\|\frac{2\|\mathbf{b}_{n,M,L}\|}{\sigma_{min}(\mathbf{A}_{\mathcal{I}}) \sigma_{min}((\mathbf{{A}}_{n,M,L})_{\mathcal{I}})}+  \|\mathbf{A}_{\mathcal{I}}^+ \|\mathbf{b}_{n,M,L}-\mathbf{b}\|
\\&\leq \| \mathbf{A}_{\mathcal{I}} -(\mathbf{{A}}_{n,M,L})_{\mathcal{I}}\|\frac{4\|\mathbf{b}_{n,M,L}-\mathbf{b}\|}{\sigma_{min}^2(\mathbf{A}_{\mathcal{I}})} + \| \mathbf{A}_{\mathcal{I}} -(\mathbf{{A}}_{n,M,L})_{\mathcal{I}}\|\frac{4\|\mathbf{b}\|}{\sigma_{min}^2(\mathbf{A}_{\mathcal{I}})} +\|\mathbf{A}_{\mathcal{I}}^+ \|\mathbf{b}_{n,M,L}-\mathbf{b}\|
\\& \lesssim \frac{\sqrt{n|\mathcal{I}|}(\Delta x +\Delta t)}{\sigma_{min}^2(\mathbf{A}_{\mathcal{I}})},
\end{align*}  where the third line of the inequality follows from  Theorem 3.4 in \cite{stewart1977perturbation} and $\|\mathbf{b}-\mathbf{b}_{n,M,L}| \lesssim \sqrt{n}(\Delta x +\Delta t)$; the symbol ``\(\lesssim\)" indicates that there is a constant independent of \(\Delta x\) and \(\Delta t\).}
\end{proof}

\subsection{Noisy data}

To test the robustness of the proposed method, we also consider the case where the solution data is corrupted by observational noise. In particular, we analyze the effects of adding i.i.d random noise with zero mean to the discretized samples of $\rho$. Hence, in this case, the final data set used for the estimation of the interaction kernel in the numerical examples is given by
\begin{equation}\label{eq: perturbed data}
\{\tilde{\rho}(t_\ell,x_m) \}_{m=-M,\ell=1}^{M,L}\ ,
\end{equation}
where $\tilde{\rho}(t_\ell,x_m)={\rho}(t_\ell,x_m)+\epsilon^{\ell}_m$. In our numerical examples, we used  $\epsilon^\ell_m \overset{\mathrm{iid}}{\sim} \mathcal{N}(0,\sigma^2)$. To ensure that the perturbation due to the noise is on a similar scale to the solution $\rho$ we will set
$$ \sigma = \frac{p}{100} \left( \sum_{\ell=1}^{L} \sum_{m=-M}^{M} (\rho_{m}^\ell)^2\Delta x \Delta t \right)^{\frac{1}{2}}\ , $$
for some constant $p \in [0,100]$. We refer to this as the noise being $p$-percent. In what follows we denote by $\norm{\cdot}_{L^2(\varepsilon)}$ the $L^2$ norm over the probability space $(\R^d,\mathcal{B}(\R^d), \mathbb{P})$, where $\mathcal{B}(\R^d)$ is the Borel $\sigma$-algebra and $\mathbb{P}$ is a probability measure. Let us define 
\begin{align}
 \mathbf{\widetilde{A}}_{n,M,L}(i,j) :&=\frac{1}{T} \sum_{m=-M,\ell=1}^{M,L} (\tilde{C}_{n,M,L}^{i})_m^{\ell} (\tilde{C}_{n,M,L}^{j})_m^{\ell}\tilde\rho_m^{\ell}\Delta x \Delta t\ ,\label{eq: discretised A with noise}
\\
\mathbf{\widetilde{b}}_{n,M,L}(i):&=  -{\frac{1}{T}\sum_{m=-M,l=1}^{M,L} \bigg( (\widehat {  \partial_t}\tilde\rho \tilde{R}_{n,M,L}^{i})_m^{\ell}+ (\tilde{C}_{n,M,L}^{i}\tilde{F}_{M,L})_m^{\ell} \bigg)\Delta x \Delta t}\label{eq: b with integration by parts with noise}\
\end{align}
where $\tilde{C}_{n,M,L}, \tilde{R}_{n,M,L}$ and $\tilde{F}_{n,M,L}$ are defined analogously to \eqref{eq:R discrete}-\eqref{eq: F discrete}, but depending on $\tilde{\rho}$. In this framework, we can obtain the following extension of the error bounds in Proposition \ref{prop: error A and b}. 

\begin{proposition}\label{prop: error A and b noise}
    The numerical error of $\mathbf{A}_{n,M,L}$ in \eqref{eq: discretised A} when we consider the perturbed solution $\tilde \rho$ as in \eqref{eq: perturbed data} is 
    \begin{align}
        &\| \mathbf{A}-\mathbf{\widetilde{A}}_{n,M,L} \|_{L^2(\epsilon)}\leq \alpha n (\Delta t+\Delta x) + nC(\sigma \sqrt{\Delta t \Delta x}+{  \sigma^2 \Delta x}) \label{eq: bound in A noise}\ ,
    \end{align}
where $\alpha$ is as in Proposition \ref{prop: error A and b} and $C>0$ is a constant depending on $R,T$ and $\norm{\psi_i}_\infty$, $i=1,\dots,n$. 
\end{proposition}
\begin{proof} 
The error induced by random perturbations is additive, and as a consequence 
 \begin{align*}
        &\| \mathbf{A}-\mathbf{\widetilde{A}}_{n,M,L} \|_{L^2(\epsilon)}\leq \| \mathbf{A}-\mathbf{A}_{n,M,L} \|_{L^2(\epsilon)}+\| \mathbf{A}_{n,M,L} -\mathbf{\widetilde{A}}_{n,M,L} \|_{L^2(\epsilon)}.
    \end{align*}
The first part is estimated as in Proposition \ref{prop: error A and b} while the second term is discussed in Appendix \ref{app: error estimate}.    
\end{proof}

\begin{remark}\label{remark: error bound b}
     If a centered finite differences method is used and we have $H(\rho) = \frac{\rho^2}{2}$, then we can obtain the following error estimate for the numerical error of $\mathbf{b}_{n,M,L}$ with added noise,
    \begin{align}
        &\| \mathbf{b}-\mathbf{\widetilde{b}}_{n,M,L} \|_{L^2(\epsilon)}\leq \beta \sqrt{n}(\Delta t+\Delta x) + \sqrt{n}C\sigma^2 (\Delta x ^{-1}+ {  \Delta x \Delta t^{-1}})\label{eq: bound in b noise}\ ,
    \end{align}
    where we note that the inverse dependence on the mesh size is due to the discrete derivatives in $\mathbf{b}_{n,M,L}$. The estimate shows the errors introduced by the presence of noise in the discretisation of the matrix $b$ for a fixed space-time mesh size $\Delta x, \Delta t$. We remark that if an upwind scheme is used for the computation of derivatives or we have an arbitrary free energy kernel, $H$, the nonlinearities impede any explicit numerical error estimate. 
\end{remark}
\begin{proof}
    As for \eqref{eq: bound in A noise}, the error stemming from the random noise is additive so we have 
 \begin{align*}
        &\| \mathbf{b}-\mathbf{\widetilde{b}}_{n,M,L} \|_{L^2(\epsilon)}\leq \| \mathbf{b}-\mathbf{b}_{n,M,L} \|_{L^2(\epsilon)}+\| \mathbf{b}_{n,M,L} -\mathbf{\widetilde{b}}_{n,M,L} \|_{L^2(\epsilon)}\ .
    \end{align*}
Again, the first term is controlled as in Proposition \ref{prop: error A and b} and the second term is treated in Appendix \ref{app: error estimate}. 
\end{proof}

\subsection{Support pruning algorithm}\label{sec: support_pruning}

In this section, we discuss strategies for finding the right support of the coefficient vector when the data are not accurate. Once the true support is identified, we can perform restricted least squares on the support set and therefore improve the robustness of the algorithm. 

Given the discrete data, we apply Algorithm \ref{alg:euclid} on the BP problem \eqref{pbp}. When the discretization error terms in \eqref{eq: bound in A noise} and \eqref{eq: bound in b noise}, are small, it is effective to select the sparsity level of PartInv algorithm $K=s$, i.e., the exact sparsity of the true coefficient vector $\mathbf{c}$. However, in cases where this error becomes significant such as when the dimension of the dictionary (i.e. $n$) is large,  or errors coming from the discretization and noise, $(\Delta x,\Delta t, \sigma)$ increase, this choice often results in inaccurate support identification, adversely affecting the recovery of the interaction potential (see Figure \ref{fig:rec_NCW_supppruning} (a)). In such situations, it is advantageous to choose $K \geq s+1$ in our PartInv method as PartInv consistently produces a support set $\mathcal{I}^{(k)}$ that contains the true support $\mathcal{I}$ as a subset. However, when we perform restricted least squares regression---a method where the regression coefficients are estimated under certain linear constraints---on $\mathcal{I}^{(k)}$, we may still encounter large estimation errors. This is often due to the ill-conditioning of the regression matrix, which can adversely affect the accuracy of the estimates (see Figure \ref{fig:rec_NCW_supppruning} (b)). So it is necessary to prune $\mathcal{I}^{(k)}$ to identify the true support $\mathcal{I}$.

 We note that it is possible to skip the basis pursuit step, and perform restricted least squares on all possible combinations of indices from the beginning. However, the computational cost in this case is very high. The PartInv helps to reduce the number of combinatorial trials, and increases the computational efficiency of the estimation procedure. We propose the following algorithm, which combines residual error and time evolution error:  

\begin{enumerate}
    \item [Step 1:] For each subset ${J}$ from $\mathcal{I}^{(k)}$, the PartInv output, we compute the coefficient vector $\mathbf{c}_{{J}} \in \mathbb{R}^n$ using
    \[
    \mathbf{c}_{{J}}({J}) = (\mathbf{A}_{n,M,L})_{J}^{+} \mathbf{b}_{n,M,L}\ ,
    \] and $\mathbf{c}_{J}(J^c) =0$. 
    We then calculate its associated residual error (RE)
    \[
    \mathbf{c}_{J}^{\top} \mathbf{A}_{n,M,L}\mathbf{c}_{J} - 2\langle \mathbf{c}_{J},\mathbf{b}_{n,M,L}\rangle\ .
    \]

    \item [Step 2:] We sort the REs in descending order and identify a cluster of subsets whose residual errors are close to the smallest one falling within a predefined precision threshold $\tau$. This threshold $\tau$ is defined as a proportion of the norm of the error vector $|\mathbf{e}|$, previously defined in \eqref{ee}.

    \item [Step 3:] For each subset within the identified cluster, we use the interaction potential associated with $\mathbf{c}_{J}$ to incorporate it into \eqref{nonlocal}. Subsequently, we perform a forward solver on a much smaller space-time mesh size $( \widehat{\Delta x}, \widehat{\Delta t})$ than $(\Delta x,\Delta t)$
    and calculate the time evolution error (TEE) using the formula
    \[
    \mathrm{TEE}^2 = \sum_{m=-M',\ell=1}^{M',L'} |\widehat \rho_{m}^\ell-\tilde{\rho}_{m}^\ell|^2\Delta x\Delta t\ ,
    \] where we may use a subset of training data on a smaller time interval  $[0, \widehat{T}]$ for validation. 

\end{enumerate}

Step 2 draws its motivation from Proposition \ref{prop: informal_stability_estimate}, guiding the pursuit of estimators capable of accurately reproducing the training data. { But due to noise $\mathbf{e}$, we found that the smallest RE does not always
yield the best result and is highly problem-dependent. However, the trajectory evolution error (TEE) is theoretically guaranteed to work, provided the numerical solver is convergent and we choose sufficiently small $\widehat{\Delta x},\widehat{\Delta t}$. Considering calculating TEE can be computationally expensive, especially when the support candidate set is large, we look at clusters formed by RE values and then refine the true support from those with smaller RE values using trajectory evolution errors. This hybrid approach balances computational efficiency and accuracy.
}

Particularly when $\mathbf{e}$ is in a reasonable range, this strategy effectively narrows down candidate estimators for Step 3, providing computational efficiency given the potentially high computational cost of this subsequent step. In our numerical experiments, it is often easy for us to identify such a cluster of values that are close to a minimum. 

It is noteworthy that while the literature on sparse signal processing does present support pruning algorithms, our learning problem distinctively diverges due to the nonlinear relationship between the coefficient vector and the solution data. TEE, initially proposed in \cite{kang2021ident}, is employed for support pruning in the sparse identification of nonlinear PDEs, using a LASSO-based algorithm. At its core, the fundamental notion is that if the true PDE identifies the underlying dynamics, any further refinement in the discretization of the time domain should adhere to the given data. This adherence is ensured by the consistency, stability, and convergence of a numerical scheme.

To conclude, it is important to note that when $\mathbf{e}$ is large, estimators may yield approximately equivalent TEEs. In such instances, it is prudent to select the estimator yielding the sparser solution, aligning with the Akaike information criteria.

\section{Numerical examples}\label{sec: numerics}

In this section we systematically apply the algorithm outlined in Section~\ref{sec: algorithms} for the estimation of the interaction potential, 
to several instances of 
\begin{equation}
\partial_t\rho =\nabla \cdot [\rho \nabla(H'(\rho)+V(\bx)+W*\rho)]\ ,
\label{eq: main numerics}
\end{equation}
showcasing a wide range of dynamics. In particular, we consider examples with different initial data and potentials, as well as dynamics modulated by an external potential $V$, in  one and two dimensions. In the examples below, we either consider nonlinear diffusion, where $H(\rho)=\kappa\frac{\rho^m}{m-1}$, or linear diffusion, where $H(\rho)=\kappa\rho(\log \rho-1)$.

The evaluation of the algorithm's performance hinges on the computation of the relative reconstruction error defined as
\begin{equation}\label{errormetric}
E_{\textnormal{reconst}} = \frac{\|\mathbf{c}^{}-\widehat{\mathbf{c}}\|_2}{\|\mathbf{c}\|_2}.
\end{equation}

\subsection{Data generation}\label{eq: data generation}

To evaluate the  estimation approach, {the data is produced by solving \eqref{eq: main numerics}} employing a finite volume method on a grid of {high} resolution, using a space-time mesh size of $(\delta x, \delta t)$, and the solution is obtained over the time interval $[0,T]$. One could choose very fine $\delta x,$ and $\delta t$  so that we minimize the numerical error from the solver to a negligible level. More precisely, we use a semi-discrete (discrete in space only) second-order finite volume scheme as presented in \cite{carrillo2015finite}. This scheme uses a third-order strong preserving Runge-Kutta ODE solver \cite{gottlieb2001strong}. It preserves positivity of the average solution in each cell provided a CFL condition, $\delta t\leq \frac{\delta x}{2\max_m\big \{ {u^\ell}_{m+\frac{1}{2}}^{+},-{u^\ell}_{m-\frac{1}{2}}^-\big \} }$, is satisfied, where $u_{m+\frac{1}{2}}^+$ and $u_{m-\frac{1}{2}}^-$ are the right and left discrete velocity fields in each cell, respectively.
Since this finite volume scheme is obtained by integrating Equation \eqref{eq: main numerics} over each cell, it is easily generalized to higher dimensions where, in the 2D case,  the velocity field is computed over  squared cells. We consider no-flux boundary conditions in all cases.

Subsequently, this simulated data is constrained to a coarser grid characterized by a mesh of size $(\Delta x, \Delta t)$, where $\Delta x = C_x \delta x$ and $\Delta t = C_t \delta t$. Here $C_x$ and $C_t$ are referred to as the downsampling factors.  These factors represent the level of resolution present in the observational data.  

\begin{table}[h!]
\begin{center}
\begin{tabular}{l l} 
\hline
Notation &  Description\\ [0.5ex] 
 \hline\specialrule{0.1em}{0.1em}{0em}
 $(\delta x, \delta t)$ & Space-time step size used in finite volume solver \\ 
 \hline
 $\Delta x = C_x\delta x $ & Space size in observational data   \\ 
 \hline
 $\Delta t = C_t\delta t$ & Time step size in observational data \\ 
 \hline
$(\widehat \Delta x,\widehat \Delta t)$ & Space-time step size used in finite volume solver in support pruning step \\
  \hline
\end{tabular}
\caption{Notations of space-time step size.}
\end{center}
\end{table}

\paragraph{Overview of numerical experiments.} In the following section, we  test the effectiveness of algorithms over 1D and 2D numerical examples that display various collective behaviors. 

\begin{itemize}

\item For each example, we assess the effectiveness of PartInv across different data scenarios by using the error metric defined in \eqref{errormetric}. Firstly, we examine the case of noise-free data, we first generate our data by using an approximation of the PDE obtained in a very fine mesh. Then, the major source of errors arises from the evaluation error of the functional \eqref{eq: estimator matrix form2} introduced in the observational data by the downsampling procedure above. Secondly, we explore scenarios with noise contamination, keeping the space-time resolution constant. It is important to note that introducing Gaussian noise might lead to negative values in the solution data. This scenario is at odds with the reality that the actual solution data should be positive. However, in this study, we intentionally avoid using any denoising techniques. Our aim is to evaluate the resilience of PartInv even when the solution data deviates from physical constraints. Finally, we also test the robustness of the method when the solution data is obtained at coarse scale in a 2D example (See Example 5). 

\item We test the effects of different choices of the sparsity parameter $K$ in the PartInv algorithm on the reconstruction accuracy and show how the support pruning algorithm can help stabilize the results. We thereby provide a comprehensive check of robustness for PartInv. 

\item We show regularization is necessary in our estimation problem and sparsity-promoting is effective. Indeed, the least squares estimator yields inaccurate estimators while promoting sparsity can yield very accurate estimations. See  Figure \ref{fig:NCWleastsquare}. 

\item  We perform comparative tests between PartInv and standard solvers in the field of PDE sparse identification: LASSO-type estimators, Greedy methods such as Subspace pursuit\footnote{the algorithm only differs from CoSaMP in choosing sparsity.}, and Sequential Thresholded Least Squares (SINDy), all within the framework of basis pursuit, see Figure \ref{fig:rec_1dmeta_perturbation} and \ref{fig:rec_KF_perturbation} in Example 2 and 3. Additional examination is performed to contrast the proposed data-fidelity term with that invoked by the strong form of PDEs, commonly utilized in PDE literature; for a relevant example we refer to Figure \ref{fig:rec_KF_perturbation}.

\item {We note that although Example 1 and Example 4 do not meet the regularity conditions that guarantee the error estimates in Section \ref{sec: algorithms}, we observe overall good performance of our methods in the recovery of the interaction potential. }

\end{itemize}

\subsection{One dimensional examples}

Consider the one-dimensional aggregation-diffusion equation given by
\begin{align*}
    \partial_t\rho=(\rho(\kappa \rho^{{ m}-1}+W*\rho+V)_x)_x\ ,
\end{align*} 
where $W(x) =\Phi(|x|)$ and $\Phi'(|x|)=\phi(|x|)\mathrm{sign}(x)$. 

\paragraph{Example 1 (Nonlinear diffusion and compactly supported attraction potential)} We consider the nonlinear diffusion case where $m=2$, $\kappa=0.2$ and $V=0$. The initial condition is $\rho_0(x)=\chi_{[-2,2]}(x)$ and we have a compactly supported interaction potential given by
\begin{equation*}
   { W(x)=-5(1-|x|)_+}\ .
\end{equation*}
{ The solution data is produced with the parameters in Table \ref{t:CP_params}.}
\begin{table}[H]
\centering
\begin{tabular}{| c | c | c | c | c | c |c|}
\hline 
 $\delta t$ & $\delta x$ & Time domain  & Spatial domain & Initial condition & $\phi (|x|)$  \\ 
\hline 
 $0.5*10^{-4}$ & $10^{-2}$ &$[0,0.5]$ & $[-6,6]$ &  $\chi_{[-2,2]}(x)$ & $5\chi_{[0,1](|x|)}$\\
\hline
\end{tabular}
\caption{\textmd{{(CP) { Parameters to produce the solution data using a finite volume scheme.}} }}
\label{t:CP_params}
\end{table}

These dynamics have the capability to simulate formation of clustered solutions which, after some time, merge together as a result of the attraction potential and the very weak diffusion, see the profile of trajectory data used in our training in Figure \ref{fig:NCWtrajectorydata} (a).{ Note that the solution profile obtained is a transient state and we expect these two bumps to merge together at longer times, given the attraction range of the potential. 
Considering further away initial conditions, or weaker interaction potentials, leads to a steady state of disconnected support \cite{carrillo2015finite}.} Applications of this particular dynamics can be found for instance in interacting populations of cells. Cells from different colonies can start moving towards each other if they are at a certain sensing distance, forming bigger aggregates as a survival mechanism. In time, this will be observed in Figure \ref{fig:NCWtrajectorydata} (a). 

We consider the estimation of the interaction kernel  $\phi$ on the positive axis, and the results on the negative axis will follow automatically by employing the radial symmetry. We use a local piecewise linear ($p=0,1$) or constant basis $(p=0)$ of the form  $\{x^{p}\cdot\chi_{[\frac{6j}{n},\frac{6(j+1)}{n}]}(x):j=0,\cdots,n-1\}$.  We choose $n=12$. In the context of the piecewise linear basis (dimension $=24$), the true interaction kernel is 2-sparse with respect to this particular basis representation. Similarly, when using the piecewise constant basis (dimension $=12$), the true interaction kernel also exhibits a 2-sparse characteristic in relation to its basis representation.

In Figure \ref{fig:NCWleastsquare}, we show  the {efficacy} of sparsity-promoting in the proposed algorithm by comparing the least squares estimator using the psedoinverse depicted in (a) with our estimator in (b) {using piecewise constant basis for the noise-free data}. 
 We see the least squares estimator failed in this case, while our estimator obtained from the sparsity-promoting algorithm produced an accurate estimate since it identified a correct 2-sparse representation.
\begin{figure}[H]\label{NCWtraj}
\centering
(a)\begin{minipage}{.45\textwidth} \centering   
\includegraphics[width=0.72\textwidth,height=0.48\textwidth]{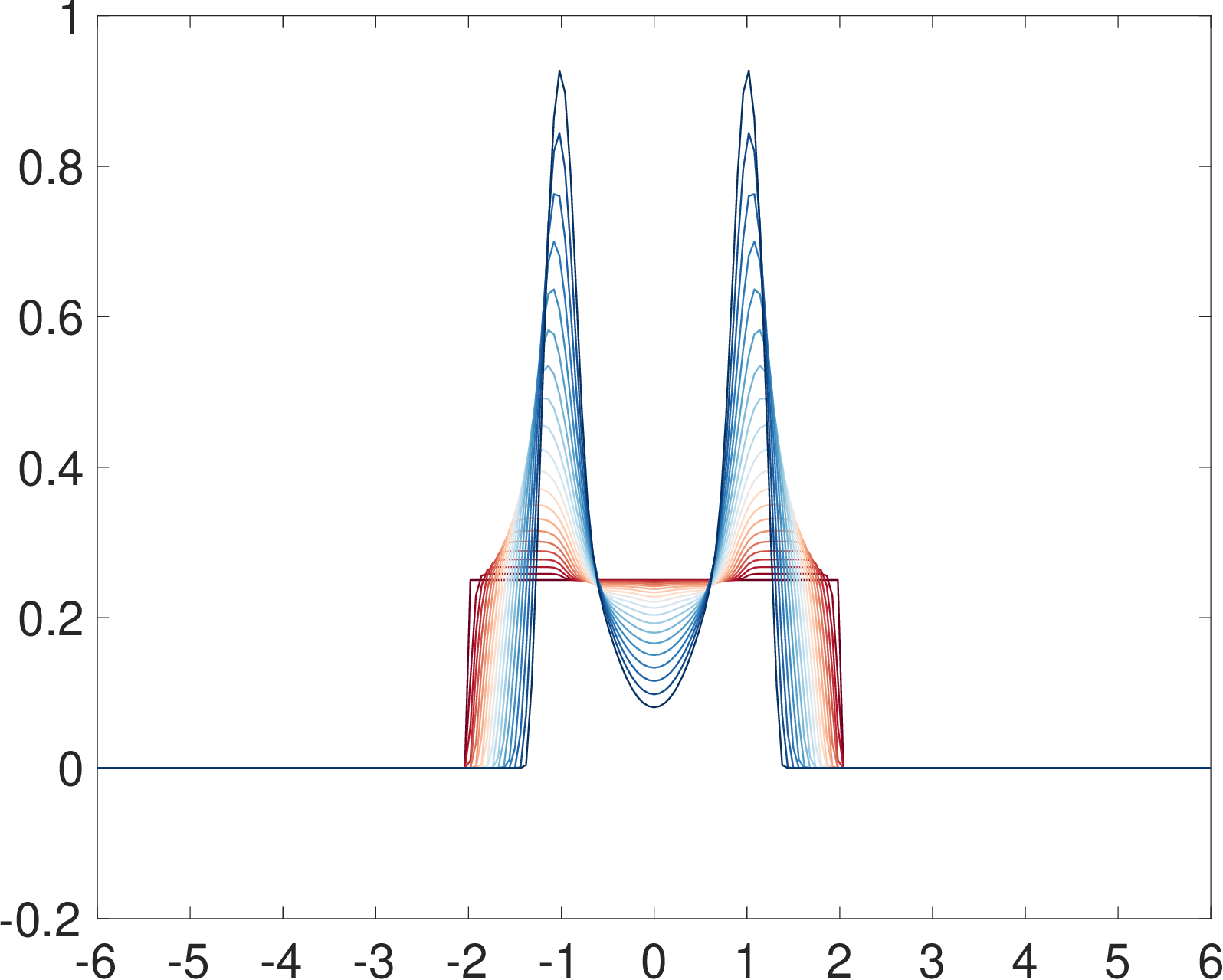}
\end{minipage}
(b)\begin{minipage}{.45\textwidth} \centering  
\includegraphics[width=0.72\textwidth,height=0.48\textwidth]{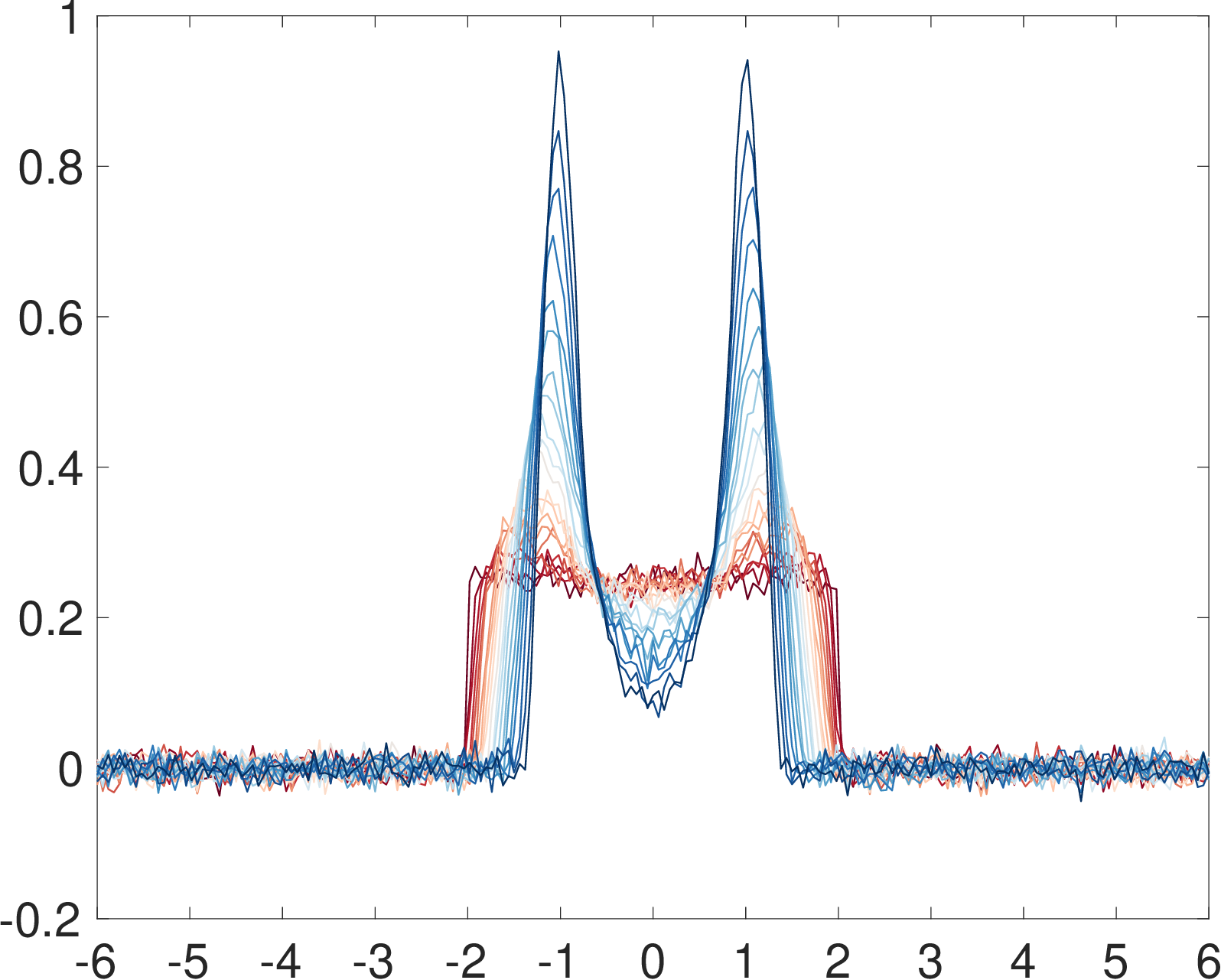}
\end{minipage}

\caption{Profile of a subset of trajectory data used in our training  where we choose $\Delta x = 6 \delta x$ and $\Delta t =50 \delta t$. A spectrum of colors transitioning from red to blue to symbolize the passage of time. (a) data generated from the numerical solver (b) data with 3\% noise added. }\label{fig:NCWtrajectorydata}
\end{figure}

In addition, we test the robustness of PartInv with respect to data perturbations coming from observation noise and discretization errors, and summarize the results in Figure \ref{fig:rec_NCW_perturbation}. 
 In this set of experiments, we observed in all challenging data regimes (large discretization error and/or large noise) PartInv {accurately} identified the correct support and the reconstruction error got amplified because of the corruption of the data. In addition, for a fixed space-time mesh size, we observe in Figure \ref{fig:rec_NCW_perturbation}(a) that the reconstruction error depends linearly on the noise variance. Furthermore, we display the relative reconstruction errors with different choices of space-time mesh size $(\Delta x, \Delta t)$ in Figure \ref{fig:rec_NCW_perturbation}(b). We see the error depends roughly linearly with respect to $\Delta x$, but did not vary much with respect to \textcolor{black}{$(\Delta t)$} in our selected range. This is possible as in our error analysis, the coefficient in front of  $(\Delta t)^2$ may be relatively small, and at the current scale it is dominated by the errors in the $\Delta x$ term.

\begin{figure}[H]
\centering

(a) \begin{minipage}{.45\textwidth} \centering  Least Squares  
\vspace{0.2cm}

\includegraphics[width=0.72\textwidth,height=0.48\textwidth]{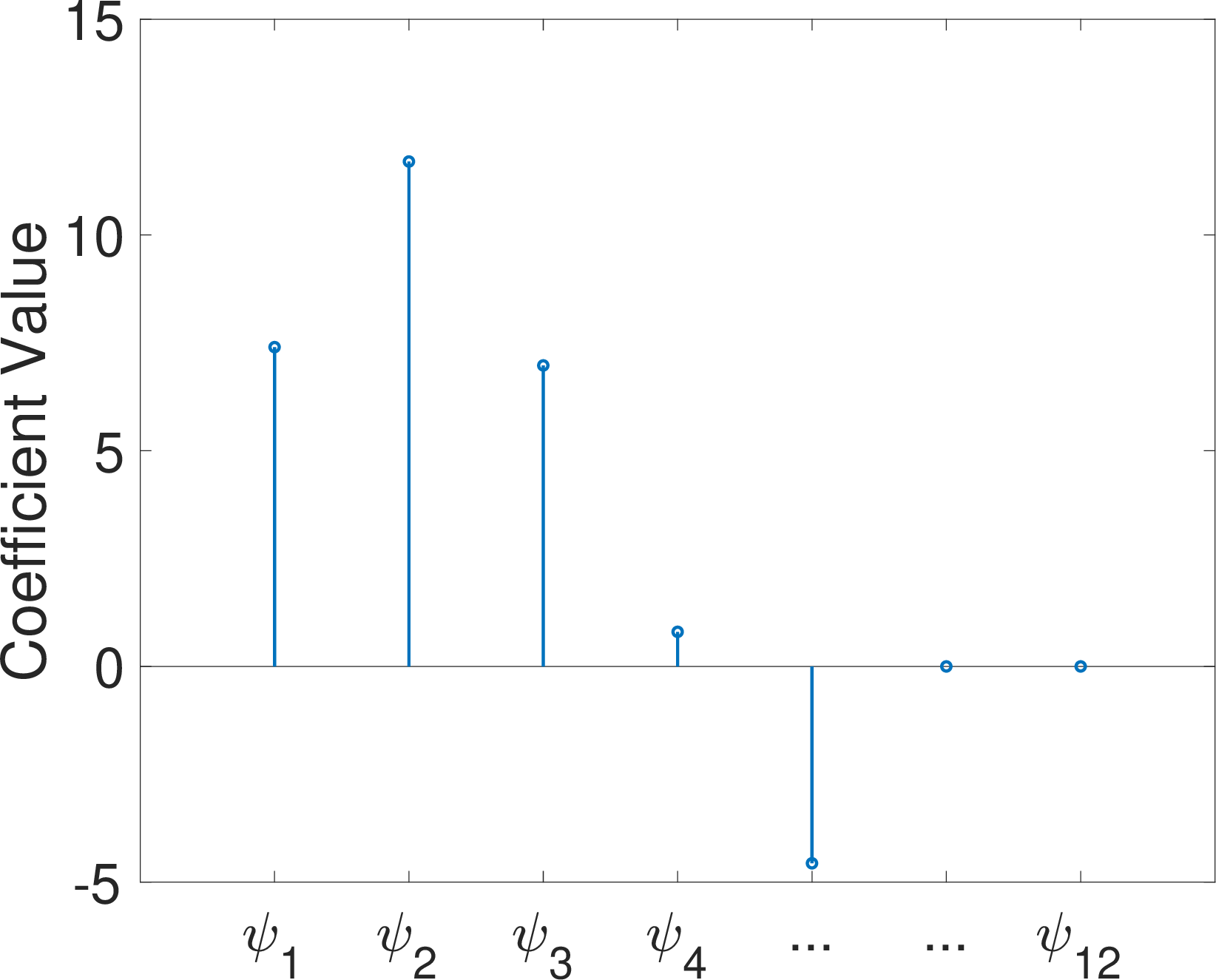}
\end{minipage}
(b)\begin{minipage}{.45\textwidth} \centering   PartInv with $K=2$
\vspace{0.3cm}

 %(This figure is new)
\includegraphics[width=0.72\textwidth,height=0.48\textwidth]{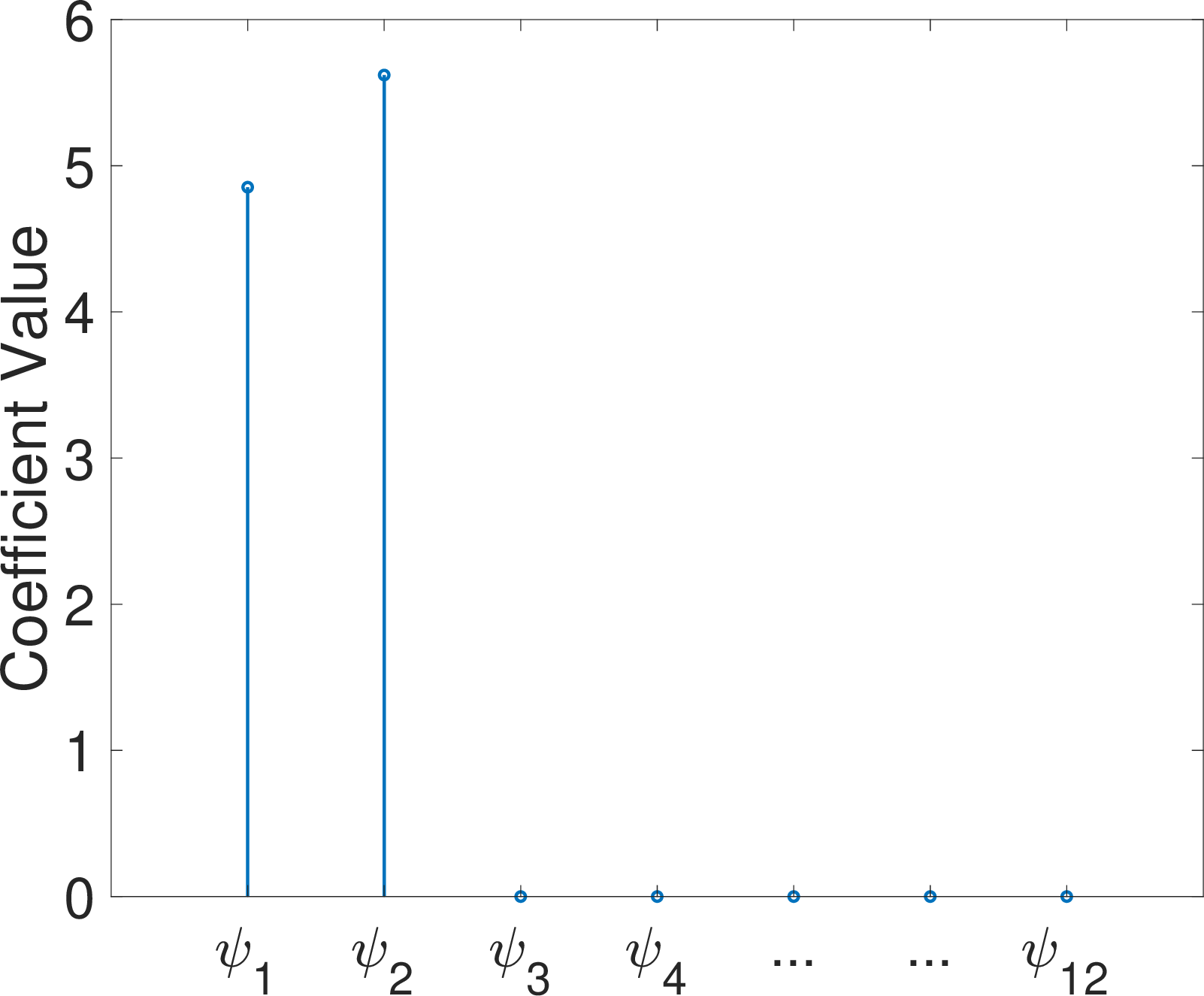}

\end{minipage}

\caption{Results with piecewise constant basis where we choose $\Delta x = 6 \delta x$ and $\Delta t =50 \delta t$.  From (b), { we clearly see that sparsity effectively regularizes the inverse problem and finds a solution that closely aligns with the ground truth coefficient vector} $[5,5]$ with respect to the basis $[\psi_1,\psi_2]$. 
  }\label{fig:NCWleastsquare}
\end{figure}

\begin{figure}[H]
\centering
(a)\begin{minipage}{.45\textwidth}\centering  
\includegraphics[width=0.72\textwidth,height=0.48\textwidth]{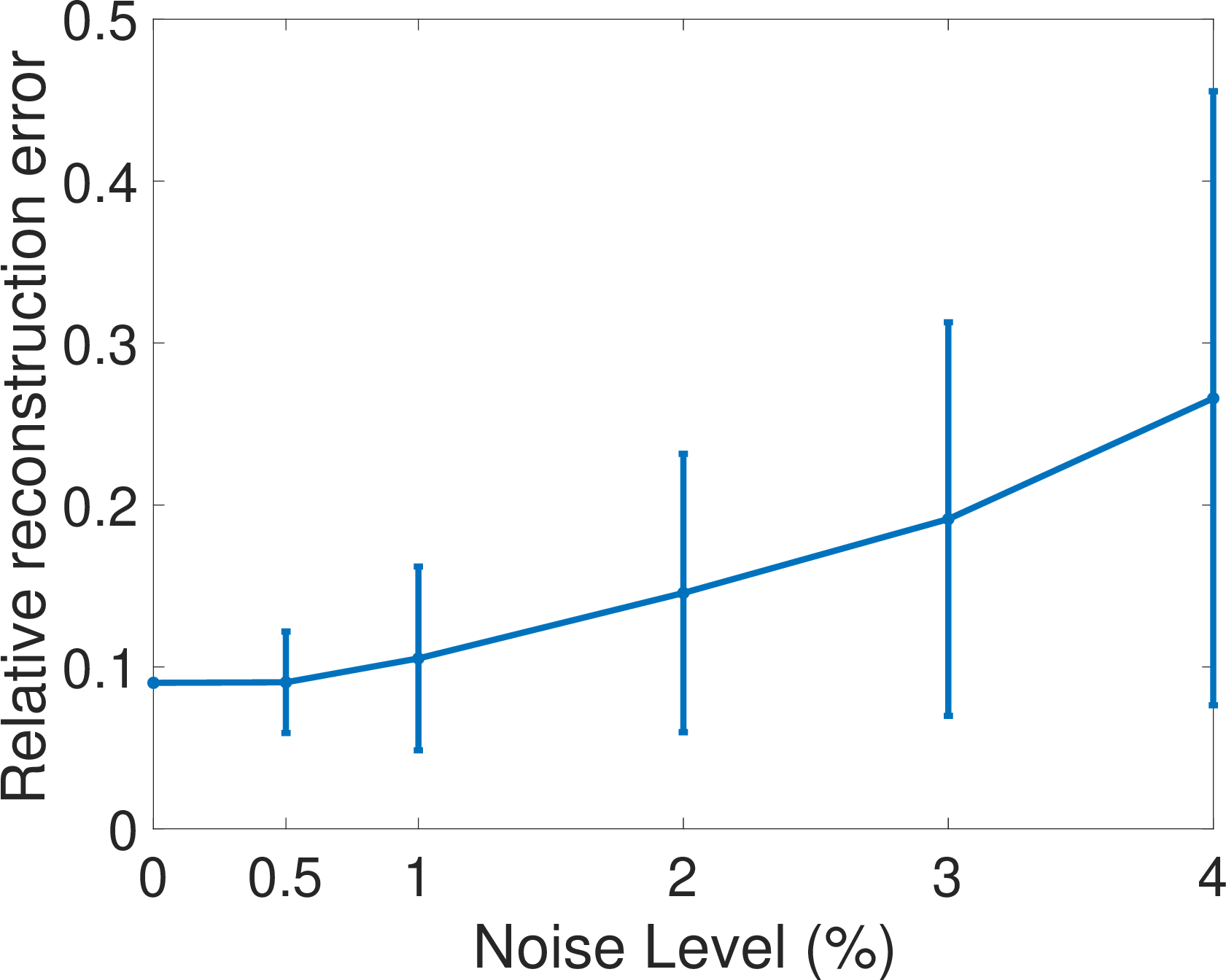}

\end{minipage}
(b)\begin{minipage}{.45\textwidth}  \centering 
\includegraphics[width=0.84\textwidth,height=0.6\textwidth]{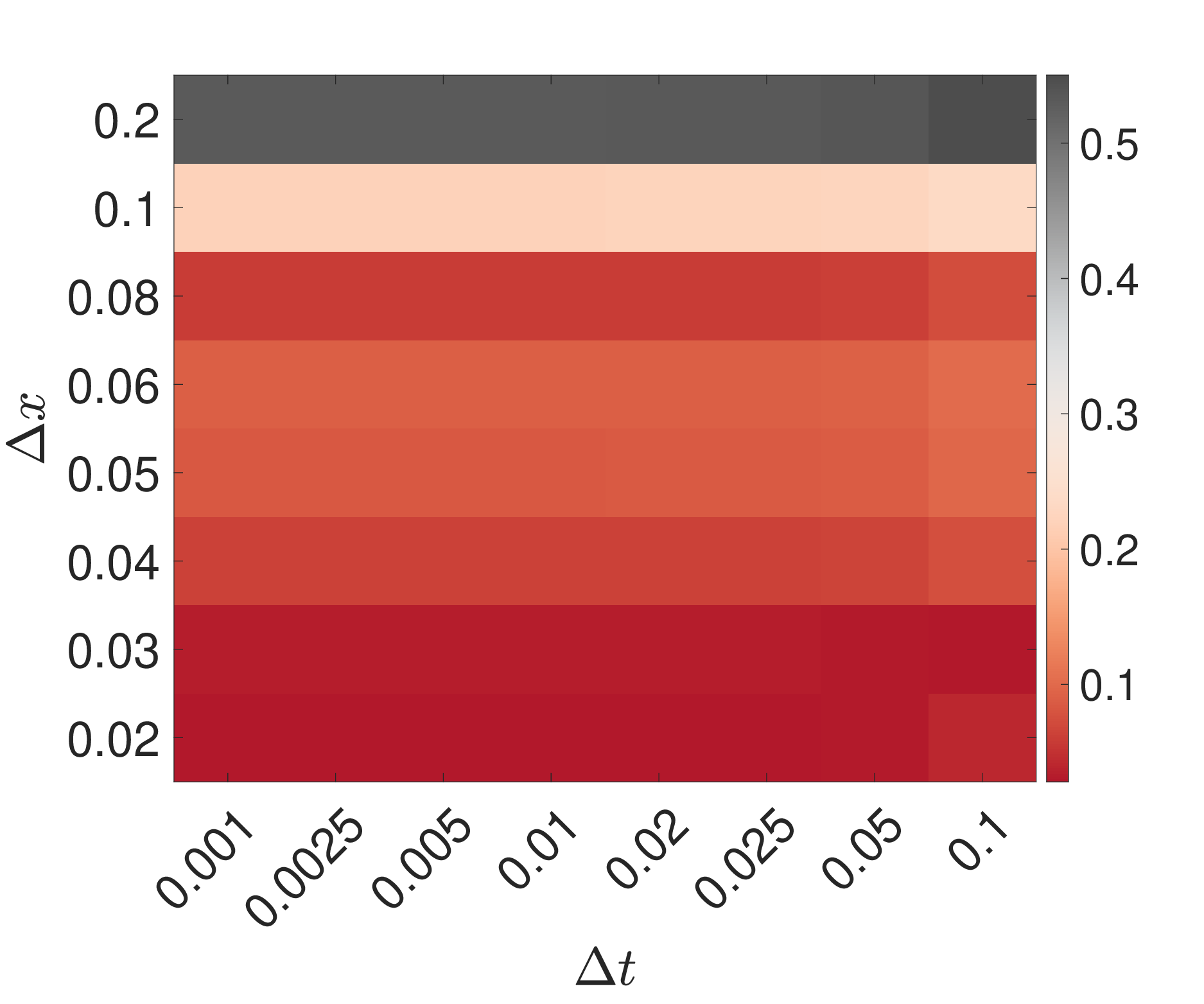}
\end{minipage}
\caption{Results of PartInv with {sparsity} $K=2$ using piecewise constant basis. (a) Accuracy for different levels of noise where we display the mean and standard deviation of relative errors over 100 trials. { (b) Relation between the relative reconstruction error, given by the different color intensities, and the mesh size  $(\Delta x,{\Delta t})$.}}\label{fig:rec_NCW_perturbation}
\end{figure}
The effectiveness of PartInv depends, in part, on the choice of the dictionary. To illustrate this, we investigate the performance of the algorithm over a piecewise linear basis. As evidenced in Figure \ref{fig:rec_NCW_supppruning} (a), using the piecewise linear basis leads to inaccurate recovery using the same noise-free training data and parameters as in Figure \ref{fig:NCWleastsquare}. The reason is that using a larger dictionary increases the probability of obtaining a matrix {$\mathbf{A}$ whose} columns present high coherence with those corresponding to the true support, making the identification of the true support particularly difficult in such instances. 

To circumvent this challenge, it is advantageous to choose a larger $K$ and implement our support pruning algorithm, introduced in Section \ref{sec: support_pruning}. We see from Figure \ref{fig:rec_NCW_supppruning} (b) that PartInv outputs a support set $\{1,2,3\}$. Then we apply the support pruning algorithm, as depicted in Figure \ref{fig:rec_NCW_supppruning} (c), where we identify the right support set $\{1,3\}$ \footnote{ We choose the natural ordering  in our piecewise linear basis $\{{x^0 \bm{1}_{[0, 1/2]}, x^1  \bm{1}_{[0, 1/2]}, x^0 \bm{1} _{[1/2, 1]}, x^1 \bm{1}_{[1/2, 1]}, ... }\}$ so the true interaction kernel is spanned by $\psi_1$ and $\psi_3$.} using the numerical values in Table \ref{tab:ncw_supppruning}, yielding accurate coefficient estimation. { Note that by our empirical evaluations, the ones with smallest REs do not necessarily yield accurate estimations, so we recommend using both RE and TEE. }
\begin{table}[H]
    \centering
        \resizebox{0.9\textwidth}{!}{%
 \begin{tabular}{|c|c|c|c|c|c|c|c|}
        \hline
        Active terms & Coefs & RE & TEE & Active Terms & Coefs & RE & TEE\\
        %\hline
        $\psi_1$ & 11.56 & -0.24 &  & $[\psi_1,\psi_2]$ & [-52.97,208.05] & 0.59& \\
        %\hline
        $\psi_2$  & 37.88 & -0.29 &  & $[\psi_1,\psi_3]$ & \textbf{[4.69,5.63]} & \textbf{-0.40}& \textbf{0.04}\\
        %\hline
       $\psi_3$  & 8.40 & -0.31 &  & $[\psi_2,\psi_3]$ & [16.49,5.28] & -0.39& 0.16\\
        %\hline
       $[\psi_1,\psi_2,\psi_3]$ & [9.55, -17.18, 6.01] & -0.38 & 0.40 &  & & & \\
        \hline
    \end{tabular}}\caption{Numerical results for the pruning algorithms where {we refine the finite volume solution using a mesh size $\widehat \Delta x = \frac{\delta x}{2}$ and $\widehat \Delta t =\frac{\delta t}{4}$.} }
    \label{tab:ncw_supppruning}
\end{table}

\begin{figure}[H]
\centering

\begin{minipage}{.3\linewidth} \centering \small \hspace{2mm} (a) PartInv with $K=2$ \end{minipage}
\begin{minipage}{.3\linewidth} \centering \small \hspace{2mm} (b) PartInv with $K=3$ \end{minipage}
\begin{minipage}{.3\linewidth} \centering \small \hspace{2mm} (c) Results with support pruning 
\vspace{0.3cm}
 \end{minipage}\\
\includegraphics[width=.3\linewidth,height=0.25\linewidth]{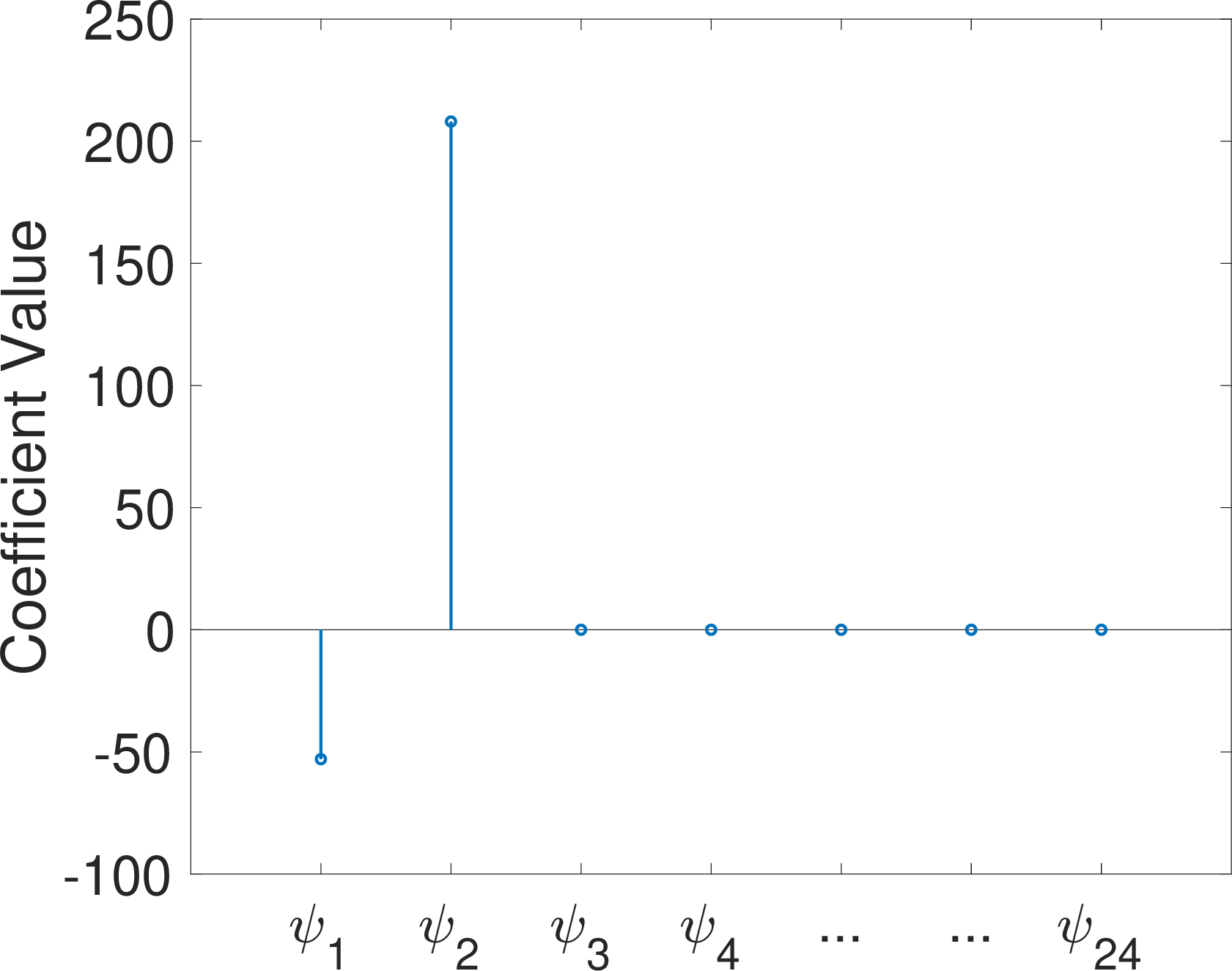}
\includegraphics[width=.3\linewidth,height=0.25\linewidth]{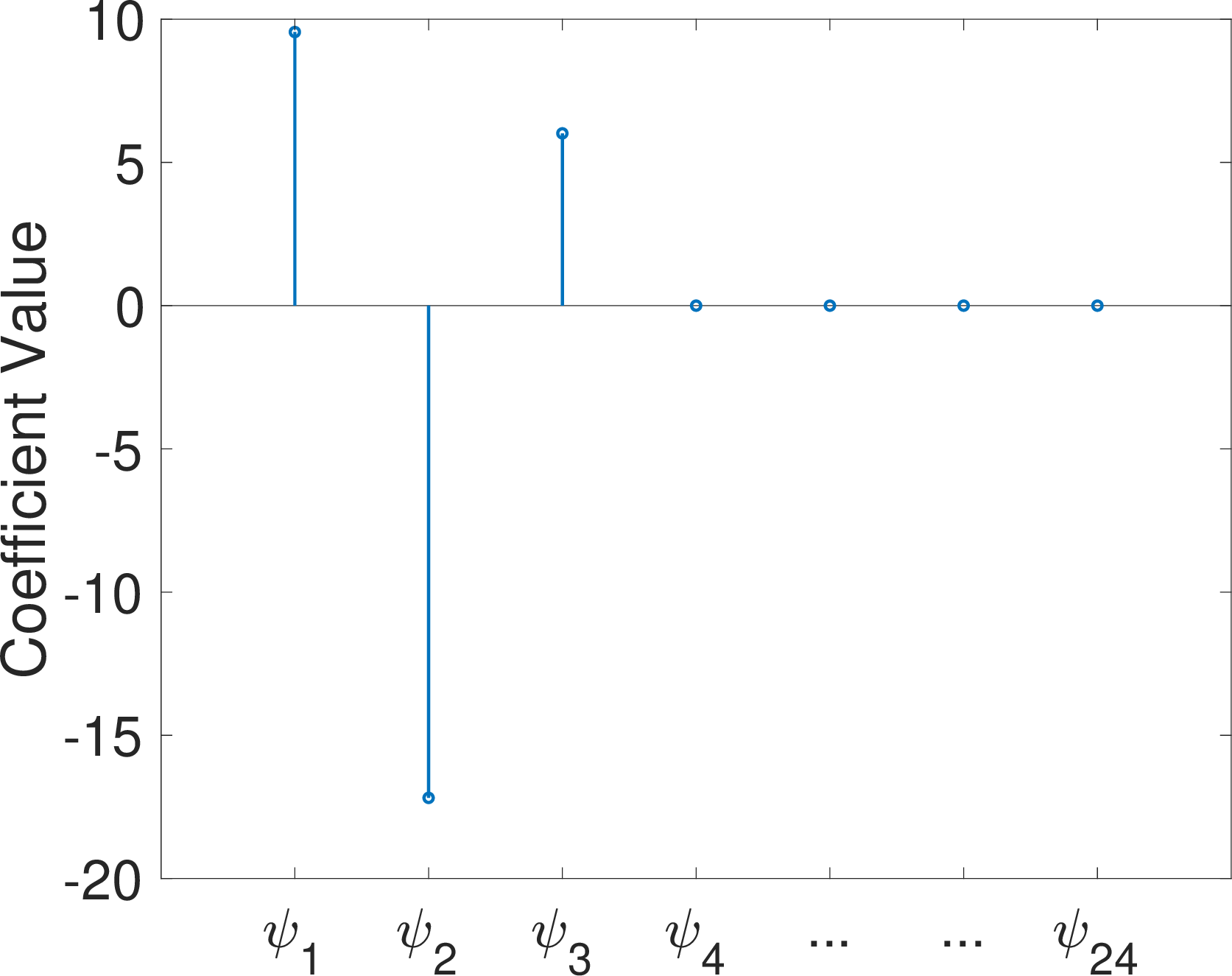}
\includegraphics[width=.3\linewidth,height=0.25\linewidth]{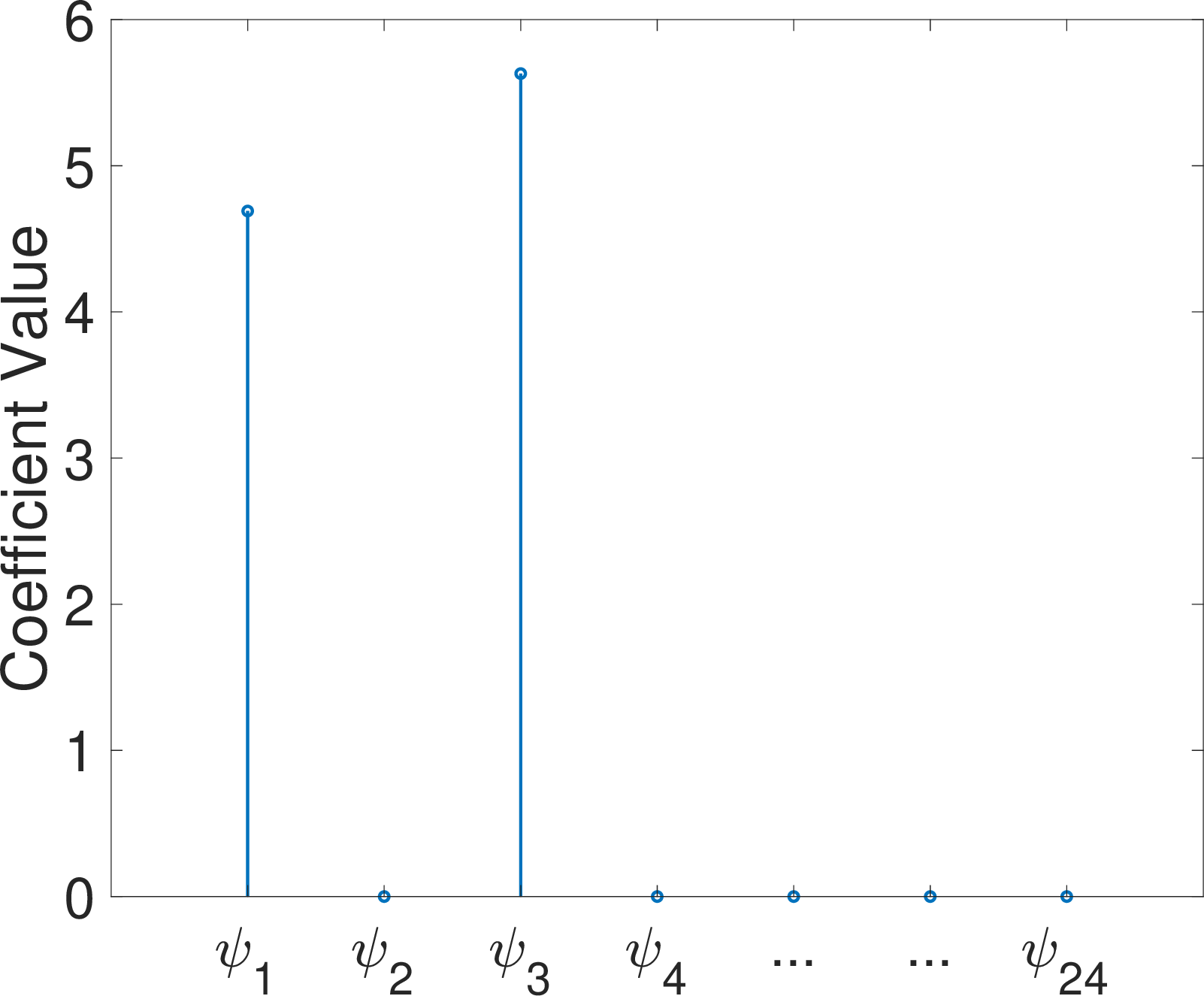}\\
\caption{Results for PartInv  with piecewise linear basis where we use the same training data as in Figure \ref{fig:NCWleastsquare}. (a)-(b) are the cases without support pruning. The case with support pruning with $K=3$ is presented in (c). We see it produced the most accurate estimation of the true coefficient $[5,5]$ with respect to the basis $(\psi_1,\psi_3)$.}\label{fig:rec_NCW_supppruning}
\end{figure}

\paragraph{Example 2 (Nonlinear diffusion and nonlocal attraction potential) }

 We consider the nonlinear diffusion case where this time $m=3$, $\kappa=0.48$ and $V=0$. We have a nonlocal attraction interaction potential given by
\begin{equation*}
    W(x)=-2\frac{\exp(-|x|^2)}{\sqrt{\pi}}-2\frac{\exp(-|x|^2/2)}{\sqrt{2\pi}} .
\end{equation*} 

This equation describes spontaneous biological aggregation of e.g. bacteria colonies \cite{topaz2006nonlocal}. An extensive study of the steady states for an analogous example was carried out in \cite{burger2014stationary} where it was observed that, when $m>2$ the attraction dominates the dynamics leading to compactly supported steady states as observed in Figure \ref{fig:1dmeta}. The dynamics in this case is governed by a competition between the nonlocal attraction, characterized by the term $W*\rho$, and the nonlinear diffusion with exponent $m$, which represents a local repulsion.
To generate the solution data we used as initial condition 
 $\frac{\mathcal{N}(1,0.5^2)+\mathcal{N}(-1,0.5^2)}{2}$. The solution data profile is plotted in Figure \ref{fig:1dmeta} (a) and its noisy version in (b).

\begin{table}[H]
\centering
\begin{tabular}{| c | c | c | c | c | c |}
\hline 
 $\delta t$ & $\delta x$ & Time domain  & Spatial domain  & $\phi (|x|)$  \\ 
\hline 
 $10^{-4}$ & $1.25\cdot 10^{-2}$ &$[0,1.5]$ & $[-6,6]$ &  $\frac{4}{\sqrt{\pi}}|x|\exp(-|x|^2)+\frac{2|x|}{\sqrt{2\pi}}\exp(-\frac{|x|^2}{2})$ \\
\hline
\end{tabular}
\caption{(1D Metastable) { Parameters to produce the solution data using a finite volume scheme.}}

\end{table}

\begin{figure}[H]
\centering (a)\begin{minipage}{.45\textwidth}   
\centering\includegraphics[width=0.72\textwidth,height=0.48\textwidth]{ 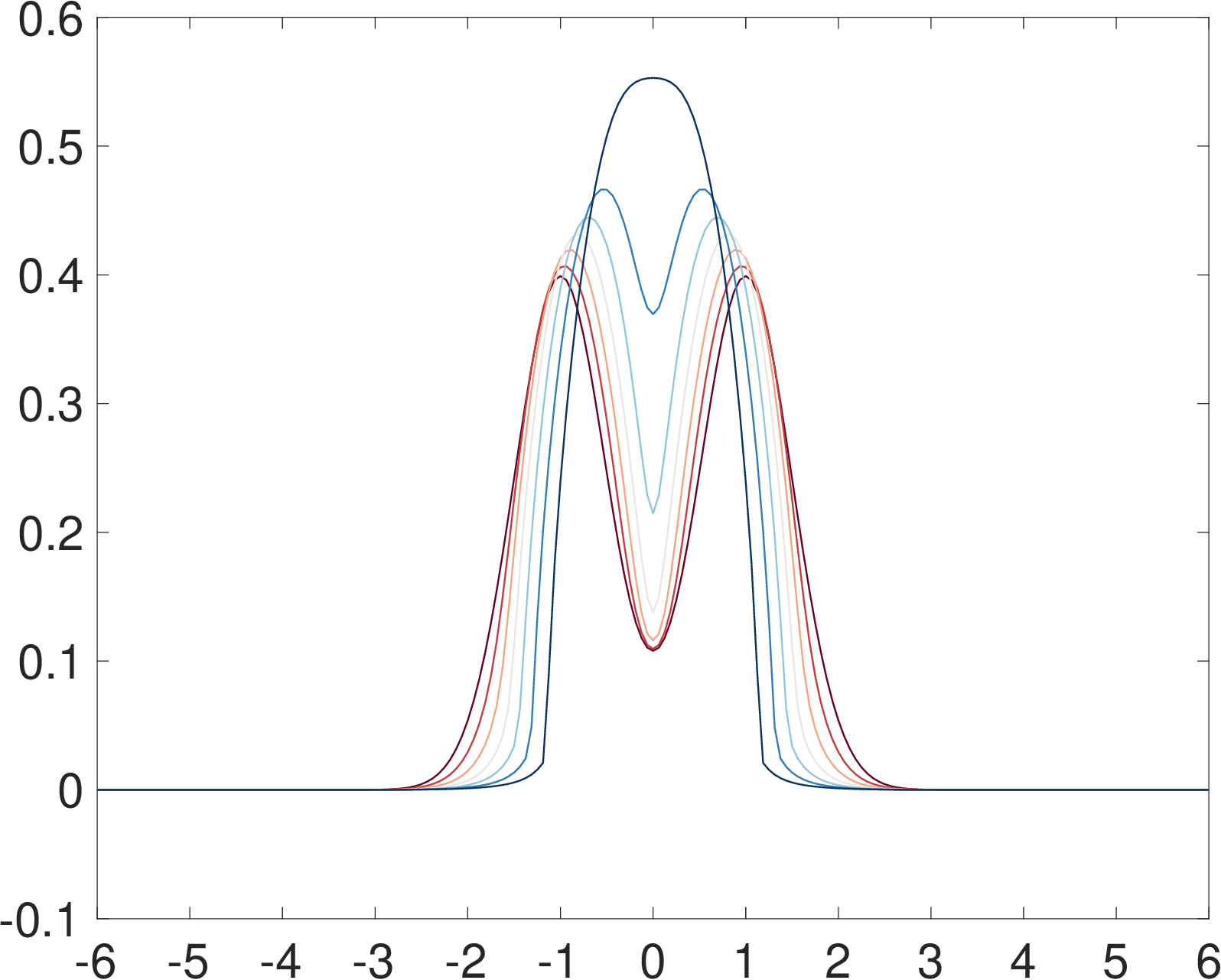}
\end{minipage}
(b)\begin{minipage}{.45\textwidth}
\centering\includegraphics[width=0.72\textwidth,height=0.48\textwidth]{ 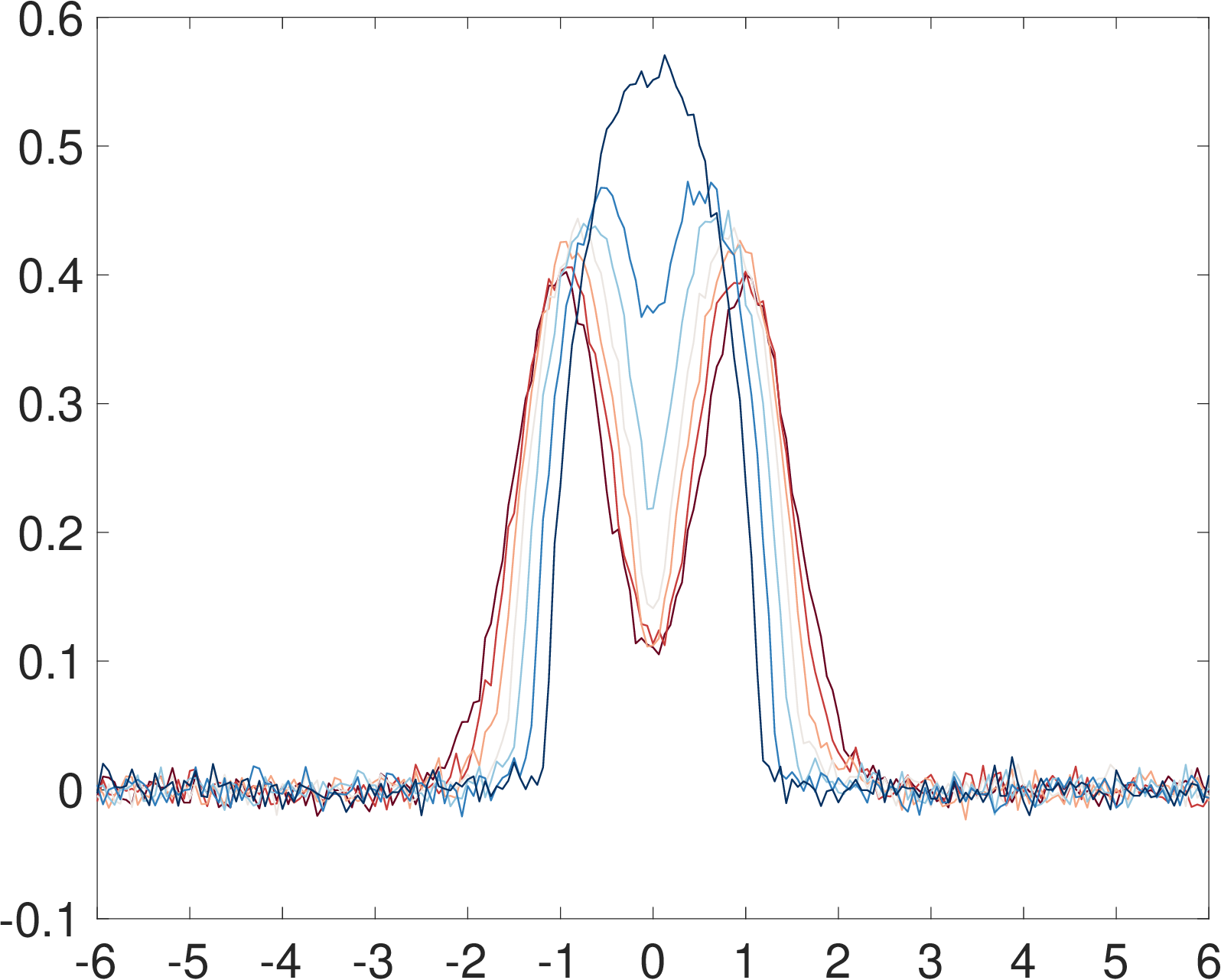}
\end{minipage}
\caption{Profile of the solution for $\Delta x = 5 \delta x, \Delta t = 2500 \delta t$. (a) a subset of solution data generated from the numerical solver (b) the solution data with $1\%$ noise added. }\label{fig:1dmeta}
\end{figure}

To estimate the interaction kernel, we use a  set of  exponential basis of the form  $\{\frac{|x|}{6}\exp(-w|x|^2): w=0.5:0.5:5 \}$ (see Section \ref{sec Notation} for notation). Then, the true interaction kernel is $2$-sparse with respect to this particular basis representation. 
Figure \ref{fig:1d_incoherence} (a) shows that it yields a very coherent basis in our sparse learning problem, as the  coherence parameter ranges from 0.982 to 1 (see discussion in Section \ref{subsec: bp_intro}). In the algorithm, we set the $K=2$ and PartInv can yield a very accurate estimation as observed in Figure \ref{fig:rec_1dmeta_perturbation} (a), for the solution data in Figure \ref{fig:1dmeta} (a) and (b), where the time observations are very sparse. We also explore its robustness with respect to different $\Delta x$ and different noise levels for the choice $\Delta x =5\delta x, \Delta t = 2500 \delta t$ and summarize the result in Figure \ref{fig:rec_1dmeta_perturbation} (b)-(d),  {where we also compare with the subspace pursuit and LASSO approach. We see that the reconstruction error is significantly smaller using our approach.}

\begin{figure}[!ht]
\centering (a)
\begin{minipage}{.45\textwidth} \centering

\includegraphics[width= 0.9\textwidth,height=0.6\textwidth]{ 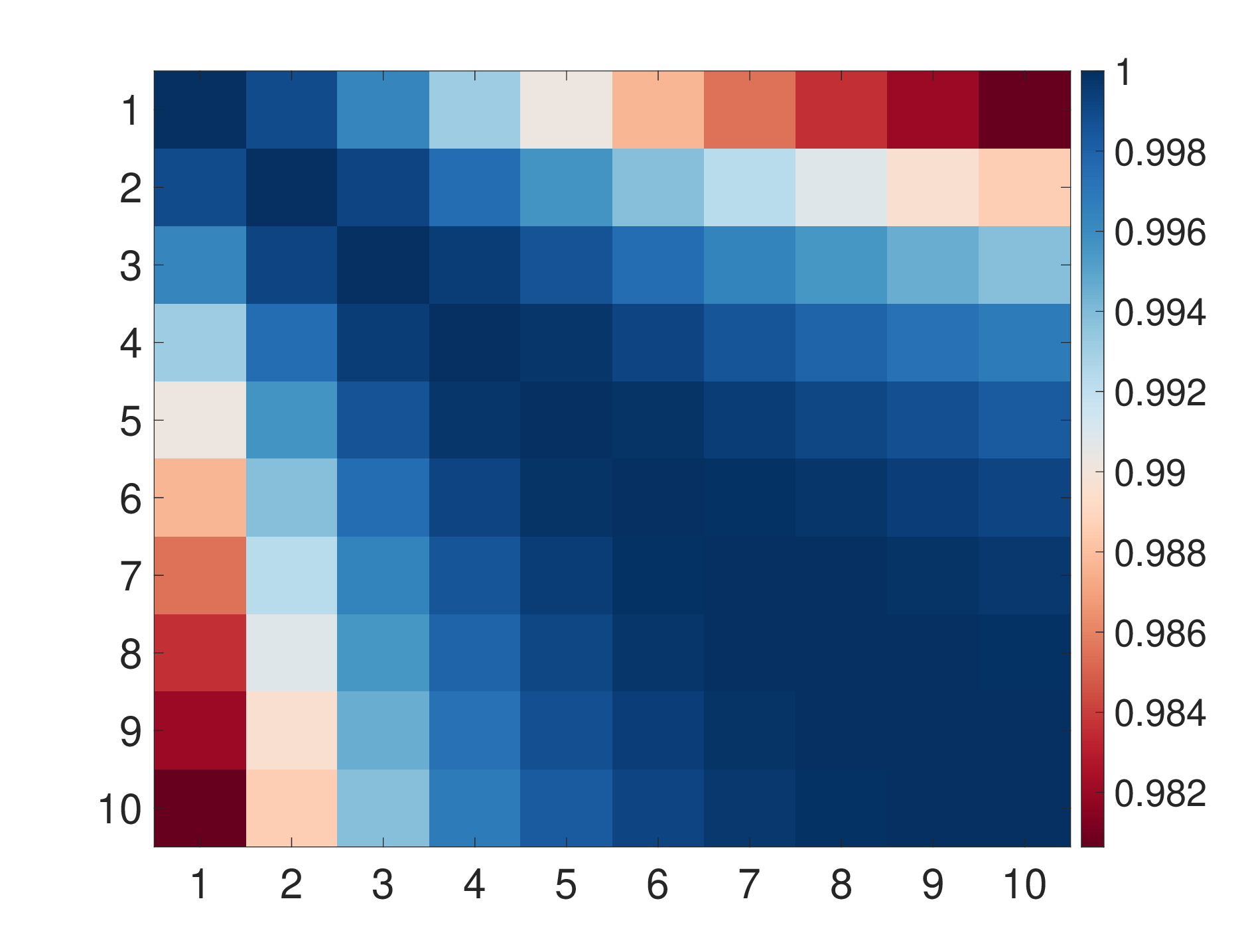}
\end{minipage} 
(b)\begin{minipage}{.45\textwidth}  \centering
\includegraphics[width=0.9\textwidth,height=0.6\textwidth]{ 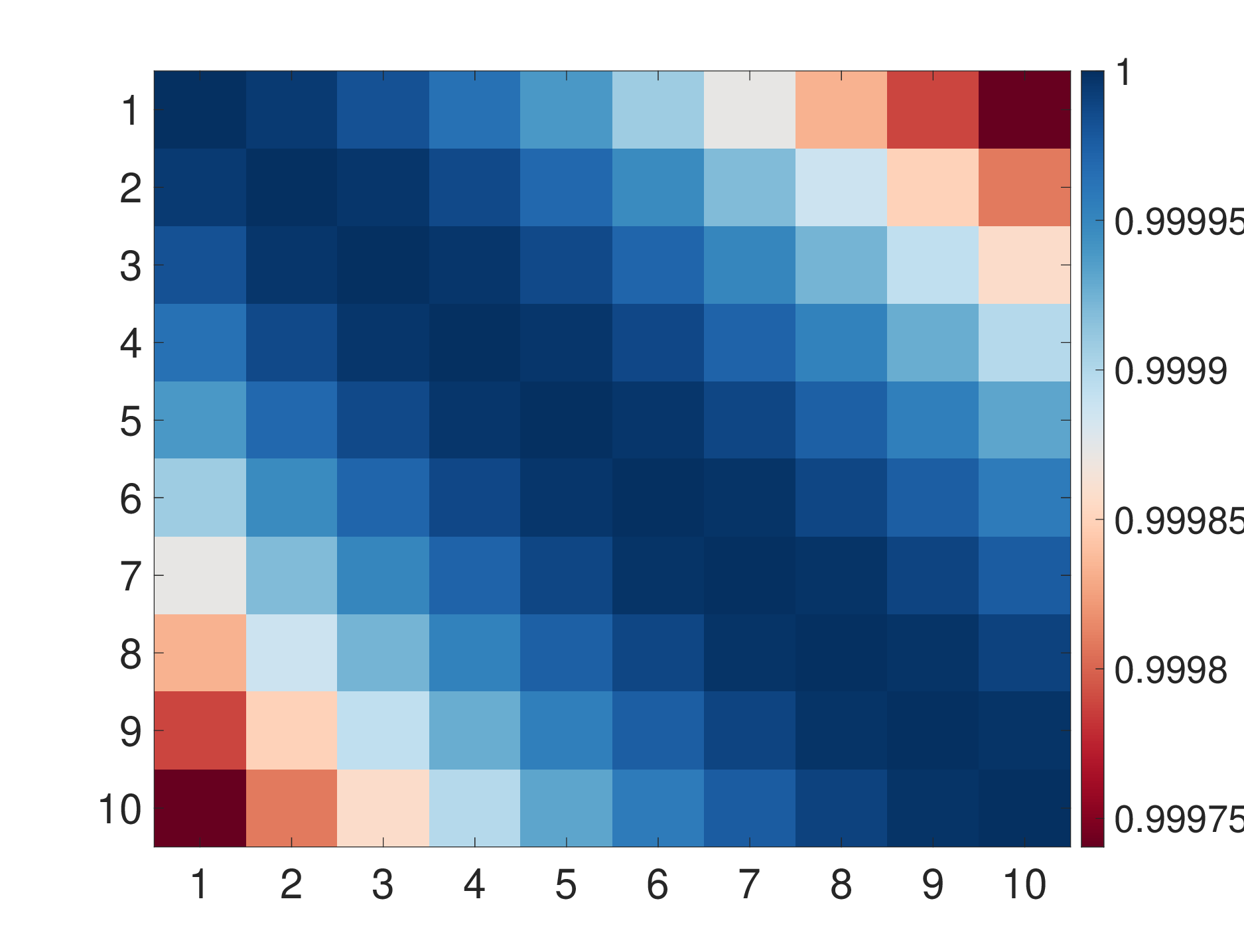}
\end{minipage}

\caption{Patterns of incoherence in the regression matrix, illustrating the entries of the product of the normalized regression matrix and its transpose. (a) corresponds to Example 2, where we use Gaussian type basis of size 10. (b) corresponds to Example 3, where we use polynomial basis of size 10. }\label{fig:1d_incoherence}
\end{figure}
\begin{figure}[!ht]
\centering (a)
\begin{minipage}{.45\textwidth} \centering 
\includegraphics[width=0.95\textwidth,height=0.6\textwidth]{ 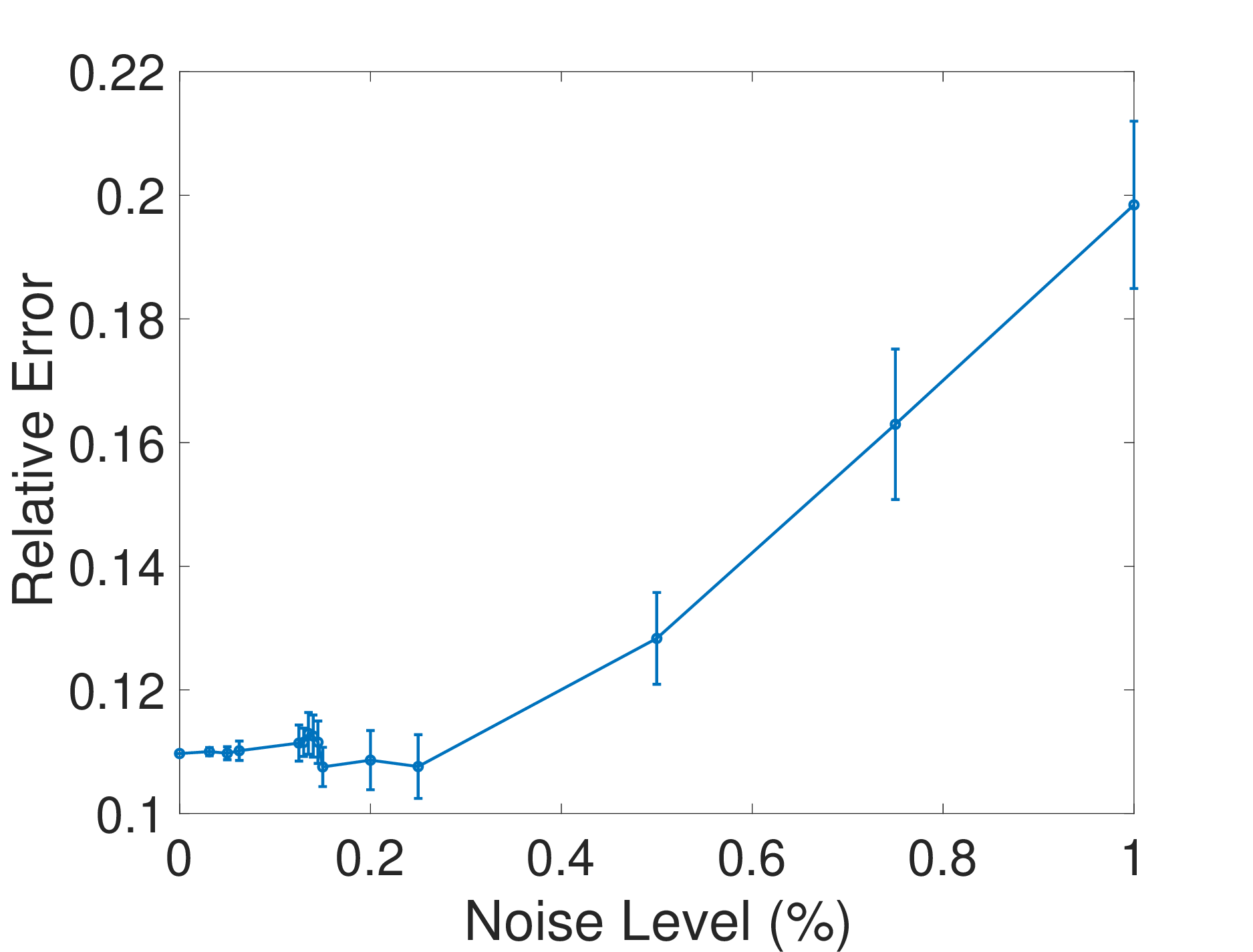}
\end{minipage}
(b)\begin{minipage}{.45\textwidth} \centering 
\includegraphics[width=0.95\textwidth,height=0.6\textwidth]{ 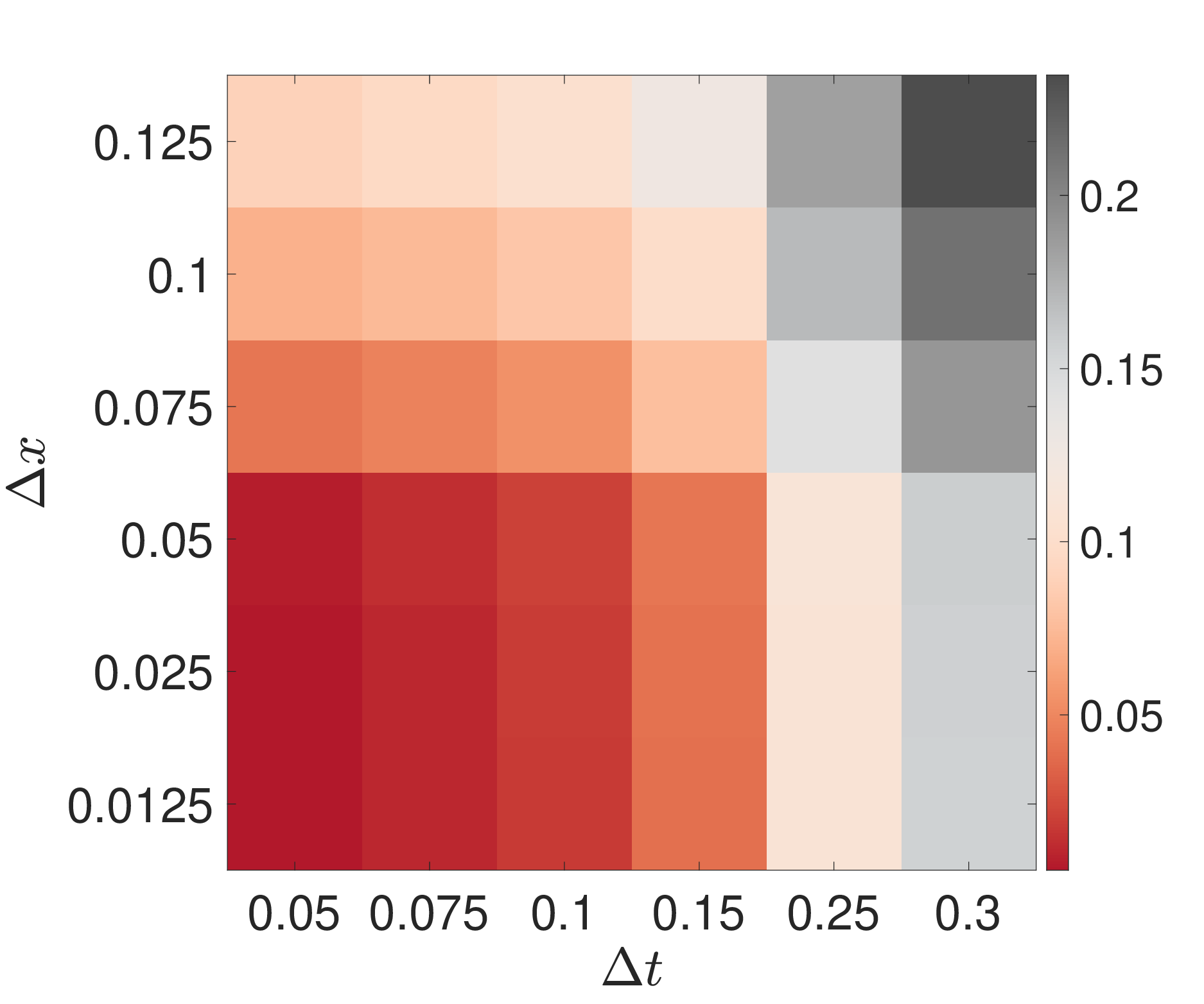}
\end{minipage} 

(c)\begin{minipage}{.45\textwidth}  \centering
\includegraphics[width=0.95\textwidth,height=0.6\textwidth]{ 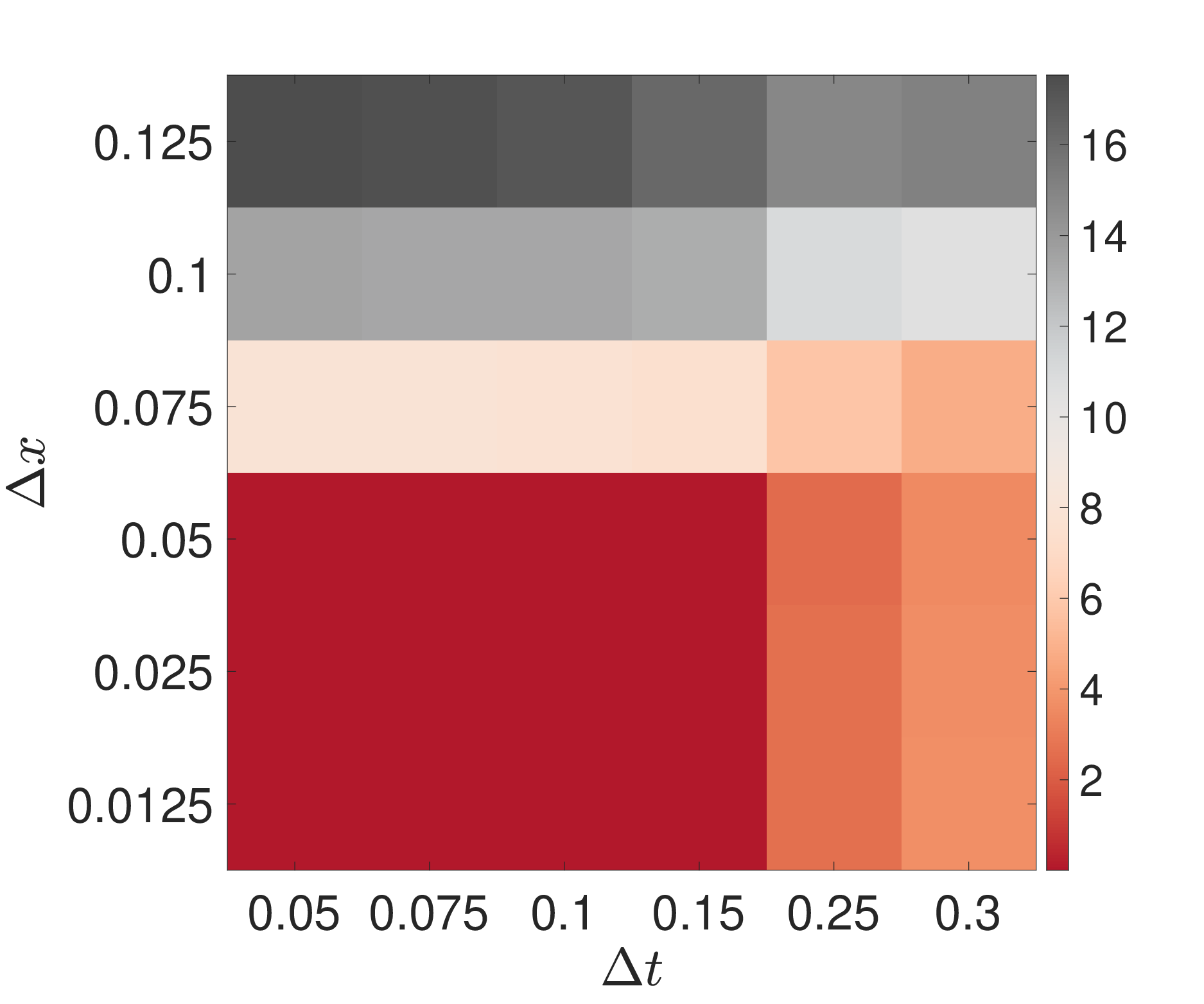}
\end{minipage}
(d)\begin{minipage}{.45\textwidth}  \centering
\includegraphics[width=0.95\textwidth,height=0.6\textwidth]{ 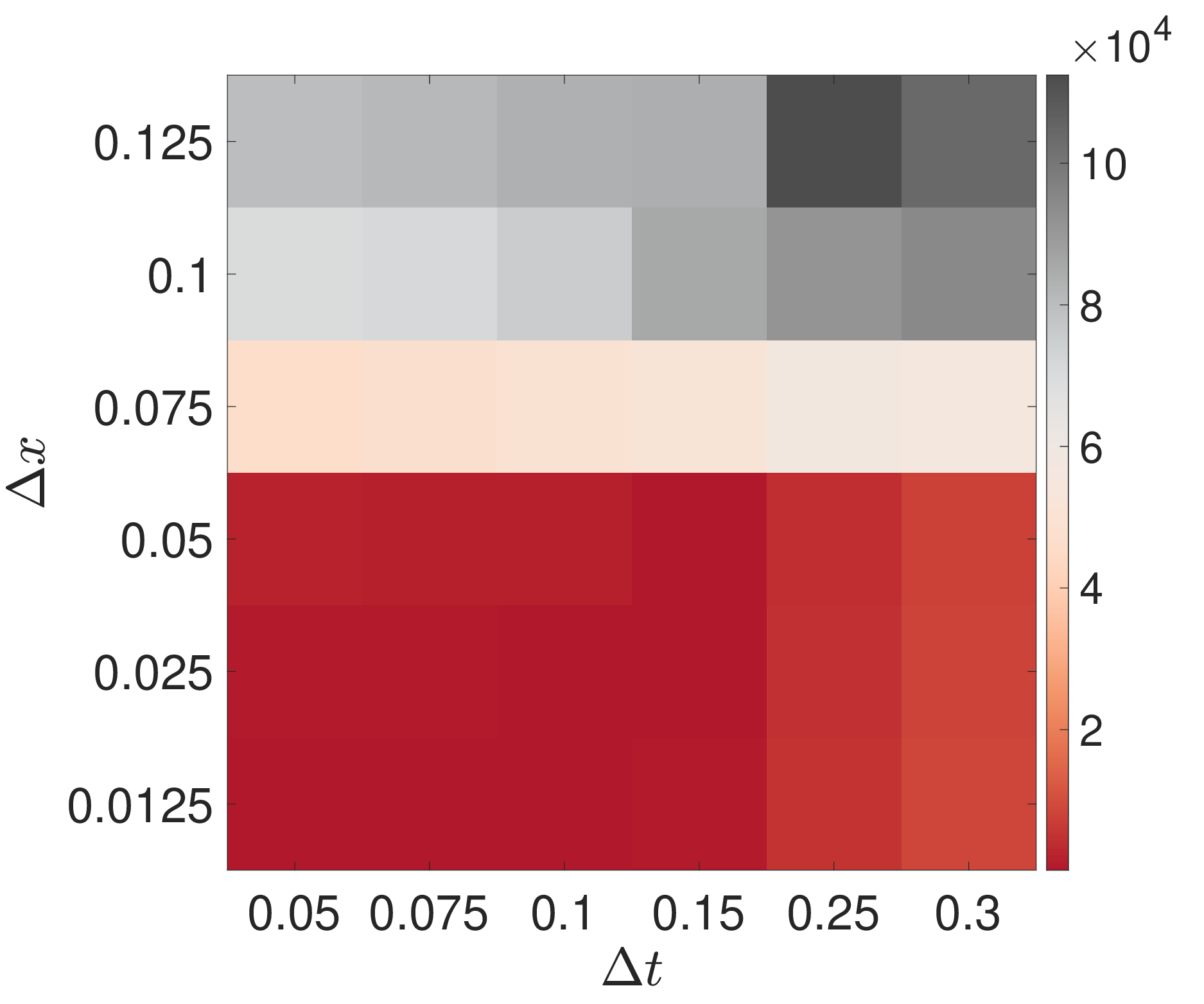}
\end{minipage}
\caption{Reconstruction errors for solution data in Figure \ref{fig:1dmeta}. (a) We have different levels of noise perturbation  where we display the mean and standard deviation of relative errors over 100 trials. (b)  Accuracy of PartInv with $K=2$ using different $(\Delta x, \Delta t)$.  (c) Accuracy of subspace pursuit with $K=2$. (d)  Accuracy of LASSO. For the LASSO algorithm, the Matlab-built-in LASSO solver was employed with the 'IndexMinMSE' option.}\label{fig:rec_1dmeta_perturbation}
\end{figure}

\paragraph{Example 3 (Linear diffusion with external potential $V$)}

In this one-dimensional example, we consider an external confinement potential given by a double-well function and linear diffusion. Therefore we have
\[
\partial_t\rho=\kappa\partial_{xx}^2\rho+(\rho(W*\rho+V)_x)_x\ ,\ \textnormal{where}
\]

\begin{equation*}
W(x)=\frac{|x|^2}{2}\ ,\qquad V(x)=\frac{|x|^4}{4}-\frac{|x|^2}{2}\ .
\end{equation*}
This equation describes a model for self-propelled agents \cite{barbaro2016phase} with a noisy term given by the linear diffusion. The confinement potential describes the tendency of individuals to move in a preferred direction while the interaction potential $W$ models the alignment component of the movement.
We consider $\kappa=0.1$ and simulation parameters are provided in Table \ref{t:KF_params}. As expected, low values of the diffusion coefficient result in flocking for certain initial conditions \cite{CGPS18}.

\begin{table}[H]
\centering
\begin{tabular}{| c | c | c | c | c | c |c|}
\hline 
 $\delta t$ & $\delta x$ & Time domain  & Spatial domain & Initial condition & $\phi (|x|)$  \\ 
\hline 
 $ 10^{-2}$ & $1.2\cdot10^{-2}$ &$[0,5]$ & $[-6,6]$ &  $\mathcal{N}(0,0.3^2)$ & $|x|$\\
\hline
\end{tabular}
\caption{\textmd{{(KF) { Parameters to produce the solution data using the finite volume scheme.}} }}
\label{t:KF_params}
\end{table}
Figure \ref{fig:kF} illustrates the solution profile used as training data. Given that we started with symmetric initial data, a symmetric steady state is anticipated \cite{bailo2020fully}. For kernel estimation, we use a polynomial basis of the form \(\left\{\left(\frac{|x|}{6}\right)^{n}: n=0, \ldots, 9\right\}\), such that \(\phi\) is \(1\)-sparse relative to this dictionary. More explicitly, the coefficient for \(\psi_2\) is \(6\), because $\psi_2$ corresponds to $n=2$ in the previous set. The normalization factor of \(6\) on the basis ensures bounded entries in the matrix \(\mathbf{A}_{n,M,L}\).

Provided continuous-time trajectory data, the uniqueness of a $1$-sparse solution to the normal Equation \eqref{normalequation} lies in the prerequisite that any two distinct columns of matrix $\mathbf{A}$
 are linearly independent. This 1-sparse solution is the coefficient of the true interaction kernel. 

Nonetheless, our numerical result reveals that every pair of columns in \(\mathbf{A}\) tends to exhibit near-linear dependence, a phenomenon   evidenced by the coherence patterns manifested in its empirical regression matrix \(\mathbf{A}_{n,M,L}\), as shown in Figure \ref{fig:1d_incoherence} (b). This suggests that sparse identification of kernels from discrete noisy data is expected to  be difficult, even though the ground truth is $1$-sparse. This anticipation is, in part, foreseeable, considering that the potential of type $|x|^n$ is capable of promoting analogous collective dynamics. 

The unfavorable coherent patterns lead to failures in LASSO and SINDy estimators, as depicted in Figure \ref{fig:rec_KF_perturbation} (c)-(d). In contrast, the greedy type methods yield much more accurate estimations see Figure \ref{fig:rec_KF_perturbation} (a) for PartInv and (b) for CoSamP.

When the solution data is tainted by noise, kernel identification becomes considerably more challenging. In Figure \ref{fig:rec_KF_perturbation} (e), we employ PartInv on the data, introducing a noise level of $0.5\%$ and imposing a sparsity constraint \(K=2\).  Notably, even when {PartInv} outputs a support set \(\{1,2\}\) encompassing the true support \(\{2\}\), the estimated coefficients diverge substantially from the ground truth (see 3rd row of Table \ref{tab:KF_supppruning}). We first used {RE} to narrow down two candidates and then compute their {TEEs} utilizing a space-time mesh size \((\delta x, 0.1\delta t)\) and evolve the PDE over the time interval \([0,0.1]\). In this instance, the TEE emerges as a robust quantitative metric, aiding in the identification of the correct support set \(\{2\}\). As a result, we obtain a significantly accurate estimation.\\

For this example, we also compare our loss function with the PDE residual. In Figure \ref{fig:rec_KF_perturbation} (f), we show that,
even when provided with the true support, the restricted least squares (note that this is the optimal outcome attainable through a sparsity-promoting algorithm) yields an estimated coefficient that is notably divergent from the ground truth of {  6}. This demonstrates the advantages of our loss functional over the PDE residual. 

\begin{figure}
\centering
(a) \begin{minipage}{.45\textwidth} \centering  
\includegraphics[width=0.72\textwidth,height=0.48\textwidth]{ 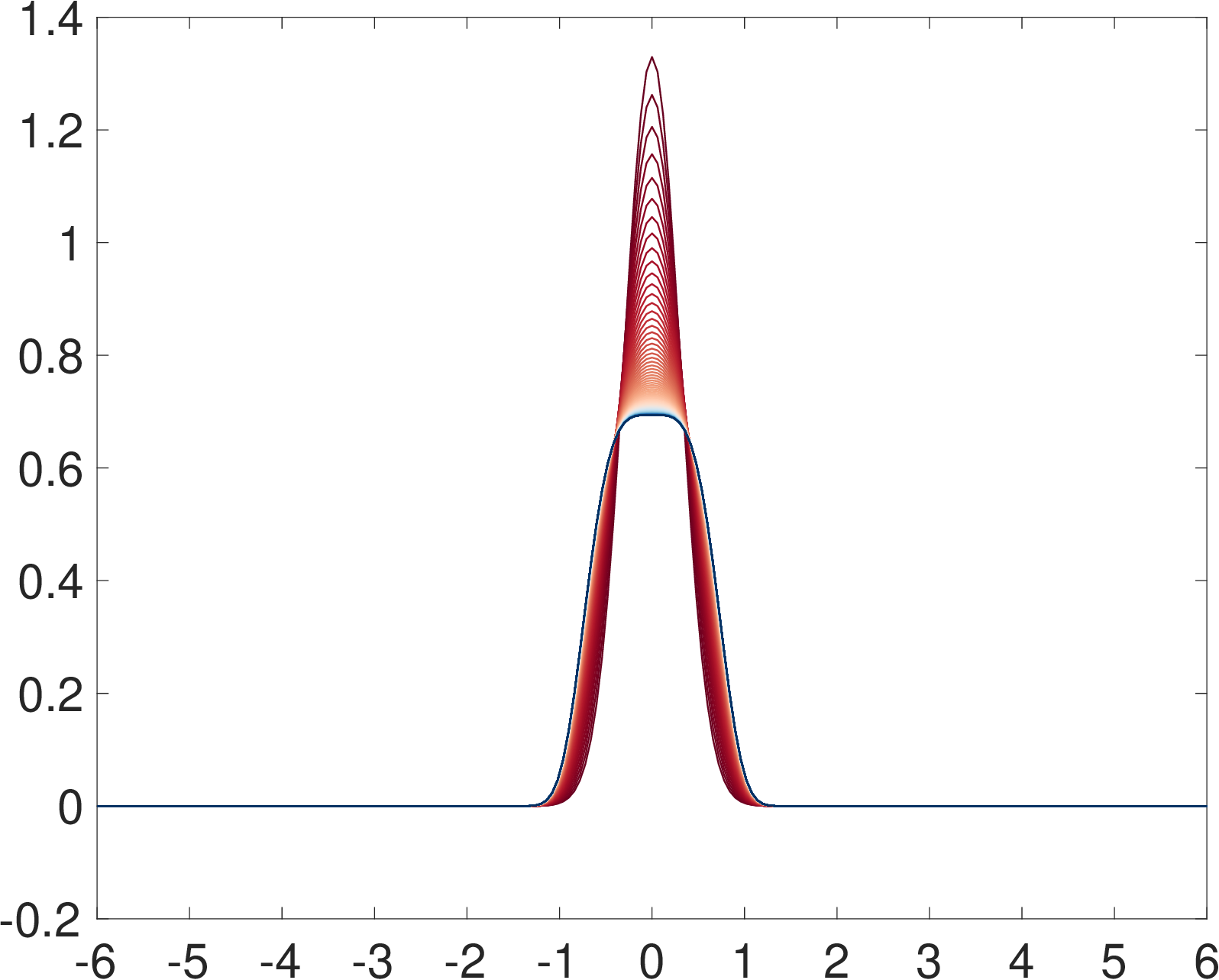}
\end{minipage}
(b) \begin{minipage}{.45\textwidth} \centering  
\includegraphics[width=0.72\textwidth,height=0.48\textwidth]{ 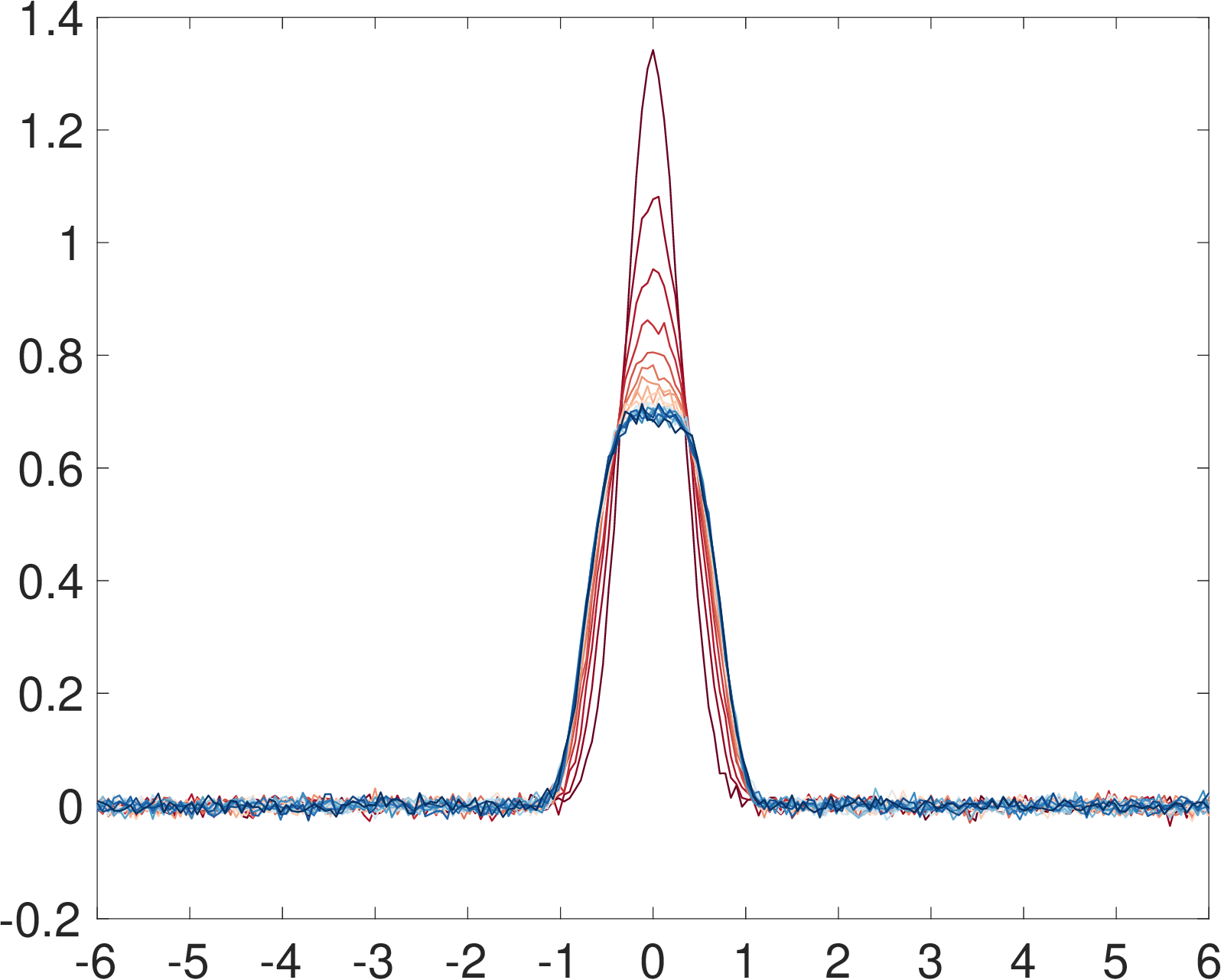}
\end{minipage}

\caption{Profile of the solution for $\Delta x = 5 \delta x, \Delta t = 5 \delta t$.  (a) a subset of the solution data generated from the numerical solver (b) the solution data with 0.5\% noise added. }\label{fig:kF}
\end{figure}

\begin{figure}[htbp]
\centering (a)
\begin{minipage}{.45\textwidth} \centering

\includegraphics[width= 0.95\textwidth,height=0.6\textwidth]{ 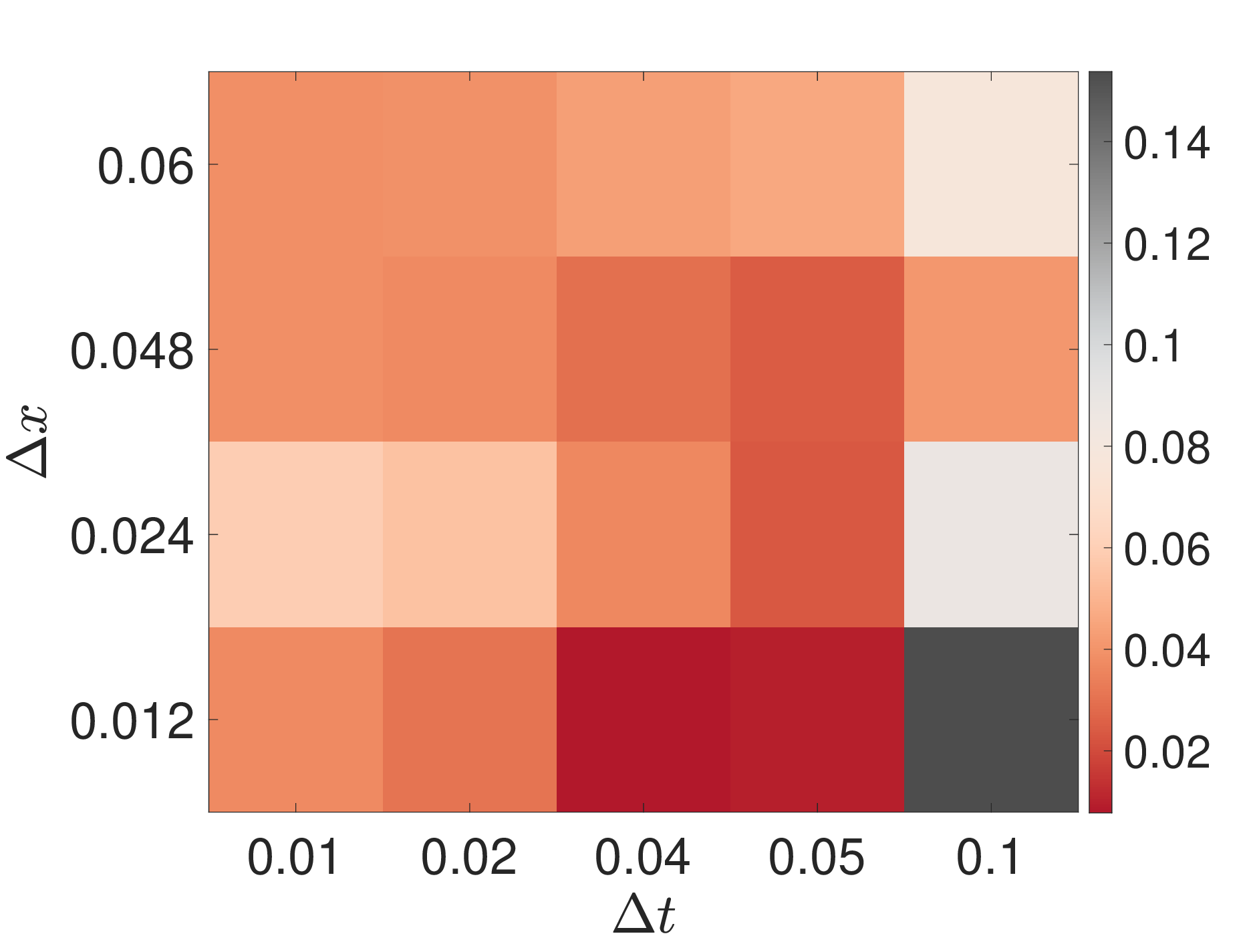}
\end{minipage} 
(b)\begin{minipage}{.45\textwidth}  \centering

\includegraphics[width=0.95\textwidth,height=0.6\textwidth]{ 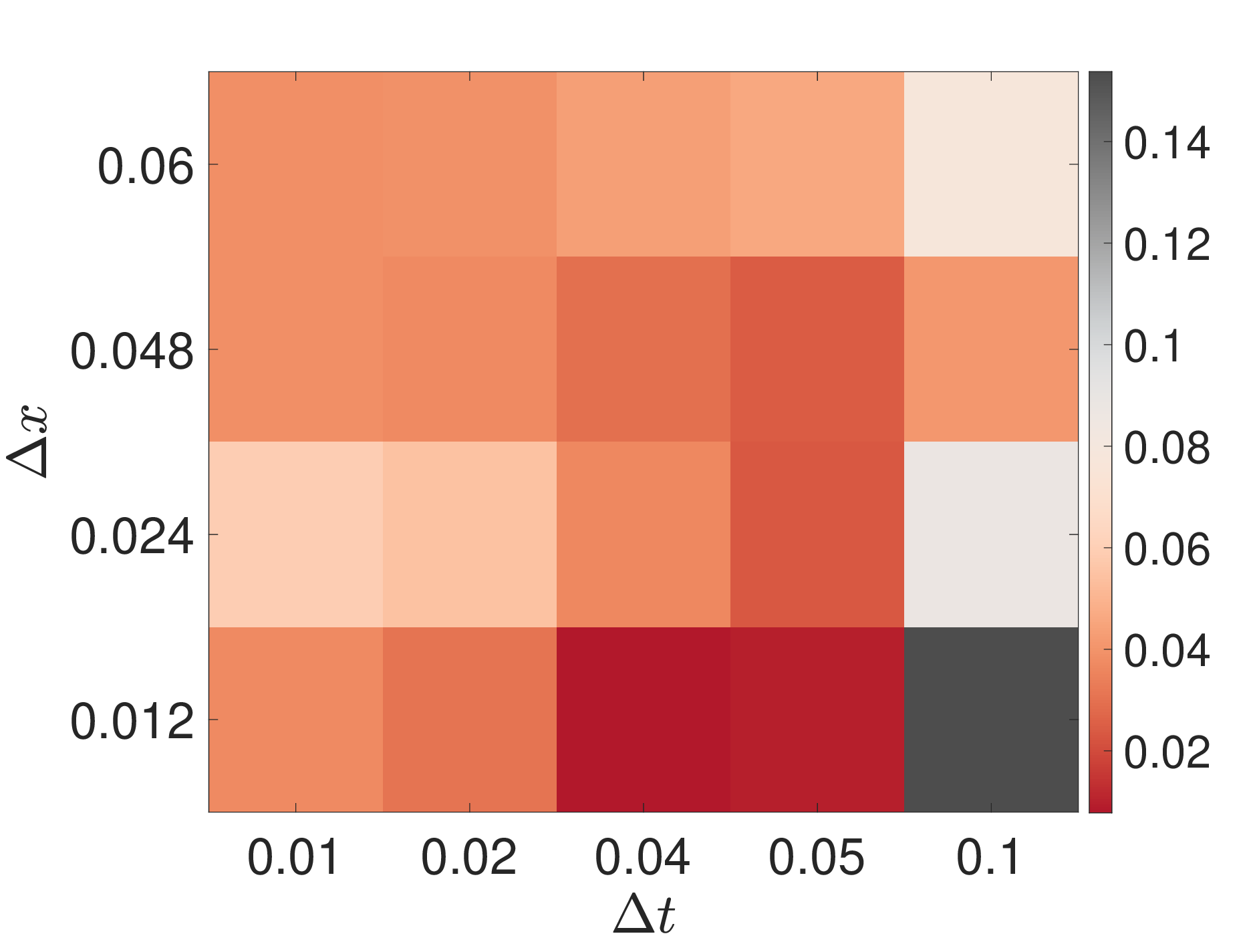}
\end{minipage}

 (c)
\begin{minipage}{.45\textwidth} \centering

\includegraphics[width= 0.95\textwidth,height=0.6\textwidth]{ 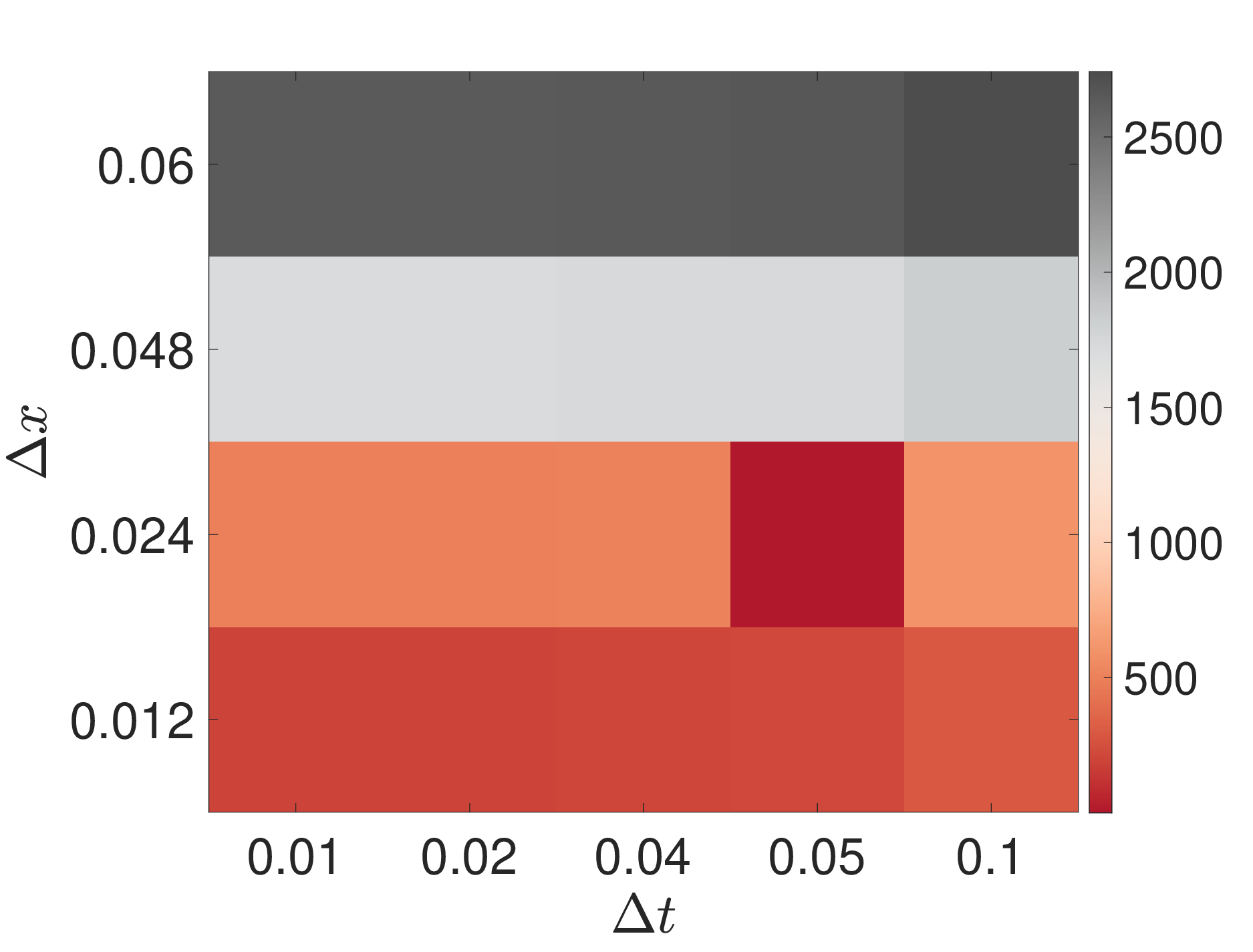}
\end{minipage} 
(d)\begin{minipage}{.45\textwidth}  \centering

\includegraphics[width=0.95\textwidth,height=0.6\textwidth]{ 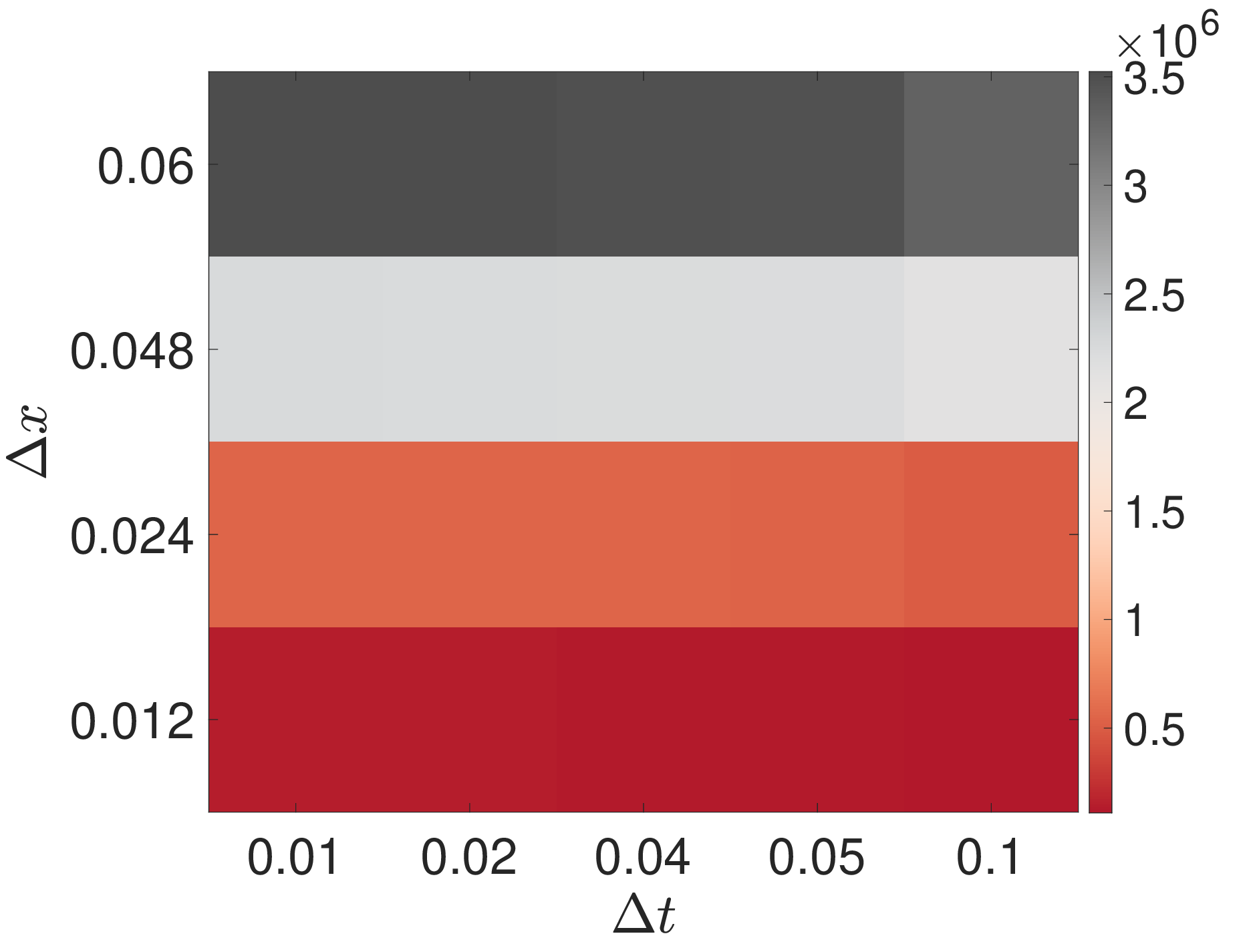}
\end{minipage}

(e)\begin{minipage}{.45\textwidth} \centering
\includegraphics[width=0.78\textwidth,height=0.55\textwidth]{ 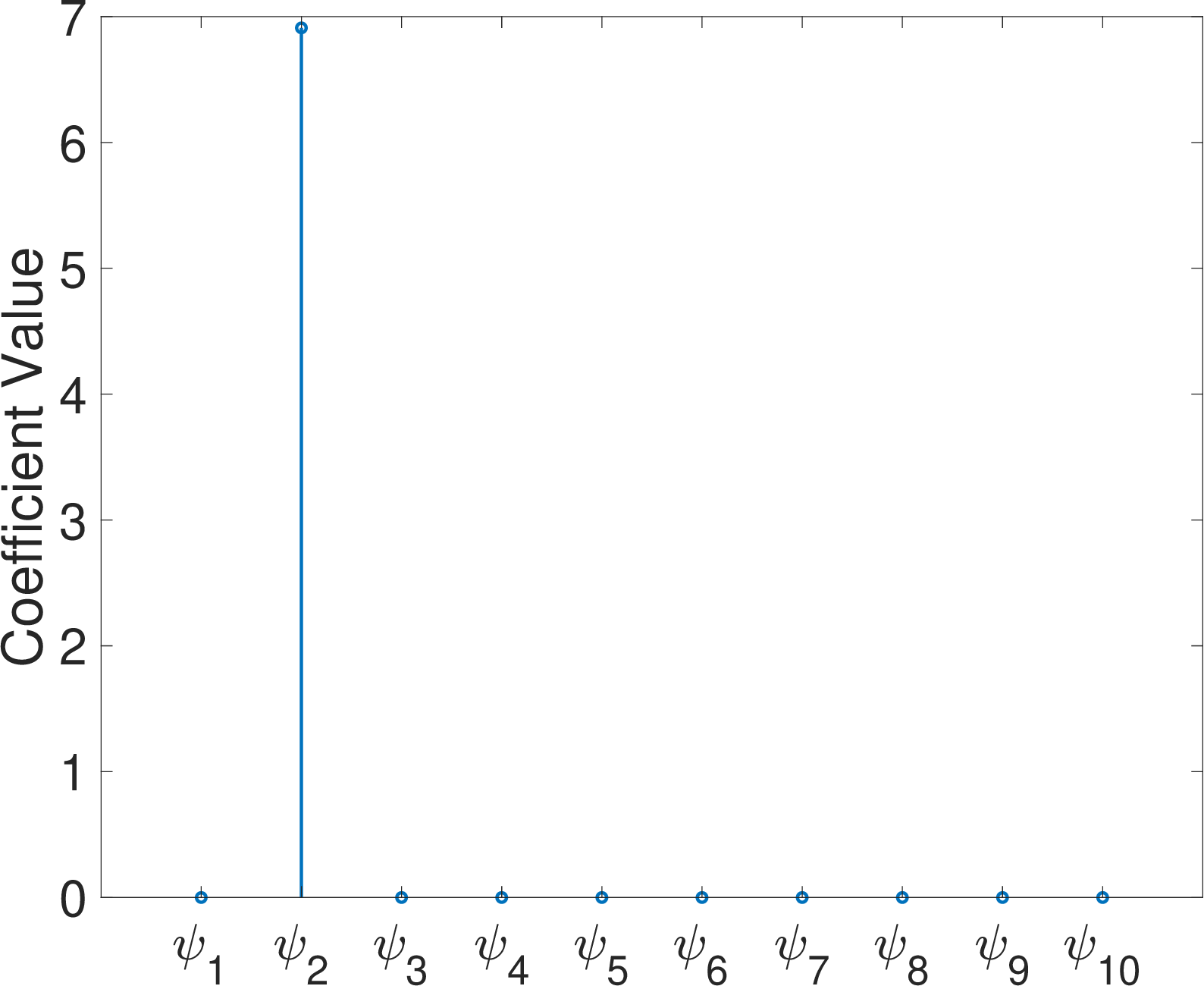}
\end{minipage}
(f)\begin{minipage}{.45\textwidth}  \centering
\includegraphics[width=0.9\textwidth,height=0.6\textwidth]{ 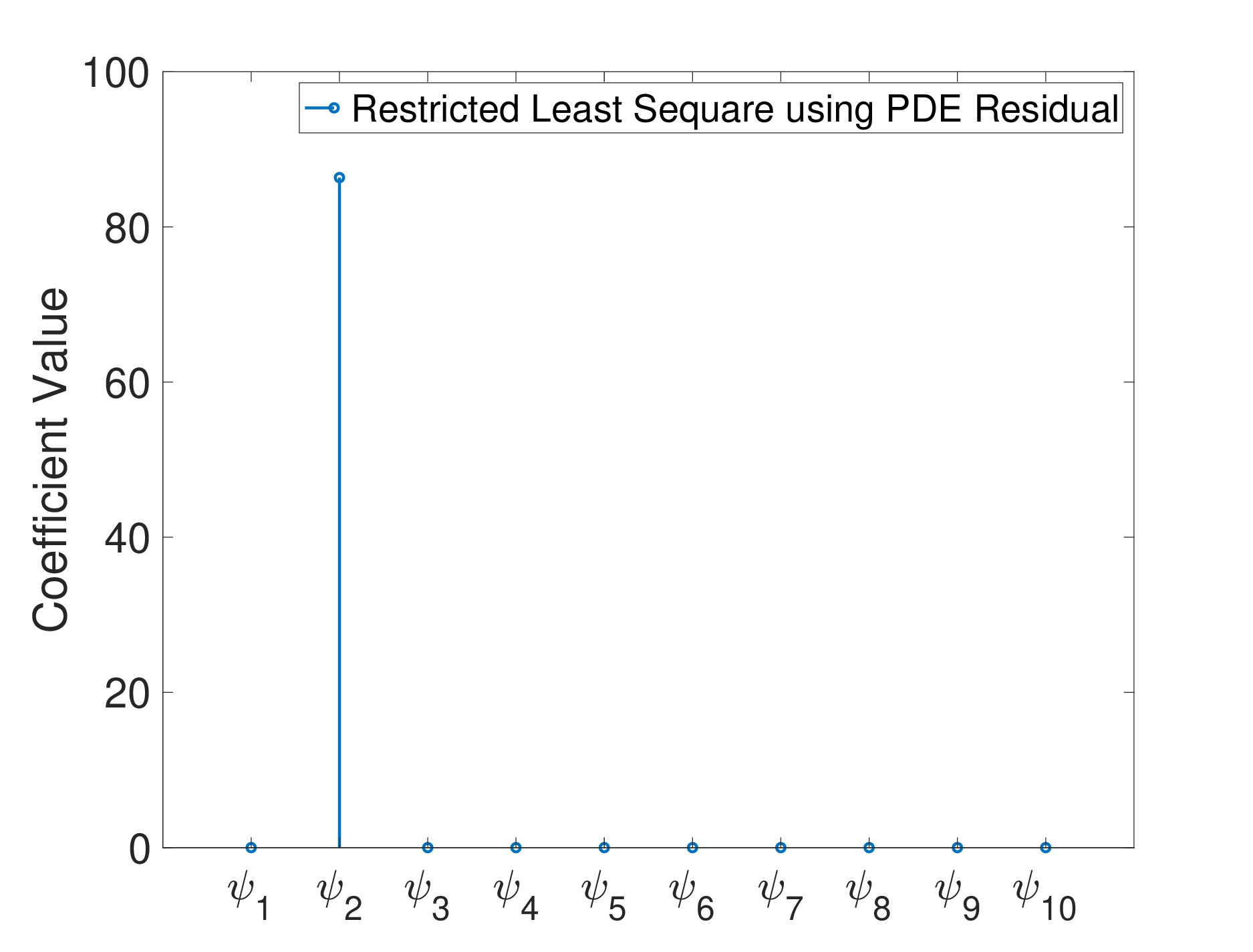}
\end{minipage}
\caption{ (a) the learning outcomes of PartInv with a parameter \(K=3\) with respect to space-time mesh size.  (b)-(d): a comparative analysis with  subspace pursuit algorithms (b), LASSO (c), and SINDy (d) using identical training data, where the subspace pursuit algorithm was configured with a sparsity level set to 3.  We found results in (a) and (b) are very close. In (e) we used  the same training data as in Figure \ref{fig:kF} (b) by choosing $K=2$ and performing support tuning where the numerical values are summarized in Table \ref{tab:KF_supppruning}. In (f), the training data is the same as in (e) and we display the restricted least squares estimator using the PDE residual provided the true support set $\{2\}$.  }\label{fig:rec_KF_perturbation}
\end{figure}

\begin{table}[H]
    \centering
    
 \begin{tabular}{|c|c|c|c|c|}
        \hline
        Active terms & Coefs & RE & TEE \\
        %\hline
        $\psi_1$ & 0.87 & -1.22 & 0.39 \\
        %\hline
        $\psi_2$  & \textbf{6.91} & -3.34 & \textbf{0.13}\\  
      $[\psi_1,\psi_2]$ & [40.20, -313.11] & 121.19&\\
    \hline
    \end{tabular}\caption{Numerical values for support pruning. When $K=2$, PartInv produced $I^{(k)}=\{1,2\}$ and we perform support pruning by calculating the TEE on time steps $0:5\delta t: 10 \delta t$. }
    \label{tab:KF_supppruning}
\end{table}

\subsection{Two dimensional examples}

In this part, we delve into the performance of two dimensional examples. Compared to one dimension scenarios, 2D examples {require considerably more computational resources since} the computation of the integration kernel $G$, given by \eqref{ggg}, becomes increasingly complex with higher dimensionality, reaching computational limits if the mesh size exceeds 100. We present two specific examples: the first involves closed-form analytic solutions, a scenario free from forward errors.  The second example features observations on a very coarse scale, designed to evaluate the effectiveness and robustness of the PartInv method with the implementation of support pruning.

\paragraph{Example 4 (2D Fokker Planck Equation with Nonlinear diffusion)} 

In this 2D scenario, we examine a nonlinear diffusion case characterized by $m=2$, with $\kappa$ set to 1. The functions $W(\mathbf{x})$ and $V(\mathbf{x})$ are defined as follows:
\begin{equation*}
W(\mathbf{x})=\frac{|\mathbf{x}|^2}{2}, \qquad V(\mathbf{x})=0.
\end{equation*}
In this case we have
\begin{equation*}
\nabla W(\bx)= \nabla\Phi(|\bx|)=\phi(|\bx|)\frac{\bx}{|\bx|}\ .
\end{equation*}
Note that now $\bx\in \mathbb{R}^2$, and to avoid instability issues when \(\boldsymbol{x}\) is close to the origin, we learn the kernel \(\frac{\phi(|\boldsymbol{x}|)}{|\boldsymbol{x}|}\) instead. Analogous to the Example 3, we employ the polynomial basis \(\{1,|\boldsymbol{x}|,\ldots, |\boldsymbol{x}|^{n-1}\}\) with \(n=10\), {so the true kernel is  $1$-sparse to this dictionary of size $10$}. For the training data, we use the closed form of its stationary solution given by \(\rho_t(\mathbf{x}) =  \max\left(\sqrt{\frac{1}{\pi}}-|\mathbf{x}|^2,0\right)\). The computational parameters are summarized in Table \ref{t:2DFP_params}.

\begin{table}[H]
\centering
\begin{tabular}{| c | c | c | c |}
\hline 
 Time domain  & Spatial domain & $\frac{\phi (|\bx|)}{|\bx|}$  \\ 
\hline 
 $[0,0.1]$ & $[-2,2]\times [-2,2]$ & $1 $\\
\hline
\end{tabular}
\caption{\textmd{{(2DFP) Parameters to produce the solution data using the finite volume scheme.} }}
\label{t:2DFP_params}
\end{table}

\begin{figure}
\centering (a)
\begin{minipage}{.45\textwidth} \centering

\includegraphics[width= 0.9\textwidth,height=0.6\textwidth]{ 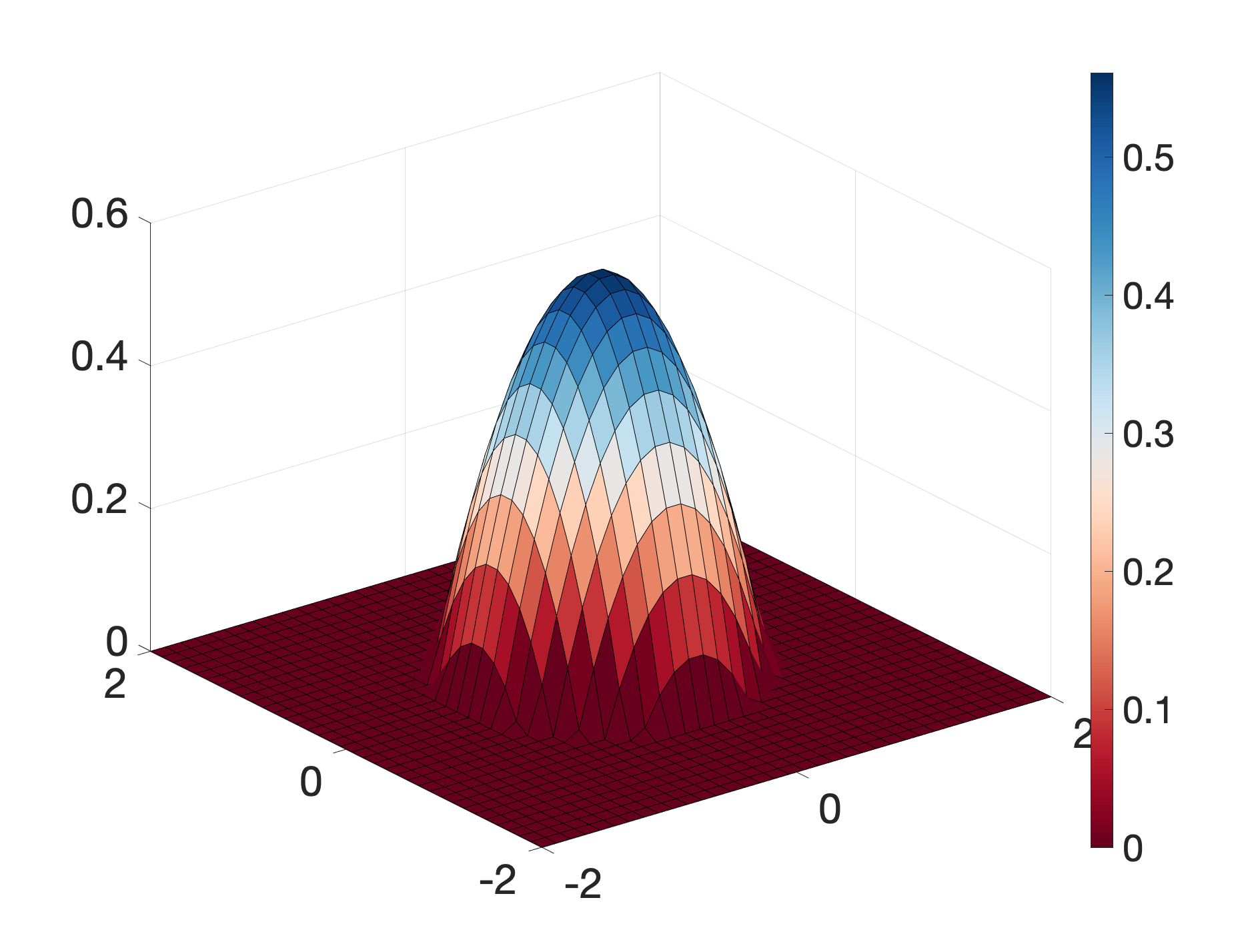}
\end{minipage} 
(b)\begin{minipage}{.45\textwidth}  \centering

\includegraphics[width=0.9\textwidth,height=0.6\textwidth]{ 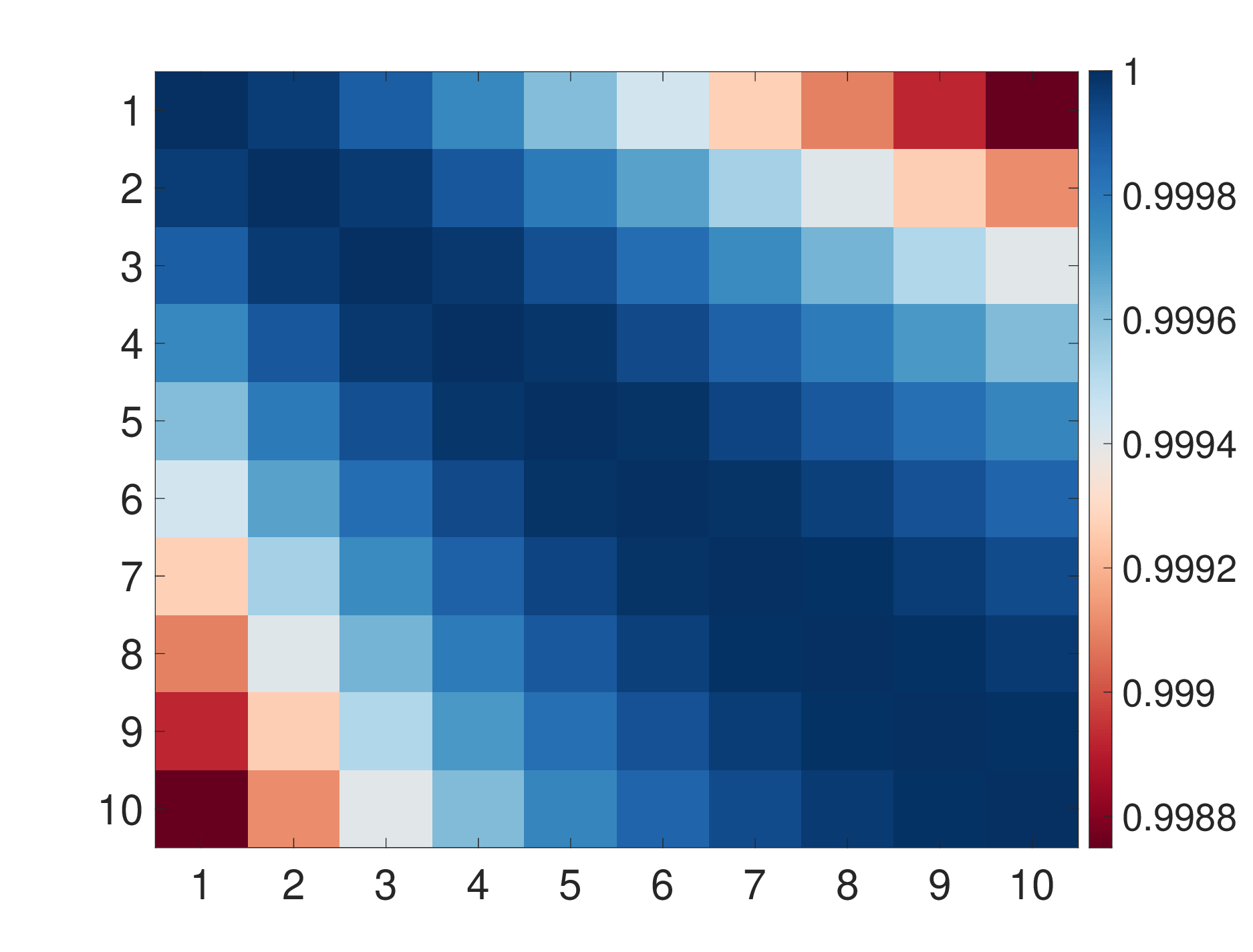}
\end{minipage}

(c)\begin{minipage}{.45\textwidth} \centering
\includegraphics[width=0.9\textwidth,height=0.6\textwidth]{ 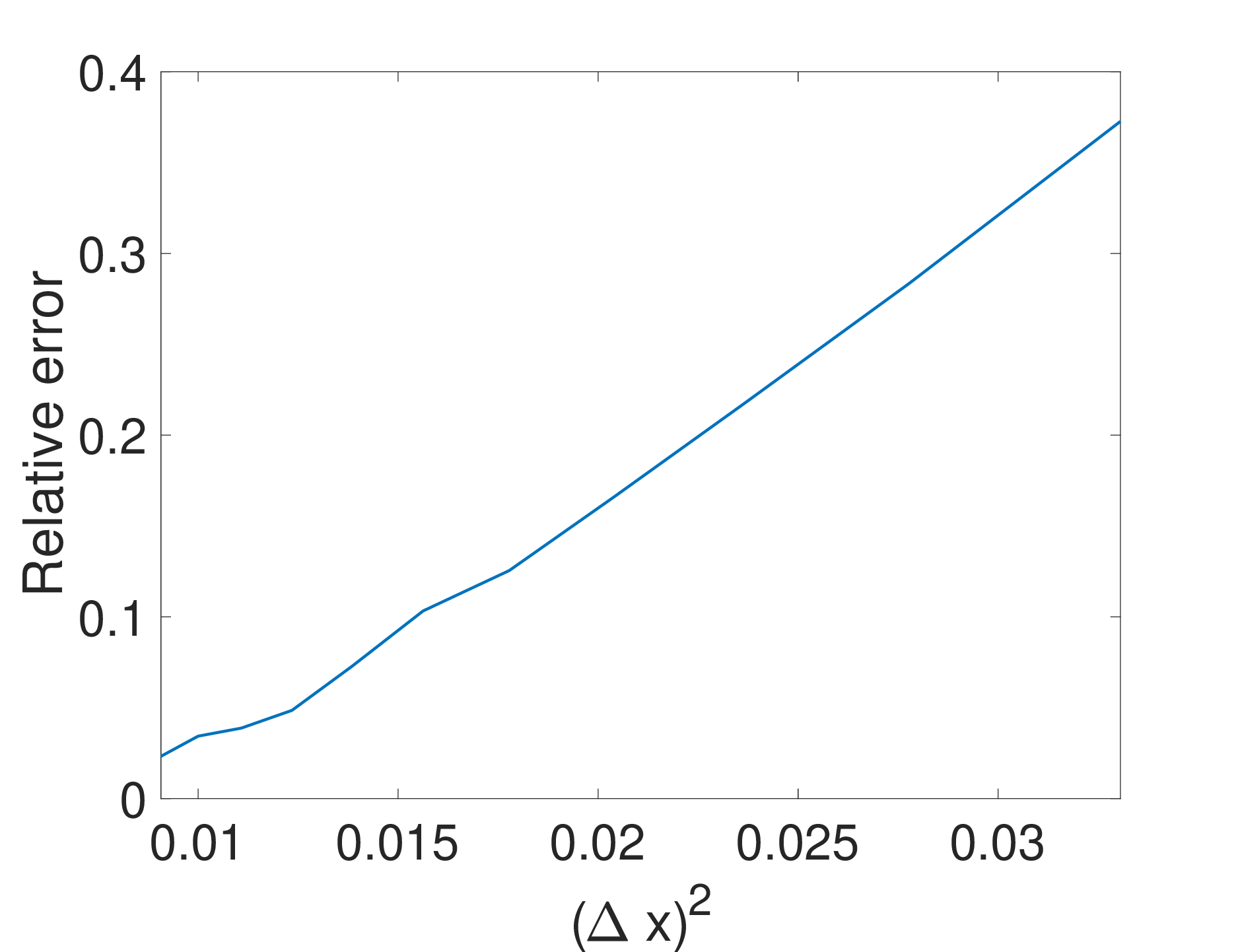}
\end{minipage}
(d)\begin{minipage}{.45\textwidth}  \centering
\includegraphics[width=0.81\textwidth,height=0.54\textwidth]{ 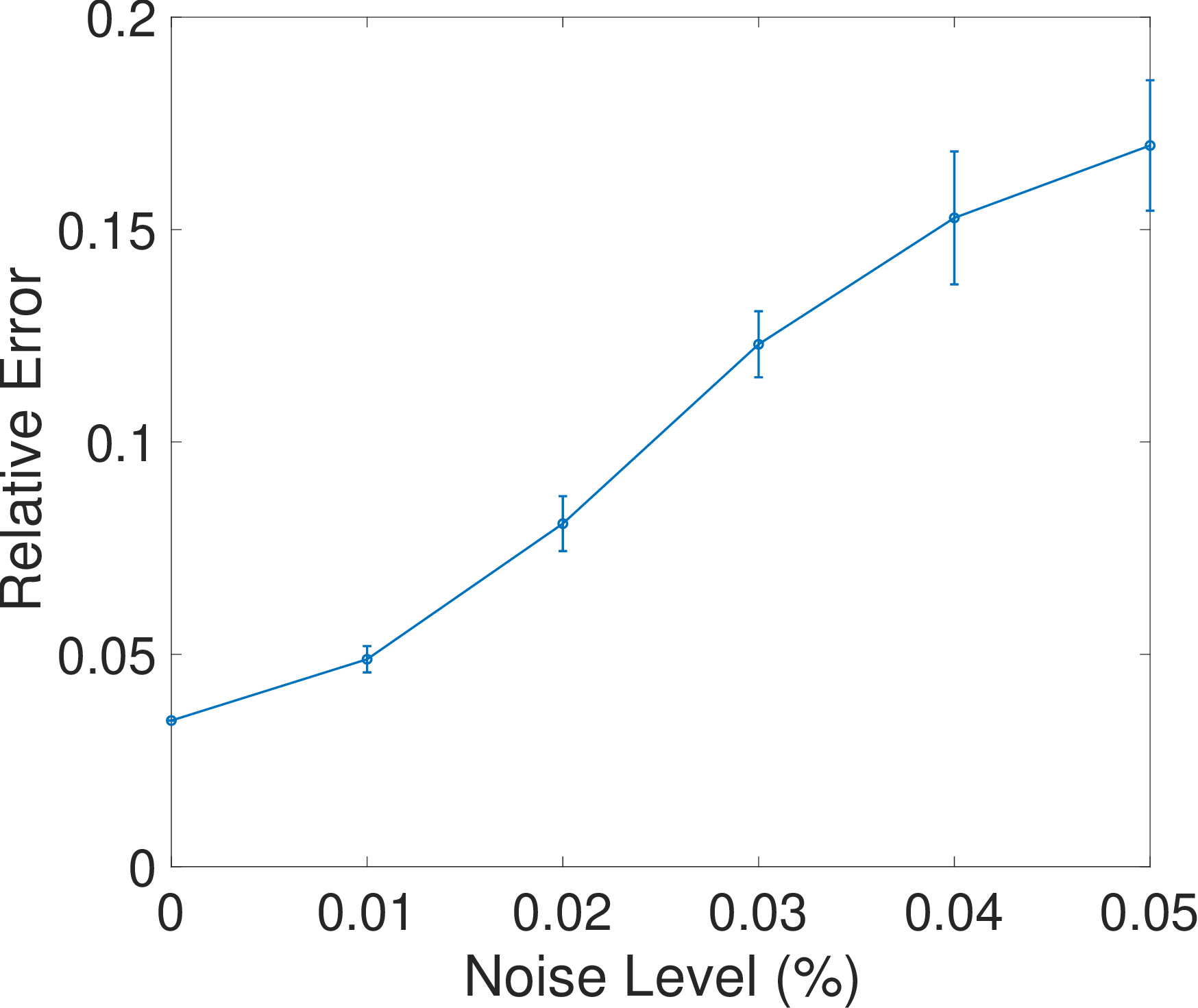}
\end{minipage}
\caption{ \textbf{Top Panel:} (a) The profile of the stationary solution at a single time instance (b) The coherence pattern of the regression matrix $\mathbf{A}_{n,M,L}$.   \textbf{Bottom Panel:}  (c) We run PartInv with $K=1$ and display the relative error versus the squares of space mesh size finding an approximately linear relationship. (d) We test the robustness of PartInv with $K=1$ for a variety of noise levels. }\label{fig:rec_stationary_perturbation}
\end{figure}
\noindent In the context of this identification problem, we observe a phenomenon within the regression matrix \(\mathbf{A}_{n,M,L}\) that mirrors { Example 2 and 3: it has highly coherent columns (see Figure \ref{fig:rec_stationary_perturbation} (b)).} Consequently, the estimation of the coefficient is acutely sensitive to the choice of nonzero locations, requiring the identification of the accurate support of the ground truth.

In Figure \ref{fig:rec_stationary_perturbation} (c), across various mesh sizes defined as \(\Delta x = \Delta y = \frac{4}{[22:2:40]}\), the PartInv algorithm with \(K=1\) {produced accurate estimators by identifying the correct support set} \(\{1\}\), and therefore effectively tackles the data corruption coming from discrete-time observations in this challenging basis pursuit problem. {For competitor methods, even with the training data where we have smallest mesh size  \(\Delta x = \Delta y = 0.1\),  the CoSaMP (also subspace pursuit) and LASSO estimators are not accurate and even failed to find the right support set.} 

As we use an analytic solution for the underlying PDE, the only error source is coming from the discrete time observations. In Figure \ref{fig:rec_stationary_perturbation} (c), we found that the convergence rate of the relative error with respect to space-time mesh size aligns with our theoretical error analysis in  Proposition \ref{prop: error A and b}. Given the absence of numerical error from the solver, the second-order convergence is achieved as we do not need to compute the time derivatives here. 

The robustness relative to the measurement noise was also tested, as depicted in Figure \ref{fig:rec_stationary_perturbation} (d).

\paragraph{Example 5 (2D nonlinear diffusion with nonlocal interaction)} In this example, we explore the performance of our algorithm using very coarse scale data that are subjected to both discretization and numerical solver forward errors.

We consider an initial condition and interaction potential given by
$$
    \rho_0(\mathbf{x})= 5\Bigl(\frac{e^{-((x+0.5)^2+(y+0.5)^2)}}{0.2}+ \frac{e^{-((x-0.5)^2+(y-0.5))^2}}{0.2}\Bigr)\ ,\quad W(\mathbf{x})=-3e^{-2{|\mathbf{x}|^2}}\ ,
$$
respectively. We set $m=2$, $\kappa=1$ and the rest of the computational parameters are summarized in Table \ref{t:2Dmetastable_params}.

\begin{table}[H]
\centering
\begin{tabular}{| c | c | c | c | c | c|}
\hline 
 $\delta t$ & $\delta x$ & Time domain  & Spatial domain &  $\frac{\phi (|\mathbf{x}|)}{|\mathbf{x}|}$  \\ 
\hline 
 $10^{-3}$ & $ 2*10^{-1}$ &$[0,0.05]$ & $[-2.1,2.1]\times[-2.1,2.1]$ &   $ 12e^{-2|\mathbf{x}|^2}$\\
\hline
\end{tabular}
\caption{\textmd{{Parameters to produce the solution data using the finite volume scheme.} }}
\label{t:2Dmetastable_params}
\end{table}

To estimate the interaction kernel as in Table \ref{t:2Dmetastable_params}, we use a basis of the form  $\{-2w\exp(-w|x|^2): w=1:1:10\}$. Then the true interaction kernel is $1$-sparse with respect to this particular basis representation.  Figure \ref{fig:2dmeta} (c) shows that it yields  a very coherent basis in our sparse learning problem. In this example, we use solution data defined on a coarse mesh as shown in  Figure \ref{fig:2dmeta} (a) and (b) corresponding to different times. When we set the $K=1$,  PartInv  failed to find the right support.  We then set $K=2$, and  run our support pruning algorithm  setting $\widehat \Delta t  =10^{-4}$ and $\widehat \Delta x = 10^{-1}$. We observed that under a variety of noise levels, our algorithm enables accurate estimation thanks to the support pruning step which consistently found the right support. See the results in Figure \ref{fig:2dmeta} (d). 
\begin{figure}[H]
\centering (a)
\begin{minipage}{.45\textwidth} \centering

\includegraphics[width= 0.9\textwidth,height=0.6\textwidth]{ 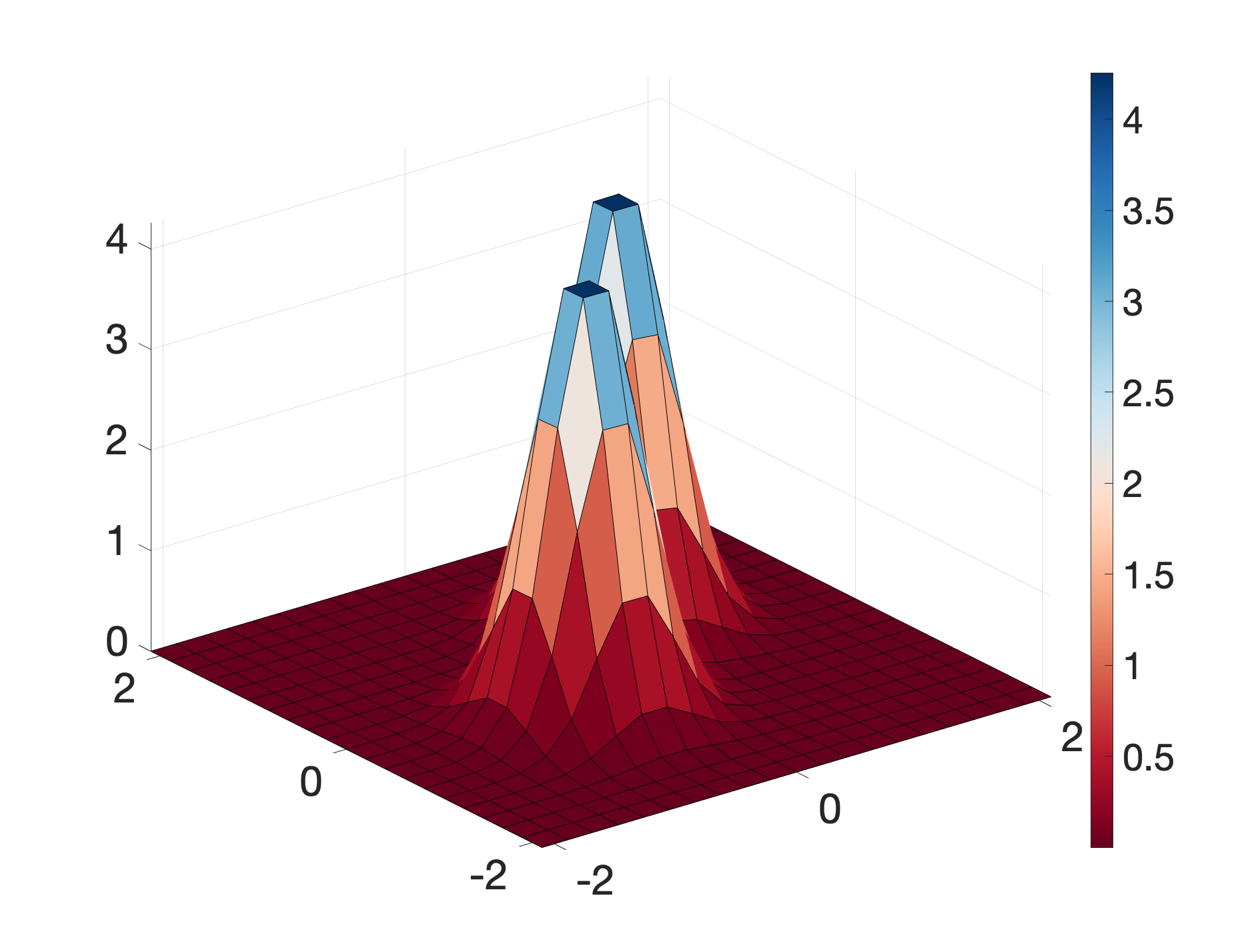}
\end{minipage} 
(b)\begin{minipage}{.45\textwidth}  \centering

\includegraphics[width=0.9\textwidth,height=0.6\textwidth]{ 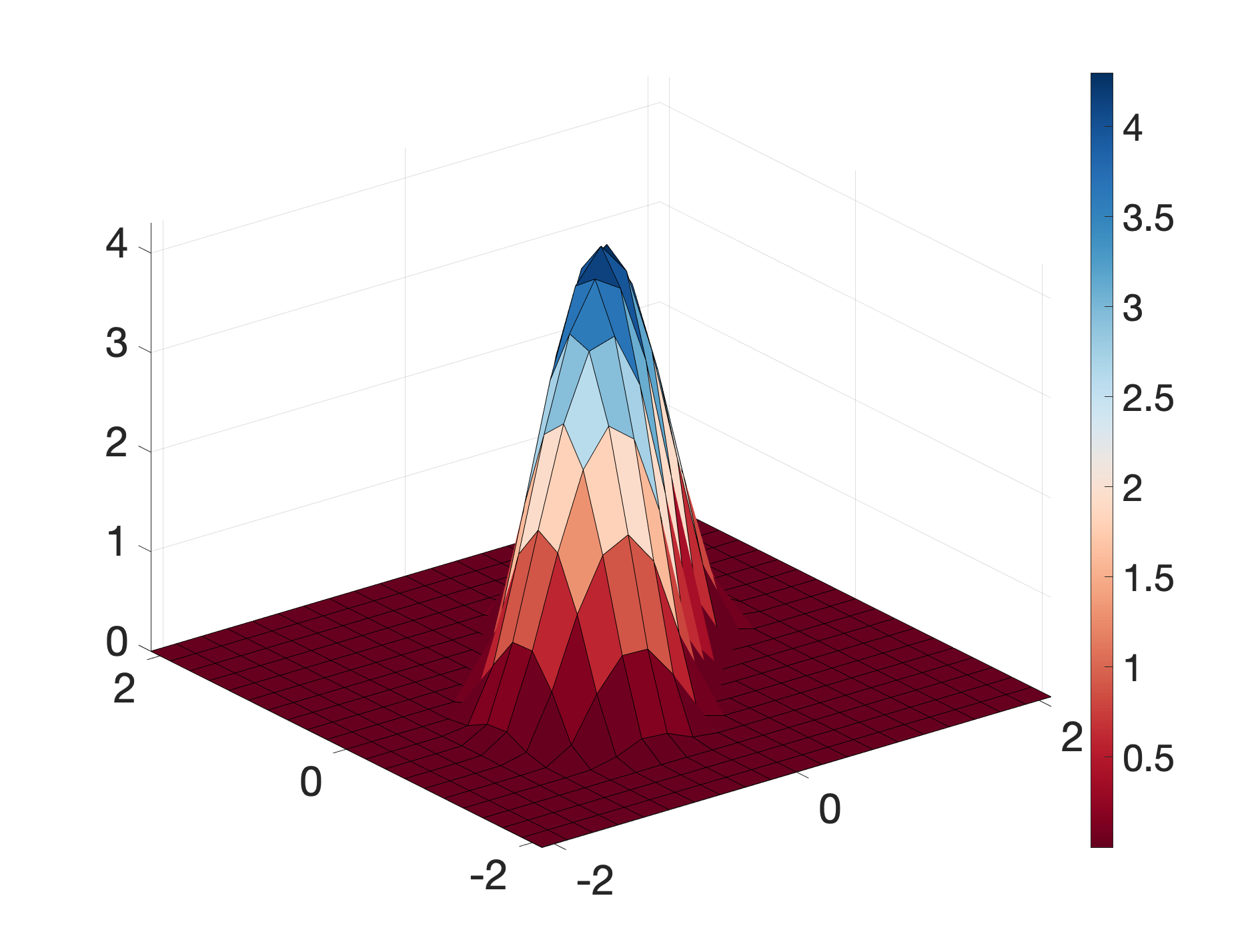}
\end{minipage}

(c)\begin{minipage}{.45\textwidth} \centering
\includegraphics[width=0.9\textwidth,height=0.6\textwidth]{ 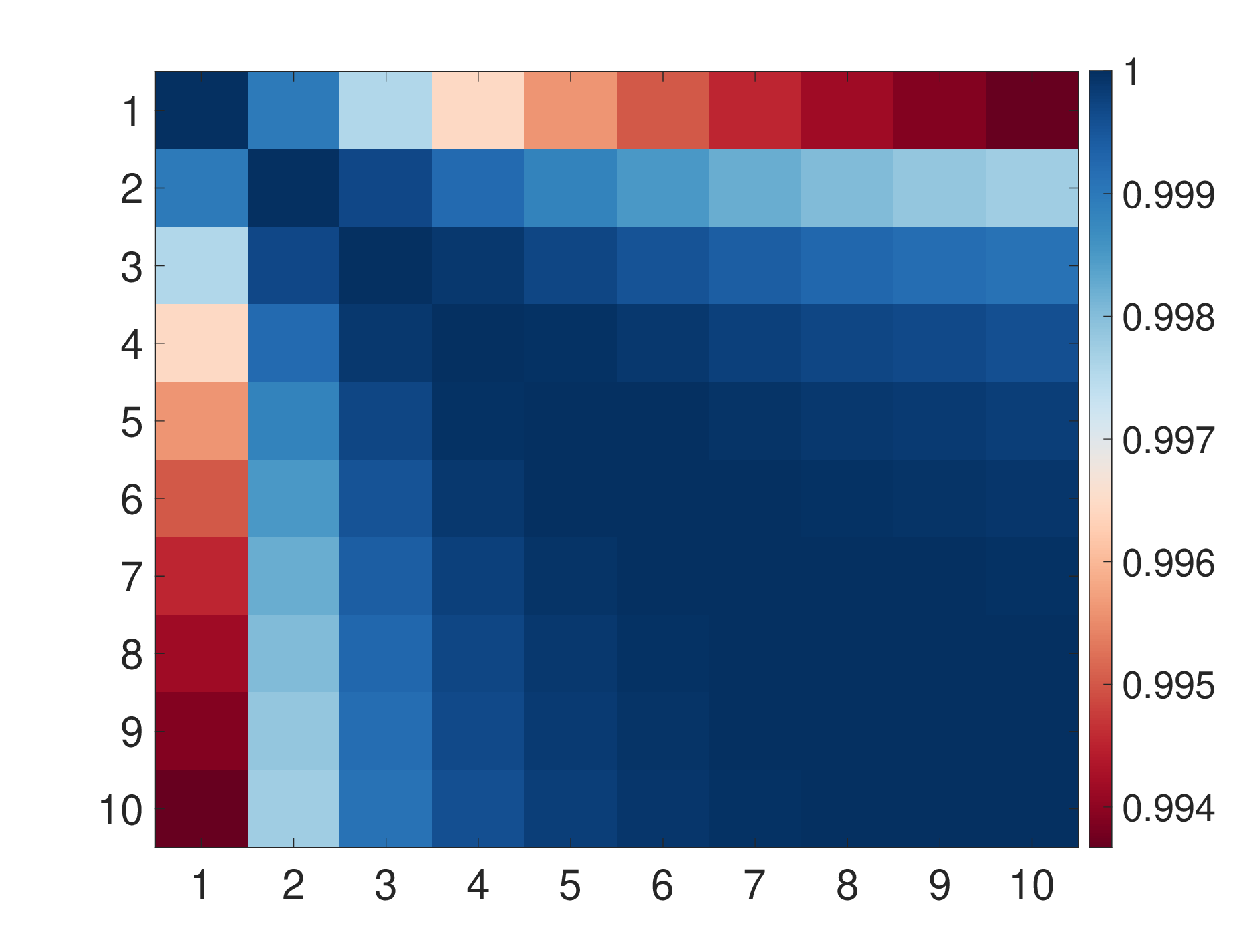}
\end{minipage}
(d)\begin{minipage}{.45\textwidth}  \centering
\includegraphics[width=0.81\textwidth,height=0.54\textwidth]{ 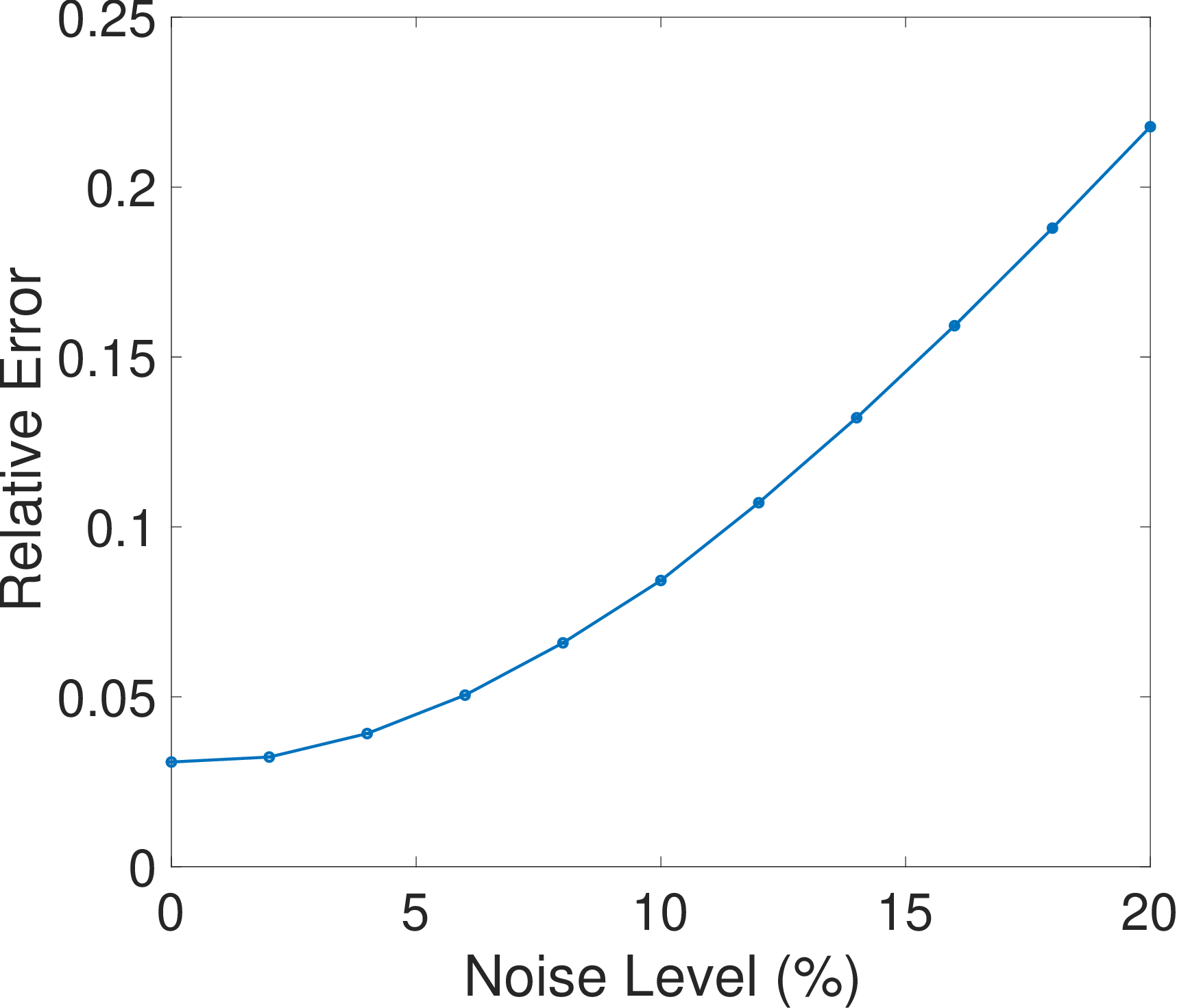}
\end{minipage}
\caption{ \textbf{Top Panel:} (a)-(b): the profile of the $t=0$ and $t=0.05$. \textbf{Bottom Panel:}  (c) The coherence pattern of the regression matrix $\mathbf{A}_{n,M,L}$. (d) We run PartInv with $K=2$ and used the support pruning algorithm for kernel estimation.  { For each noise level, we run 100 trials and display the relative reconstruction error bar  versus  noise levels but variances are very small in this set of experiments and therefore are not visible in the current scale}}.  \label{fig:2dmeta}
\end{figure}

\section{Conclusion and future work}\label{sec:conclusion_and_future_work}

In this study we present a new sparse identification algorithm designed to estimate the nonlocal interaction kernel within a broad spectrum of nonlocal gradient flow equations using noisy and discrete data. We establish new stability estimates that demonstrate the ability of our learned estimator to accurately reflect the training data. Additionally, we conduct an error analysis of our estimators and elucidate the dependency of their accuracy on factors such as the noise level and the mesh discretization. When compared to alternative sparse regression algorithms, our PartInv algorithm stands out for its simplicity in implementation and hyperparameter tuning while effectively addressing coherent regression matrices. It surpasses other methods like LASSO, subspace pursuit, and SINDy in performance. The main constraint is the prerequisite of selecting a suitable basis that ensures exact sparsity in the interaction kernel.

Future work will delve into the development of robust techniques tailored to solution data for aggregation diffusion equations.
This will include the exploration of advanced denoising techniques and regularization algorithms, such as sparse Bayesian methods, aimed at reducing the dependence of the prior knowledge on the estimated kernels.
Another direction of future work is to extend the current algorithm to cover systems with multiple interaction kernels, which models the heterogeneous interactions in multi-species systems.

\appendix
\section{Proofs of Dobrushin-type stability estimates} 
\subsection{Proof of Proposition \ref{prop:Dobrushin_}} \label{subsec:proof_stability_agg_eq}

\begin{proof}
Recalling standard results \cite{ambrosio2005gradient}, it is known that, given our assumptions on $W$, $\hatw, V$ and $\widehat{V}$,  the solutions of \eqref{eq:aggregation_eq} are of the form $\mu_t = \Phi_t \# \mu_0, \ \hatmu_t = \hatphi_t\#\hatmu_0$, where $\Phi_t$, $\hatphi_t$ are the flow maps induced by the velocity fields $\nabla W* \mu_t+\nabla V$ and $\nabla \hatw* \hatmu_t+\nabla \widehat{V}$, respectively. 
 Then we have the following estimate
\begin{align}\label{eq:W2_stability_aggregation}
d_2^2(\mu_t,\widehat{\mu}_t)& =  d^2_2( \Phi_t \#\mu_0,\widehat{\Phi}_t\#\widehat{\mu}_0)  \leq d^2_2( \Phi_t \#\mu_0,\widehat{\Phi}_t\#{{\mu}}_0) + d^2_2( \widehat{\Phi}_t \#{\mu}_0,\widehat{\Phi}_t\#\widehat{\mu}_0)\nonumber 
\\
& \leq     \int_{\R^d}|\Phi_t(\bx) - \widehat{\Phi}_t(\bx)|^2 \dd{\mu}_0(\bx)  + d^2_2( \widehat{\Phi}_t \#{\mu}_0,\widehat{\Phi}_t\#\widehat{\mu}_0)\ .
\end{align}
We can bound the final term above in the following way. Denote the product measure $\Pi_t := (\widehat{\Phi}_t \times \widehat{\Phi}_t) \# \Pi_0$, where $\Pi_0$ is the optimal transport plan between $\mu_0$ and $\widehat{\mu}_0$. Then, by definition of the 2-Wasserstein metric we have that 
\begin{align*}
   d^2_2( \widehat{\Phi}_t \#{\mu}_0,\widehat{\Phi}_t\#\widehat{\mu}_0) &\leq \int_{\R^d \times \R^d}|\bx-\by|^2 \dd\Pi_t = \int_{\R^d \times \R^d}| \widehat{\Phi}_t(\bx) - \widehat{\Phi}_t(\by)|^2 \dd \Pi_0
    \\
    & \leq e^{2(L_W+L_V)t}\int_{\R^d \times \R^d}|\bx-\by|^2 \dd\Pi_0
    \leq e^{2(L_W+L_V)t}d^2_2(\mu_0,\widehat{\mu}_0)\ .
\end{align*}
Where in the second line we used the Lipschitzness of the flow map $\widehat{\Phi}_t$. We then have the following estimate for the integrand in the first term of \eqref{eq:W2_stability_aggregation}
\begin{align*}
  |\Phi_t(\bx) - \widehat{\Phi}_t(\bx)|^2  \leq & \,  t \int_0^t |(\nabla W * \mu_s)(\Phi_s(\bx)) + \nabla V(\Phi_s(\bx)) 
  \\
  &   -(\nabla \hatw * \widehat{\mu}_s)(\widehat{\Phi}_s(\bx))- \nabla \widehat{V}(\widehat{\Phi}_s(\bx))|^2 \dd s\ .
  \end{align*}
  { After adding and subtracting $\nabla \hatw *\mu(\Phi(\bx))$ and $\nabla \widehat{V}(\Phi(\bx))$, we obtain}
  \begin{align*}
    |\Phi_t(\bx) - \widehat{\Phi}_t(\bx)|^2   \leq & \,4t \int_0^t |(\nabla W * {\mu}_s)(\Phi_s(\bx))- (\nabla {\hatw} * {\mu}_s)({\Phi}_s(\bx))|^2 \dd s 
  \\
 & +  4t \int_0^t |(\nabla \hatw * {\mu}_s)({\Phi}_s(\bx)) - (\nabla \hatw * \widehat{\mu}_s)(\widehat{\Phi}_s(\bx))|^2 \dd s
  \\
  &  +4t\int_0^t|\nabla V(\Phi_s(\bx)) -\nabla \widehat{V}({\Phi}_s(\bx))|^2\dd s
  \\
  &  + 4t \int_0^t|\nabla \widehat{V}({\Phi}_s(\bx)) - \nabla \widehat{V}(\widehat{\Phi}_s(\bx))|^2\dd s .
  \end{align*}
  { Adding and subtracting $\nabla \hatw * \hatmu(\Phi_s(\bx))$ and using the Lipschitzness of $\widehat{V}$, we deduce}
  \begin{align*}
  |\Phi_t(\bx) - \widehat{\Phi}_t(\bx)|^2  \leq  & \,4t \int_0^t \left|\int_{\R^d}\left[\nabla{W}({{\Phi_s}}(\bx)-\by) - \nabla \hatw({\Phi}_s(\bx)-\by)\right] \dd{\mu}_s(\by)\right|^2 \dd s
  \\
 &   + 8t \int_0^t \left|\int_{\R^d}\nabla  {\hatw}({\Phi}_s(\bx)-\by) \dd{\mu}_s(\by) - \nabla  {\hatw}({\Phi}_s(\bx)-\by)\dd{\widehat{\mu}_s(\by)}\right|^2 \dd s
 \\
 &  + 8t \int_0^t \left|\int_{\R^d}\nabla  {\hatw}({\Phi}_s(\bx)-\by)\dd{\widehat{\mu}}_s(\by) - \nabla  {\hatw}(\widehat{\Phi}_s(\bx)-\by)\dd{\widehat{\mu}_s(\by)}\right|^2 \dd s
 \\
 & +4t \int_0^t|\nabla V({\Phi}_s(\bx)) - \nabla \widehat{V}({\Phi}_s(\bx))|^2\dd s
 \\
 & +4tL^2_{\widehat{V}}\int_0^t|{\Phi}_s(\bx) - \widehat{\Phi}_s(\bx)|^2\dd s\ .
 \end{align*}
{ Rewriting these inequalities by factoring out certain terms, we infer that}
\begin{align*}
  \leq & \, 4t \int_0^t \left|\int_{\R^d}\left[\nabla{W}({{\Phi_s}}(\bx)-\by) - \nabla \hatw({\Phi}_s(\bx)-\by)\right] \dd{\mu}_s(\by)\right|^2 \dd s  
 \\
 & + 8t \int_0^t \int_{\R^d}\left|\nabla  {\hatw}({\Phi_s}(\bx)-\by)( \dd{\mu}_s(\by) - \dd \widehat{\mu}_s(\by))\right|^2 \dd s 
 \\
 &   +8t \int_0^t \left|\int_{\R^d}\nabla  {\hatw}({\Phi}_s(\bx)-\by)- \nabla  {\hatw}(\widehat{\Phi}_s(\bx)-\by)\dd{\widehat{\mu}_s(\by)}\right|^2 \dd s
  \\
 &  +4t \int_0^t|\nabla V({\Phi}_s(\bx)) - \nabla \widehat{V}({\Phi}_s(\bx))|^2\dd s
 \\
 &  +4tL^2_{\widehat{V}}\int_0^t|{\Phi}_s(\bx) - \widehat{\Phi}_s(\bx)|^2\dd s
 \\
 \leq & \,  4t \int_0^t \left|\int_{\R^d}\left[\nabla{W}({{\Phi_s}}(\bx)-\by) - \nabla \hatw({\Phi}_s(\bx)-\by)\right] \dd{\mu}_s(\by)\right|^2 \dd s  
 \\
 &  + 8t L^2_{\hatw} \int_0^t e^{2(L_W + L_V)s} d^2_2(\mu_s,\hatmu_s) \dd s 
 \\
 &  +8t L^2_{\hatw} \int_0^t \left|{\Phi}_s(\bx)- \widehat{\Phi}_s(\bx)\right|^2 \dd s
  \\
 &  +4t \int_0^t|\nabla V({\Phi}_s(\bx)) - \nabla \widehat{V}({\Phi}_s(\bx))|^2\dd s
 \\
 &  +4tL^2_{\widehat{V}}\int_0^t|{\Phi}_s(\bx) - \widehat{\Phi}_s(\bx)|^2\dd s\ ,
\end{align*}
{ where in the final inequality we used the Lipschitzness of $\nabla \hatw$ and $\Phi$ together with the definition by duality of the $1$-Wasserstein distance to obtain the second term; the third term also follows from the Lipschitzness of $\nabla \hatw$}. Then, integrating with respect to $\mu_0$ yields
\begin{align*}
    \int_{\R^d}|\Phi_t(\bx) - \widehat{\Phi}_t(\bx)|^2 \dd{\mu}_0(\bx)  \leq &\,  4t \int_0^t \int_{\R^d} \left|\int_{\R^d}\left[\nabla{W}({{\Phi_s}}(\bx)-\by) - \nabla \hatw({\Phi}_s(\bx)-\by)\right] \dd{\mu}_s(\by)\right|^2\dd \mu_0(\bx) \dd s
    \\
    &    + 8t L^2_{\hatw} \int_0^t e^{2(L_W + L_V)s} d^2_2(\mu_s,\hatmu_s) \dd s 
    \\ 
    &  +8t L^2_{\hatw} \int_0^t \int_{\R^d}\left|{\Phi}_s(\bx)- \widehat{\Phi}_s(\bx)\right|^2 \dd\mu_0(\bx)\dd s 
    \\
 &   +4t \int_0^t\int_{\R^d}|\nabla V({\Phi}_s(\bx)) - \nabla \widehat{V}({\Phi}_s(\bx))|^2\dd\mu_0(\bx)\dd s
 \\
 &  +4tL^2_{\widehat{V}}\int_0^t\int_{\R^d}|{\Phi}_s(\bx) - \widehat{\Phi}_s(\bx)|^2\dd\mu_0(\bx)\dd s
          \\
    = &\, 4t \int_0^t \int_{\R^d} \left|\int_{\R^d}\left[\nabla{W}(\bx-\by) - \nabla \hatw(\bx-\by)\right] \dd{\mu}_s(\by)\right|^2\dd \mu_s(\bx) \dd s
    \\
    &  + 8t L^2_{\hatw} \int_0^t e^{2(L_W + L_V)s} d^2_2(\mu_s,\hatmu_s) \dd s 
 \\
 &   +4t(2L^2_{\hatw}+L^2_{\widehat{V}})\int_0^t\int_{\R^d}|{\Phi}_s(\bx) - \widehat{\Phi}_s(\bx)|^2\dd \mu_0(\bx)\dd s \
\\
&   +4t \int_0^t\int_{\R^d}|\nabla V({\Phi}_s(\bx)) - \nabla \widehat{V}({\Phi}_s(\bx))|^2\dd\mu_0(\bx)\dd s.
\end{align*}
Then, an application of Gr\"onwall's inequality yields 
\begin{align*}
        \int_{\R^d}|\Phi_t(\bx) - \widehat{\Phi}_t(\bx)|^2 \dd\mu_0(\bx)  \leq & \,  \bigg(8t L^2_{\hatw} \int_0^t e^{2(L_W + L_V)s} d^2_2(\mu_s,\hatmu_s) \dd s
        \\
        & + 8t\int_0^t\|\nabla W*\mu_s - \nabla \hatw*\mu_s \|^2  _{L^2({\mu}_s)} \dd s
        \\
        &  + 4t\int_0^t \norm{\nabla V - \nabla \widehat{V}}^2_{L^2({\mu}_s)}\dd s \bigg)e^{2(2L^2_{\hatw}+L^2_{\widehat{V}})t^2}\ .
\end{align*}
Then going back to our original estimate in \eqref{eq:W2_stability_aggregation} we have after an additional application of Gr\"onwall's lemma that  
\begin{align*}
    d_2^2(\mu_t,\widehat{\mu}_t)\leq C_1\int_0^t\|\nabla W*\mu_s - \nabla \hatw*\mu_s \|^2  _{L^2({\mu}_s)}\dd s 
    + C_2\int_0^t \norm{\nabla V - \nabla \widehat{V}}^2_{L^2({\mu}_s)}\dd s+ C_3 d_2^2(\mu_0, \widehat{\mu}_0)
\end{align*}
where $C_1,C_2$ and $C_3$ are non-negative constants that depend on $T,L_W, L_V, L_{\widehat{W}}$ and $L_{\widehat{V}}$. Recalling the definition of the error functional $\mathcal{\tilde{E}}_\infty$ concludes the proof.
\end{proof}

\subsection{Assumptions and proof of Proposition \ref{prop:agg_diff_estimate}} \label{subsec:stability_estimate_agg_diff_eq}

Equation \eqref{eq:true_non_linear} can be interpreted as the evolution of the law 
of the solution of a stochastic differential equation (SDE) \cite{carmona2016lectures,chapter1991sznitman}. Namely, let $(\Omega, \mathcal{F}, (\mathcal{F}_t)_{t\in[0,T]}, \mathbb{P})$ be a filtered probability space and let $(B_t)_{t \in[0,T]}$ be an adapted Brownian motion in $\R^{d}$. Let us denote by 
\[
\mathbb{H}^{2, k}:=\left\{Z:[0,T]\times \R^d \to \R^{k}  \ |\ Z \text{ is progressively measurable, } \mathbb{E} \int_0^T\left|Z_s\right|^2 d s<\infty\right\} \ .
\]
Then, \eqref{eq:true_agg_diff} can be interpreted as the evolution of the law of the solution of the following SDE \cite{carmona2016lectures}
\begin{align}\label{eq:true_non_linear}
    d X_t &=\nabla W * \mu_t(X_t) d t + \sqrt{2}\sigma(K*\mu_t(X_t)) d B_t \ , \nonumber
    \\
    X_0 &= X^0 \in L^2 \text{ independent of } (B_t)_{t \in[0,T]}\ ,
\end{align}
where $\mu_t = \mathcal{L}(X_t)$ denotes the law of solution of \eqref{eq:true_non_linear} at time $t \in [0,T]$. We will need the following assumptions
\begin{assumption}\label{assumptions:aggregation-diffusion}
$\ $
\begin{enumerate}
    \item $W \in \W^{2,\infty}(\R^d)$, $\sigma:\R^n \to \R^{d}\times \R^d$ and $K:\R^d\to \R^n$ are Lipschitz and bounded. \label{ass:Lipschitz} 
    \item For any $\nu \in \P(\R^d)$, $(\nabla W *\nu (X_t))_{t \in [0,T]}\in\mathbb{H}^{2, d} $ and $(\sigma(K * \nu(X_t)))_{t \in [0,T]}\in\mathbb{H}^{2, d\times d}$. \label{ass:progressively_measurable}
\end{enumerate}
\end{assumption}

With these assumptions we are now ready to prove Proposition \ref{prop:agg_diff_estimate}. We note that for ease of notation, throughout the proof we use subscripts to denote the time argument of functions that depend on time. 
\begin{proof}
We begin by noting that by the definition of the 2-Wasserstein distance we have 
\begin{equation}\label{eq:agg_diff_first_estimate}
   d^2_2(\mu_t,\widehat{\mu}_t) \leq \mathbb{E} \sup_{s \in[0, t]} \left|X_s-\widehat{X}_s\right|^2\ . 
\end{equation}
Next, we consider the following estimate
\begin{align*}
 \mathbb{E} \sup_{s \in[0, t]} \left|X_s-\widehat{X}_s\right|^2= \, & 3\E\sup_{s \in[0, t]} \Bigg(\left| {X_0}-\widehat{X_0}\right|^2 
\\
& + \left|\int_0^s\left(\nabla W*\mu_r (X_r)-\nabla \widehat{W} * \widehat{\mu}_r(\widehat{X_r})\right) d r\right|^2 
\\
& + \left|\int_0^s\sigma(K * {\mu}_r({X_r}))-\widehat{\sigma}( \widehat{K}* \widehat{\mu}_r(\widehat{X}_r)) d {B}_r\right|^2\Bigg)
\\ 
& =: I+II+III
\end{align*}
We then have
\begin{align*}
    II\leq & \,  3  \mathbb{E}\sup_{s \in [0,t]} \ s \int_0^s\left|\nabla W* \mu_r(X_r)-\nabla \widehat{W} * \widehat{\mu}_r(\widehat{X_r})\right|^2  \dd r
    \\
     \leq &\, 9T \mathbb{E}\sup_{s \in [0,t]}\bigg(\int_0^s\left|\nabla W * \mu_r(X_r)-\nabla {\hatw} * {\mu}_r(X_r)\right|^2   \dd r
    \\
    &  +   \int_0^s\left|\nabla {\hatw} * {\mu}_r(X_r)-\nabla {\hatw} * \widehat{\mu}_r({X_r})\right|^2   \dd r
    \\
    &  +\int_0^s\left|\nabla {\hatw} * \widehat{\mu}_r({X_r})-\nabla \widehat{W} * \widehat{\mu}_r(\widehat{X_r})\right|^2  \dd r\bigg)
    \\
    \leq &\, 9T \int_0^t \|\nabla W*\mu_s - \nabla \widehat{W}*\mu_s\|^2_{L^2({\mu}_s)}\dd s 
    \\
    &  + 9T L_{\hatw}^2 \int_0^td_2^2(\mu_s,\widehat{\mu}_s)\dd s + 9TL_{\hatw}^2  \int_0^t \mathbb{E} \sup_{r \in [0,s]}|X_r - \widehat{X_r}|^2\dd s\ .
\end{align*}
For the third term, using the Burkholder-Davis-Gundy inequality we have
\begin{align*}
    III  = \, & 6 \E  \sup_{s \in[0,t]} \left| \int_0^s \sigma(K*{\mu}_r({X}_r)) - \widehat{\sigma}( \widehat{K}*{\widehat{\mu}}_r(\widehat{X}_r)) \dd B_r\right|^2 
    \\
     \leq &\, 6 C T\E\int_0^t\  \left| \sigma(K*{\mu}_s({X}_s)) - \widehat{\sigma}( \widehat{K}*{\widehat{\mu}}_s(\widehat{X}_s)) \right|^2\dd s
    \\
      \leq  & \,24 C T\E\int_0^t\  \left| \sigma(K*{\mu}_s({X}_s)) - {\sigma}( \widehat{K}*{{\mu}}_s({X_s})) \right|^2\dd s 
    \\
    &  + 24 C T\E\int_0^t\  \left| \sigma(\widehat{K}*{\mu}_s({X}_s)) - {\sigma}( \widehat{K}*{\widehat{\mu}}_s({X}_s)) \right|^2\dd s   \\
    &  + 24 C T\E\int_0^t\  \left| \sigma(\widehat{K}*\widehat{\mu}_s({X}_s)) - {\sigma}( \widehat{K}*{\widehat{\mu}}_s(\widehat{X}_s)) \right|^2\dd s
    \\
    &  + 24 C T\E\int_0^t\  \left| {\sigma}(\widehat{K}*\widehat{\mu}_s(\widehat{X}_s)) - \widehat{\sigma}( \widehat{K}*{\widehat{\mu}}_s(\widehat{X_s})) \right|^2\dd s 
    \\
     =: \ & (i) + (ii) + (iii) + (iv)   \ .
\end{align*}
Next, we have the following estimates for $(i), (ii), (iii)$ and $(iv)$.
\begin{align*}
 (i) &\leq  24 C T L^2_{{\sigma}} \int_0^t \norm{K*\mu_s - \widehat{K}*\mu_s}^2_{L^2({\mu}_s)} \dd s \ . 
\end{align*}
For $(ii)$ we have
\begin{align*}
(ii) &\leq   24 C T L^2_\sigma \E \int_0^t\left|\int_{\R^d}\widehat{K}(X_s-\by)\dd \mu(\by) -\int_{\R^d}\widehat{K}(X_s-\mathbf{z}) \dd \widehat{\mu}(\mathbf{z}) \right|^2 \dd s
 \\
 & \leq 24 C T L^2_\sigma L_{\widehat{K}}^2\int_0^t d^2_2(\mu_s,\widehat{\mu}_s) \dd s \ . 
   \end{align*}
Term $(iii)$ can be bounded by 
\begin{align*}
 (iii) \leq 24 C T L_\sigma^2 L_{\widehat{K}}^2  \int_0^t \sup_{r \in [0,s]} \E| X_r - \widehat{X}_r|^2 \dd s    \ ,
\end{align*} 
and finally, we have the following bound for term $(iv)$ 
\begin{align*}
    (iv) \leq 24  C T^2 \norm{\sigma - \widehat{\sigma}}^2_\infty .
\end{align*}
Putting all the previous estimates together we get 
\begin{align*}
\mathbb{E} \sup_{s \in[0, t]}\left|X_s-\widehat{{{X}}}_s\right|^2 \leq \, & 3 \mathbb{E}\left|X_0-\widehat{X}_0\right|^2
 + 9T \int_0^t \|\nabla W*\mu_s - \nabla \widehat{W}*\mu_s\|^2_{L^2({\mu}_s)}\dd s 
    \\
&  + 9T L_{\hatw}^2 \int_0^td_2^2(\mu_s,\widehat{\mu}_s)\dd s 
 +9TL_{\hatw}^2  \int_0^t \mathbb{E} \sup_{r \in [0,s]}|X_r - \widehat{X_r}|^2\dd s
\\
&  +24 C T L^2_{{\sigma}} \int_0^t \norm{K*\mu_s - \widehat{K}*\mu_s}^2_{L^2({\mu}_s)} \dd s
\\
&  + 24 C T L^2_\sigma L_{\widehat{K}}^2\int_0^t d^2_2(\mu_s,\widehat{\mu}_s) \dd s + 24  C T^2 \norm{\sigma - \widehat{\sigma}}^2_\infty
\\
&  +24 C T L_\sigma^2 L_{\widehat{K}}^2  \int_0^t \sup_{r \in [0,s]} \E| X_r - \widehat{X}_r|^2 \dd s \,. 
\end{align*}
A first application of Gr\"onwall's inequality then yields
\begin{align*}
    \mathbb{E} \sup_{s \in[0, t]}\left|X_s-\widehat{{{X}}}_s\right|^2 \leq &\,\bigg(3 \mathbb{E}\left|X_0-\widehat{X}_0\right|^2+\left(9L_{\hatw}^2 T +T24 CL_\sigma^2 L^2_{\widehat{K}}\right) \int_0^t  d_2^2\left(\widehat{\mu}_s, \mu_s\right) \dd s
    \\
    & +9 T \int_0^t\|\nabla W*\mu_s-\nabla \widehat{W}*\mu_s\|^2_{L^2({\mu}_s)} \dd s +24  CT^2\norm{\sigma - \widehat{\sigma}}^2_\infty
    \\
    & + 24  C T L^2_{{\sigma}} \int_0^t \norm{K*\mu_s - \widehat{K}*\mu_s}^2_{L^2({\mu}_s)} \dd s\bigg)\exp\left(12 CL_\sigma^2L^2_{\widehat{K}}T^2 +\frac{9}{2}T^2L^2_{\hatw}\right).
\end{align*}
Then, using \eqref{eq:agg_diff_first_estimate} and a final application Gr\"onwall's inequality yields
\begin{align*}
    d^2_2(\mu_t,\widehat{\mu}_t) \leq &\,C(T,L_\sigma,L_{\widehat{K}},L_{\hatw})\bigg(3 d_2^2(\mu_0,\widehat{\mu}_0) + 9 T \int_0^t\|\nabla W*\mu_s-\nabla \widehat{W}*\mu_s\|^2_{L^2({\mu}_s)} \dd s 
    \\
    & +24  CT^2\norm{\sigma - \widehat{\sigma}}^2_\infty + 24  C T L^2_{{\sigma}} \int_0^t \norm{K*\mu_s - \widehat{K}*\mu_s}^2_{L^2({\mu}_s)} \dd s\bigg),
\end{align*}
where $C(T,L_\sigma,L_{\widehat{K}},L_{\hatw})>0$ is a constant depending on the Lipschitz coefficients of $\sigma, 
\widehat{K}$ and $
\hatw$, and $T$. Using the definition of the error functional $\mathcal{\tilde{E}}_\infty$ concludes the proof.
\end{proof}
\subsubsection{Assumptions and proof of Proposition \ref{prop:non_linear_stability}}\label{subsec:proof_stability_non_linear_diff}
Let us denote by $\Pi(\sigma,\nu) \in \P(\R^d\times \R^d)$ the set of transport plans between $\sigma \in \P(\R^d)$ and $\nu \in \P(\R^d)$ for the quadratic cost function. As it will become clear later on, we will need the following estimate on the energy $\mathcal{H}(\rho):= \int_{\R^d} H(\rho(\bx))\dd \bx$.
\begin{lemma}\label{lemma:internal_energy_estimate}
Let $H:[0,+\infty] \to \R$ be the internal energy density given by $H(z)=\kappa\frac{z^m}{m-1}$ where $m\neq 1$, $m\geq 1- \frac{1}{d}$ and $m > \frac{d}{d+2}$. Let $\rho,\widetilde{\rho}$ be two smooth solutions to \eqref{eq:non_linear_diffusion} and let $\gamma_0 \in \Pi(\rho_0,\widetilde{\rho}_0)$. Then, we have the following estimate 
\begin{equation}
     \int_{\RRd} (\nabla H'(\widetilde{\rho}(0,\by)) - \nabla H'(\rho(0,\bx))) \cdot (\by - \bx)\gamma_0(\dd \bx, \dd \by) \geq 0 \ .
\end{equation}
\end{lemma}
\begin{proof}
    Following \cite[Section 5.2]{otto2001geometry} (see also \cite[Proposition 3.38]{ambrosio2013user}), the functional $\mathcal{H}(\rho)$ is displacement convex. As a consequence, we deduce the following estimate for the difference of the internal energy for two smooth solutions $\rho, \widetilde{\rho}$ of \eqref{eq:non_linear_diffusion} 
    \begin{align}\label{eq:energy_sym_1}
  \mathcal{H}(\widetilde{\rho_0}) - \mathcal{H}({\rho}_0) \geq \int_{\RRd} \nabla H'(\rho(0,\bx))\cdot(\by-\bx)\gamma_0(\dd \bx, \dd \by) \ , 
\end{align}
where $\gamma_0$ is a transference plan between ${\rho}$ and $\widetilde{\rho}$. Notice this is nothing else than the characterization of convexity by supporting hyperplanes. By symmetry we also have 
\begin{equation}\label{eq:energy_sym_2}
        \mathcal{H}(\rho_0) - \mathcal{H}(\widetilde{\rho}_0) \geq - \int_{\RRd} \nabla H'(\widetilde{\rho}(0,\by))\cdot(\by-\bx)\gamma_0(\dd \bx, \dd \by) \ .
\end{equation}
Adding \eqref{eq:energy_sym_1} and \eqref{eq:energy_sym_2} yields 
\begin{equation*}\label{eq:internal_energy_estimate}
    \int_{\RRd} (\nabla H'(\widetilde{\rho}(0,\by))- \nabla H'(\rho(0,\bx))\cdot(\by-\bx) \gamma_0(\dd \bx, \dd \by) \geq 0 \ .
\end{equation*}
\end{proof}

Furthermore, we note that by our assumption, $\rho, \hatrho$ are smooth solutions to continuity equations and hence, by \cite[Proposition 8.1.8]{ambrosio2005gradient} they admit the representations $\rho_\tau = \Phi_\tau\#\rho_0$ and $\hatrho_\tau = \hatphi_\tau\#\hatrho_0$, where $\Phi$ is the flow map associated to the problem  
\begin{align*}
    \frac{d}{d\tau}r(\tau,x) &= v(\tau,r),
    \\
    r(0) & = x \in \R^d,
\end{align*}
and $\hatphi$ is the flow map for the analogous problem with the velocity field $\widehat{v}(\hatrho)= -\nabla(  H'(\hatrho) +\hatw*\hatrho +  \widehat{V})$. Let $\gamma_0\in \Pi(\rho_{0},\hatrho_{0})$  be an optimal transport plan between $\rho_{0}$ and $\hatrho_{0}$ and note that, for $\tau \in (0,T]$ by the representations of $\rho_\tau$ and $\hatrho_\tau$, we have that $\gamma_\tau = (\Phi_\tau \times \hatphi_\tau)\#\gamma_0$ is an admissible transport plan between $\rho_\tau,\hatrho_\tau$. 
Now we are ready to present the proof of Proposition \ref{prop:non_linear_stability}.
\begin{proof}
Closely following arguments similar to \cite[Section 5.2]{otto2001geometry} (see also \cite[Corollary 5.2.5]{santambrogio2015optimal}), we have that by the representation of the solution $\rho_t,\hatrho_t$ in terms of the associated flows $\Phi_t,\hatphi_t$ we can obtain the following estimate 
\begin{align*}
\frac{1}{t}\left(d_2^2(\rho_t, \hatrho_t)-d_2^2\left(\rho_0, \hatrho_0\right)\right)
& \leq \frac{1}{t}\left(\int_{\RRd}\left|\by-\bx\right|^2 \dd\gamma_t\left( \bx,  \by\right)-\int_{\RRd}\left|\by-\bx\right|^2 \dd \gamma_0\left( \bx,  \by\right)\right) \\
& =\int_{\RRd} \frac{1}{t}\left(\left|\hatphi\left(t, \by\right)-\Phi\left(t, \bx\right)\right|^2-\left|\by-\bx\right|^2\right) \dd\gamma_0\left(\bx,  \by\right)\,.
\end{align*}
Then, letting $t\to 0^+$ and by the definition of the flow maps $\Phi, \hatphi$ we obtain
\begin{align*}\label{eq:non_linear_initial_estimate}
 \frac{d^{+}}{d \tau}\bigg\rvert_{\tau=0} d_2^2\left(\rho_t, \hatrho_t\right)\leq 2 \int_{\R^d \times \R^d}\left(\widehat{v}\left(0, \by\right)- v\left(0, \bx\right)\right) \cdot\left(\by-\bx\right) \dd\gamma_0\left(\bx, \by \right), 
\end{align*}
Using our estimate from Lemma \ref{lemma:internal_energy_estimate} and integrating in time yields
\begin{align*}
    d^2_2(\rho_t,\hatrho_t) & \leq d^2_2(\rho_0, \hatrho_0) +  2 \int_0^t\int_{\RRd} (\widehat{v}(\widehat{\rho})(s,\by) - {v}(\rho)(s,\bx)) \cdot (\by-\bx) \dd\gamma_s( \bx,\by)\dd s
    \\
    & \leq  -  2 \int_0^t\int_{\RRd}(\nabla \widehat{V}(\by)- \nabla V(\bx))\cdot(\by-\bx) \dd\gamma_s( \bx,\by)\dd s
    \\
    & \quad -  2 \int_0^t\int_{\RRd}(\nabla \hatw*\hatrho_{s}(\by)- \nabla W*\rho_{s}(\bx))\cdot(\by-\bx) \dd\gamma_s( \bx, \by) \dd s
    \\
    & \quad + d^2_2(\rho_0, \hatrho_0)\ .
\end{align*}
%%%%%%%%%%%%%%%%%%%%%%%% SECOND BLOCK %%%%%%%%%%%%%%%%%%%%%%%%%
Taking absolute value on both sides of the previous estimate and an application of Young's inequality yields
\begin{align*}
    d^2_2(\rho_t,\hatrho_t)& \leq \int_0^t\int_{\RRd}|\nabla \widehat{V}(\by)- \nabla V(\bx)|^2\dd\gamma_s( \bx, \by) + \int_{\RRd}|\by-\bx|^2 \dd\gamma_s( \bx, \by) \dd s
    \\
    & \quad + \int_0^t\int_{\RRd}|(\nabla \hatw*\hatrho_{s}(\by)- \nabla W*\rho_{s}(\bx))|^2\dd\gamma_s( \bx, \by) +  \int_{\RRd}|\by-\bx|^2\dd\gamma_s( \bx, \by)\dd s 
    \\
    & \quad + d^2_2(\rho_0, \hatrho_0)
    \\
    & \leq \int_0^t\int_{\RRd}|\nabla \widehat{V}(\by)- \nabla V(\bx)|^2\dd\gamma_s( \bx, \by) \dd s
    \\
    &\quad + \int_0^t\int_{\RRd}|(\nabla \hatw*\hatrho_{s}(\by)- \nabla W*\rho_{s}(\bx))|^2\dd\gamma_s( \bx, \by)\dd s
    \\
    & \quad + 2 \int_0^t\int_{\RRd}|\by-\bx|^2\dd\gamma_s( \bx, \by)\dd s  + d^2_2(\rho_0, \hatrho_0)\ .
\end{align*}
With similar calculations to the ones in Proposition \ref{prop:Dobrushin_} we have that 
\begin{align*}
|\nabla \widehat{V}(\by)- \nabla V(\bx)|^2 &\leq 2|\nabla \widehat{V}(\by)- \nabla \widehat{V}(\bx)|^2 + 2 |\nabla \widehat{V}(\bx)- \nabla V(\bx)|^2
\\
&\leq 2L^2_{\widehat{V}}|\by-\bx|^2 + 2 |\nabla \widehat{V}(\bx)- \nabla V(\bx)|^2\ ,
\end{align*}
and integrating with respect  to $\gamma_s$ yields 
\[
\int_{\RRd}|\nabla \widehat{V}(\by)- \nabla V(\bx)|^2\dd\gamma_s( \bx, \by) \leq 2L^2_{\widehat{V}}\int_{\RRd}|\by-\bx|^2\dd\gamma_s( \bx,\by)  +  2 \norm{\nabla V - \nabla \widehat{V}}^2_{L^2(\rho_{s})}\ .
\]
Similarly, for the term involving the interaction potential we have
\begin{align*}
    |\nabla \hatw*\hatrho_{s}(\by)- \nabla W*\rho_{s}(\bx)|^2 &\leq 2|\nabla \hatw * \hatrho_{s}(\by) - \nabla \hatw * \rho_{s}(\bx)|^2 + 2 | \nabla \hatw*\rho_{s}(\bx) - \nabla W * \rho_{s}(\bx)|^2
    \\
    & \leq 2 L^2_{\hatw} |\by-\bx|^2 + 2 | \nabla \hatw*\rho_{s}(\bx) - \nabla W * \rho_{s}(\bx)|^2  \ .
\end{align*}
Integrating with respect to $\gamma_s$ then yields 
\begin{align*}
    \int_{\RRd}|\nabla \hatw*\hatrho_{s}(\by)- \nabla W*\rho_{s}(\bx)|^2 \dd\gamma_s( \bx, \by) &\leq 2L^2_{\hatw} \int_{\RRd}|\by-\bx|^2\dd\gamma_s( \bx, \by)
    \\
    & \qquad+ 2\norm{\nabla \hatw * \rho_s- \nabla W*\rho_s}_{L^2(\rho_{s})}  \ .
\end{align*}
Hence, all in all we have 
\begin{align*}
    d^2_2(\hatrho_t,\rho_t) & \leq d^2_2(\hatrho_0,\rho_0)+  2\int_0^t(1 +L^2_{\widehat{V}} +L^2_{\widehat{W}}) \int_{\RRd}|\by-\bx|^2\dd\gamma_s( \bx, \by)\dd s
    \\
    & \qquad + 2 T\mathcal{\tilde{E}}_\infty(\hatw) +  2\int_0^t \norm{\nabla V - \nabla \widehat{V}}^2_{L^2(\rho_{s})} \dd s \ . 
\end{align*}
Taking the infimum with respect to $\gamma_s \in \Pi(\hatrho_s,\rho_s)$ in the previous equation gives
\begin{align*}
    d^2_2(\hatrho_t,\rho_t) & \leq d^2_2(\hatrho_0,\rho_0)+  2\int_0^t(1 +L^2_{\widehat{V}} +L^2_{\widehat{W}}) d^2_2(\hatrho_s,\rho_s)\dd s
    \\
    & \qquad + 2 T\mathcal{\tilde{E}}_\infty( \hatw) +  2\int_0^t \norm{\nabla V - \nabla \widehat{V}}^2_{L^2(\rho_{s})} \dd s \ . 
\end{align*}
An application of Gr\"onwall's inequality then yields 
\begin{align*}
    d^2_2(\rho_t,\hatrho_t) \leq &\exp\{2(1 +L^2_{\widehat{V}} +L^2_{\widehat{W}}) t\}
    \\
   & \qquad \times \left(d^2_2(\hatrho_0,\rho_0)+ 2 T\mathcal{\tilde{E}}_\infty(\hatw) +  2\int_0^t \norm{\nabla V - \nabla \widehat{V}}^2_{L^2(\rho_{s})} \dd s\right) \ ,
\end{align*}
which concludes the proof. 
\end{proof}

 \section{Error estimate with noisy data}\label{app: error estimate} 

\paragraph{Proof of Proposition \ref{prop: error A and b noise}}
\begin{proof}
 We need to estimate $\mathcal{B}:=\mathbf{{A}}_{n,M,L}(i,j)-\mathbf{\widetilde{A}}_{n,M,L}(i,j)$. Note that $\mathcal{B}$ does not depend on the indices $(i,j) \in \{1,\ldots,n\}^2$ as our final estimate is independent of the indices considered. 
Then we have 
\begin{align*}
    \mathcal{B} = \frac{1}{T}\sum_{\ell=1,m=-M}^{L,M}[ (C_{n,M,L}^{i}\cdot C_{n,M,L}^{j})_m^{\ell}\rho_m^{\ell} - (\tilde{C}^i_{n,M,L}\cdot \tilde{C}^j_{n,M,L})^\ell_m\tilde{\rho}^\ell_m ]\Delta x \Delta t\,.
\end{align*}
Define $\mathcal{C}^i_m(g^{\ell}) = \sum_{k=-M}^M  (\nabla \Psi_i)_{m-k}g(t_\ell,x_k) \Delta x$ for any function $g$ defined on the mesh. Then we can write $\mathcal{B}$ using this notation to emphasize the dependence of $\tilde{C}^i_{n,M,L}$ on the noise as follows
\begin{align*}
    \mathcal{B} = \frac{1}{T}\sum_{\ell=1,m=-M}^{L,M}[  (C_{n,M,L}^{i}\cdot C_{n,M,L}^{j})_m^{\ell}\rho_m^{\ell} - \mathcal{C}^i_m({\rho}^\ell+\varepsilon^{\ell})\mathcal{C}^j_m({\rho}^\ell+\varepsilon^{\ell})({\rho}^\ell_m+\varepsilon^\ell_m)]\Delta x \Delta t\,.
\end{align*}
By linearity of $\mathcal{C}^i_m$, expanding the second term yields 
\begin{align*}
    |\mathcal{B}| = |\Lambda^{i,j}|\ ,
\end{align*}
where $\Lambda^{i,j}$ is composed of the terms in $\mathcal{B}$ that depend on the noise and is given by
\begin{align*}
    \Lambda^{i,j} =&  \  \frac{1}{T} \sum_{v=1}^7  \Lambda^{i,j}_v
    \\
    =& \ \frac{1}{T}\sum_{\ell=1,m=-M}^{L,M} \bigg(\mathcal{C}^i_m(\varepsilon^\ell)\mathcal{C}^j_m({\rho}^\ell){\rho}^\ell_m\nonumber + \mathcal{C}^i_m({\rho}^\ell)\mathcal{C}^j_m({\rho}^\ell)\varepsilon^\ell_m\nonumber
    \\
    & + \mathcal{C}^i_m(\varepsilon^\ell)\mathcal{C}^j_m(\varepsilon^\ell){\rho}^\ell_m + \mathcal{C}^i_m({\rho}^\ell)\mathcal{C}^j_m(\varepsilon^\ell){\rho}^\ell_m\nonumber
    \\
    & + \mathcal{C}^i_m(\varepsilon^\ell)\mathcal{C}^j_m({\rho}^\ell)\varepsilon^\ell_m\nonumber + \mathcal{C}^i_m({\rho}^\ell)\mathcal{C}^j_m(\varepsilon^\ell)\varepsilon^\ell_m\nonumber
    \\
    & + \mathcal{C}^i_m(\varepsilon^\ell)\mathcal{C}^j_m(\varepsilon^\ell)\varepsilon^\ell_m\nonumber \bigg) \Delta x \Delta t \ .
\end{align*}
For brevity, we only display the estimates for the terms yielding the error order reported in Proposition \ref{prop: error A and b noise} as the rest of the terms are of higher order.  Namely, we will consider the terms 
\begin{align*}
    \Lambda^{i,j}_1 = \, & \sum_{\ell=1,m=-M}^{L,M}\mathcal{C}^i_m(\varepsilon^\ell)\mathcal{C}^j_m({\rho}^\ell){\rho}^\ell_m \Delta x \Delta t\ ,
\\
    \Lambda^{i,j}_2 = \, & \sum_{\ell=1,m=-M}^{L,M}\mathcal{C}^i_m({\rho}^\ell)\mathcal{C}^j_m({\rho}^\ell)\varepsilon^\ell_m \Delta x \Delta t \ , 
\end{align*}
and
\begin{align*}
   { \Lambda^{i,j}_3= } \, & { \sum_{\ell=1,m=-M}^{L,M}\mathcal{C}^i_m(\varepsilon^\ell)\mathcal{C}^j_m(\varepsilon^\ell){\rho}^\ell_m\Delta x \Delta t \ .}
\end{align*}
Note that $\Lambda_4^{i,j}$ is analogous to $\Lambda^{i,j}_1$, so it will have the same error. We begin with the estimate for $\Lambda^{i,j}_1$ given by 
\begin{align*}
    \norm{\Lambda^{i,j}_1}_{L^2(\varepsilon)} = \, & \bigg( \E\bigg|\sum_{\ell=1,m=-M}^{L,M}\mathcal{C}^i_m(\varepsilon^\ell)\mathcal{C}^j_m({\rho}^\ell){\rho}^\ell_m \Delta x \Delta t\bigg|^2\bigg)^{\frac{1}{2}}\,.
\end{align*}
By expanding the square, using the triangle inequality and noting that the expectation of terms for different points in the time mesh $\ell_1 \neq \ell_2$ for $\ell_1, \ell_2 \in \{1, \ldots, L \}$ vanishes, we obtain
\begin{align*}
    \norm{\Lambda^{i,j}_1}_{L^2(\varepsilon)} \leq & \,\bigg( \sum_{\ell=1,m=-M}^{L,M}\E\bigg|(\mathcal{C}^i_m(\varepsilon^\ell))^2(\mathcal{C}^j_m({\rho}^\ell){\rho}^\ell_m)^2 \Delta x^2 \Delta t^2\bigg|\bigg)^{\frac{1}{2}}
   \\
   & 
   + \bigg( \sum_{\ell=1,m_1\neq m_2}^{L,M}\E\bigg|(\mathcal{C}^i_{m_1}(\varepsilon^\ell)\mathcal{C}^j_{m_1}({\rho}^\ell){\rho}^\ell_{m_1})(\mathcal{C}^i_{m_2}(\varepsilon^\ell)\mathcal{C}^j_{m_2}({\rho}^\ell){\rho}^\ell_{m_2}) \Delta x^2 \Delta t^2\bigg|\bigg)^{\frac{1}{2}}\ .
   \\
    = &\, (a) + (b) \ . 
\end{align*}
 Note that for any $m \in \{-M,\ldots,M\}$ and any $\ell \in \{1,\ldots,L\}$ we have the following estimate 
\begin{align} 
\label{eq: determinitic_intermediate_estimate}
    |(\mathcal{C}^i_m({\rho}^\ell))^2| = & \, \bigg|\sum_{k=-M}^M (\nabla\Psi_i)^2_{m-k}(\rho^{\ell}_k)^2 \Delta x ^2 + \sum_{k_1 \neq k_2}^M (\nabla \Psi_i)_{m-k_1}(\nabla \Psi_i)_{m-k_2}\rho^{\ell}_{k_1}\rho^{\ell}_{k_2} \Delta x ^2 \bigg| \leq C  \ , 
\end{align}
for some positive constant $C=C(\norm{\rho}_\infty, \norm{\nabla\Psi}_\infty)$. In the estimates that follow we will write $C$ to denote a generic positive constant which can depend on $R, T, \norm{\rho}_\infty, \norm{\Psi}_\infty, \norm{\nabla\Psi}_\infty$ and can change from line to line. Recalling that we defined $M = 2R/\Delta x$ and $L= T/ \Delta t$, we have 
\begin{align*}
    (a) \leq & \ C\left( \frac{2R}{\Delta x }\frac{T}{\Delta t} \left(\frac{2R}{\Delta x} \norm{\nabla\Psi_i}_\infty^2 \sigma^2 \Delta x^2\right)\Delta x ^2 \Delta t^2\right)^{1/2}
    \\
     \leq & \  C \sigma \Delta x \Delta t ^{1/2} \ .
\end{align*}
For $(b)$ we begin by noting that, similarly to the estimate \eqref{eq: determinitic_intermediate_estimate} we have $|(\mathcal{C}^j_{m_1}({\rho}^\ell))(\mathcal{C}^j_{m_2}({\rho}^\ell))| \leq C = C(\norm{\rho}_\infty, \norm{\nabla\Psi}_\infty)$ for any $m_1,m_2 \in \{-M, \ldots, M\}$ and $\ell \in \{1,\ldots,L\}$. Then we have that 
\begin{align*}
    (b)  \leq & \,  C \left(\frac{4R^2}{\Delta x ^2} \frac{T}{\Delta t}\left(\frac{2R}{\Delta x} \norm{\nabla\Psi_i}^2 \sigma^2 \Delta x^2\right)\Delta x ^2 \Delta t^2\right)^{1/2}
\\
 \leq &\ C\sigma \sqrt{\Delta x \Delta t}\ ,
\end{align*}
from which $\norm{\Lambda^{i,j}_1}_{L^2(\varepsilon)}\leq C\sigma \sqrt{\Delta x \Delta t}$ follows. The calculations for $\Lambda^{i,j}_2$ are analogous. The only term that survives after expanding the square, using the triangle inequality and taking expectation is
\begin{align*}
\norm{\Lambda_2^{i,j}}_{L^2(\varepsilon)} = \, & \left(\sum_{\ell=1,m=-M}^{L,M}\E|\mathcal{C}^i_m({\rho}^\ell)\mathcal{C}^j_m({\rho}^\ell)\varepsilon^\ell_m\nonumber|^2  \Delta x^2 \Delta t ^2 \right)^{1/2}
\end{align*}
Then we have the following estimate
\begin{align*}
\norm{\Lambda_2^{i,j}}_{L^2(\varepsilon)} & \leq  \bigg(C\frac{2R T}{\Delta x \Delta t} \sigma^2 \Delta x^2 \Delta t^2\bigg)^{1/2}
\\
& \leq \ C \sigma\sqrt{\Delta x \Delta t}\ .
\end{align*}
{ Finally, for term $\Lambda^{i,j}_3$ we have 
\begin{align*}
\norm{\Lambda^{i,j}}_{L^2(\varepsilon)} \leq \, &\left(\sum_{\ell=1,m=-M}^{L,M}\E|\mathcal{C}^i_m({\varepsilon}^\ell)\mathcal{C}^j_m({\varepsilon}^\ell)\rho^\ell_m\nonumber|^2  \Delta x^2 \Delta t ^2 \right)^{1/2}
\\
&\  + \bigg( \sum_{\ell_1\neq \ell_2=1,m=-M}^{L,M}\E\bigg|(\mathcal{C}^i_{m}(\varepsilon^{\ell_1})\mathcal{C}^j_{m}({\varepsilon}^{\ell_1}){\rho}^{\ell_1}_{m})(\mathcal{C}^i_{m}(\varepsilon^{\ell_2})\mathcal{C}^j_{m}({\varepsilon}^{\ell_2}){\rho}^{\ell_2}_{m}) \Delta x^2 \Delta t^2\bigg|\bigg)^{\frac{1}{2}}
\\
&\  + \bigg( \sum_{\ell=1,m_1 \neq m_2=-M}^{L,M}\E\bigg|(\mathcal{C}^i_{m_1}(\varepsilon^{\ell})\mathcal{C}^j_{m_1}({\varepsilon}^{\ell}){\rho}^{\ell}_{m_1})(\mathcal{C}^i_{m_2}(\varepsilon^{\ell})\mathcal{C}^j_{m_2}({\varepsilon}^{\ell}){\rho}^{\ell}_{m_2}) \Delta x^2 \Delta t^2\bigg|\bigg)^{\frac{1}{2}}
\\
&\  + \bigg( \sum_{\ell_1\neq \ell_2=1,m_1 \neq m_2=-M}^{L,M}\E\bigg|(\mathcal{C}^i_{m_1}(\varepsilon^{\ell_1})\mathcal{C}^j_{m_1}({\varepsilon}^{\ell_1}){\rho}^{\ell_1}_{m_1})(\mathcal{C}^i_{m_2}(\varepsilon^{\ell_2})\mathcal{C}^j_{m_2}({\varepsilon}^{\ell_2}){\rho}^{\ell_2}_{m_2}) \Delta x^2 \Delta t^2\bigg|\bigg)^{\frac{1}{2}}
\\
= & \ (i) + (ii) + (iii) + (iv)\ .
\end{align*}
The expectation in all terms can be bounded by 
\[
C \sigma^4 \Delta x^4 \Delta t^2\ .
\]
Then, the worst estimate comes from $(iv)$ which has a larger number of terms giving 
\begin{align*}
   \norm{\Lambda^{i,j}_3}_{L^2(\varepsilon)} & \leq \left(C \frac{T^2}{\Delta t ^2}\frac{4R^2}{\Delta x^2}\sigma^4 \Delta x^4 \Delta t ^2\right)^{1/2}  \ 
   \\
   & \leq C \sigma^2 \Delta x\ .
\end{align*}}

The estimates for the rest of the terms in $\Lambda^{i,j}$ are obtained in an analogous way and one can check that they are of higher order than the terms presented above. Hence we can conclude that for any $(i,j) \in \{1,\ldots,n\}^2${
\[
\norm{\mathbf{{A}}_{n,M,L}(i,j)-\mathbf{\widetilde{A}}_{n,M,L}(i,j)}_{L^2(\varepsilon)}\leq C (\sigma\sqrt{\Delta x \Delta t } + \sigma^2\Delta x) \ .
\]}
\end{proof}

\paragraph{Proof of \eqref{eq: bound in b noise}}
\begin{proof}
    The structure of the proof is analogous to the one of the proof of Proposition \ref{prop: error A and b noise}. We need to estimate $\mathcal{D} :=\mathbf{{b}}_{n,M,L}(i)-\mathbf{\widetilde{b}}_{n,M,L}(i) $. Note that $\mathcal{D}$ does not depend on the indices $i \in \{1, \ldots, n\}$ as the final estimate is again independent of these. Let us put $\mathcal{R}^i_m(g^\ell) = \sum_{k=-M}^M (\Psi_i)_{m-k}g(t_\ell,x_k)\Delta x$ for any function $g$ defined on the mesh. Then $\mathcal{D}$ reads 
  {  
    \begin{align*}
    \mathcal{D} =  &\ -\frac{1}{T} \sum_{\ell=1,m=-M}^{L,M}\bigg[\bigg( (\widehat {\partial_t}\rho R_{n,M,L}^{i})_m^{\ell}+ (C_{n,M,L}^{i}F_{M,L})_m^{\ell} \bigg)
    \\
    & - \bigg(\delta^+_t(\rho^\ell_m + \varepsilon^\ell_m) \mathcal{R}^i_m(\rho^\ell+\varepsilon^\ell)+(\rho^\ell_m+\varepsilon^\ell_m)\delta^+_x(\rho^\ell_m + \varepsilon^\ell_m)\mathcal{C}^i_m(\rho^\ell+\varepsilon^\ell)\bigg)\Delta x \Delta t\ .
    \end{align*}}
By linearity of $\mathcal{C}$ and $\mathcal{R}$, expanding the second term yields 
\begin{align*}
    |\mathcal{D}| \leq |\zeta^{i}_t| + |\zeta^i_x|\ ,
\end{align*}
where $\zeta^i_t$ is the sum of all the terms in the expansion of the second term that depend on the noise and its discrete time derivative and $\zeta^i_x$ is the sum of all the terms that depend on the noise and its discrete space derivative, i.e.
 {\begin{align*}
  {\zeta_t^i = }& \,{ \frac{1}{T} \sum_{v=1}^3 \zeta_{t,v}^i}
    \\
    = & {  \frac{1}{T} \sum_{\ell=1,m=-M}^{L,M} \bigg[\delta_t^+ \varepsilon^\ell_m\mathcal{R}^i_m(\rho^\ell) +\delta_t^+ \varepsilon^\ell_m\mathcal{R}^i_m(\varepsilon^\ell) + \delta_t^+ \rho^\ell_m\mathcal{R}^i_m(\varepsilon^\ell) \bigg]\Delta x \Delta t \ ,}
\end{align*}}
and 
\begin{align*}
{\zeta_x^i} = &{ \,\frac{1}{T} \sum_{v=1}^{7} \zeta_{x,v}^i}
    \\
    = &  { \frac{1}{T} \sum_{\ell=1,m=-M}^{L,M}\bigg[\rho^\ell_m \delta_x^+ \rho^\ell_m \mathcal{C}^i_m(\varepsilon^\ell) + \rho^\ell_m \delta_x^+ \varepsilon^\ell_m \mathcal{C}^i_m(\rho^\ell) +\rho^\ell_m \delta_x^+ \varepsilon^\ell_m\mathcal{C}^i_m(\varepsilon^\ell)} 
    \\
    & {+ \varepsilon^\ell_m\delta_x^+\rho^\ell_m \mathcal{C}^i_m(\rho^\ell) + \varepsilon^\ell_m\delta_x^+\rho^\ell_m \mathcal{C}^i_m(\varepsilon^\ell) +  \varepsilon^\ell_m\delta_x^+\varepsilon^\ell_m \mathcal{C}^i_m(\rho^\ell) + \varepsilon^\ell_m\delta_x^+\varepsilon^\ell_m \mathcal{C}^i_m(\varepsilon^\ell)\bigg] \Delta x \Delta t \ .}
\end{align*}
As before, for brevity, we only display the calculations for the terms yielding the error reported in Proposition \ref{eq: bound in b noise} as the rest of the terms are of higher order. Namely, we will consider the terms
\begin{align*}
   { \zeta^{i}_{t,2} = \,  \sum_{\ell=1,m=-M}^{L,M} \delta_t^+ \varepsilon^\ell_m\mathcal{R}^i_m(\varepsilon^\ell)\Delta x \Delta t \ }, 
\end{align*}
and
\begin{align*}
   { \zeta^{i}_{x,6} = \,  \sum_{\ell=1,m=-M}^{L,M}\varepsilon^\ell_m\delta_x^+\varepsilon^\ell_m \mathcal{C}^i_m(\rho^\ell) \Delta x \Delta t \ }. 
\end{align*}
Let us begin estimating $ \norm{\zeta^{i}_{t,2}}_{L^2(\varepsilon)}$
\begin{align*}
   { \norm{\zeta^{i}_{t,2}}_{L^2(\varepsilon)} =\bigg(\E\bigg|\sum_{\ell=1,m=-M}^{L,M} \delta_t^+ \varepsilon^\ell_m\mathcal{R}^i_m(\varepsilon^\ell)\Delta x \Delta t\bigg|^2\bigg)^{1/2} \ }. 
\end{align*}
Expanding the square and an application of the triangle inequality yields
\begin{align*}
   { \norm{\zeta^{i}_{t,2}}_{L^2(\varepsilon)}}  \leq & \, { \bigg(\sum_{\ell=1, m=-M}^{L,M}\E\bigg|\delta_t^+ \varepsilon^\ell_m\mathcal{R}^i_m(\varepsilon^\ell)\Delta x \Delta t  \bigg|^2\bigg)^{1/2}}
     \\
     & {+\bigg(\sum_{\ell_1\neq \ell_2=1, m=-M}^{L,M}\E  \delta_t^+ \varepsilon^{\ell_1}_m \mathcal{R}^i(\varepsilon^{\ell_1})  \delta_t^+ \varepsilon^{\ell_2}_m\mathcal{R}^i(\varepsilon^{\ell_2})\Delta x^2 \Delta t^2 \bigg)^{1/2}}
    \\
     &{ +\bigg(\sum_{\ell=1, m_1 \neq m_2=-M}^{L,M}\E  \delta_t^+ \varepsilon^{\ell}_{m_1} \mathcal{R}^i(\varepsilon^{\ell}) \delta_t^+ \varepsilon^{\ell}_{m_2}\mathcal{R}^i(\varepsilon^{\ell})\Delta x^2 \Delta t^2 \bigg)^{1/2}}
     \\
     &{  +\bigg(\sum_{\ell_1\neq\ell_2=1, m_1 \neq m_2=-M}^{L,M}\E  \delta_t^+ \varepsilon^{\ell_1}_{m_1}\mathcal{R}^i(\varepsilon^{\ell_1}) \delta_t^+\varepsilon_{m_2}^{\ell_2}\mathcal{R}^i(\varepsilon^{\ell_2})\Delta x^2 \Delta t^2 \bigg)^{1/2}}
     \\
     =& (i) + (ii) + (iii) + (iv)   \ .
\end{align*}
We begin by bounding $(i)$ 

\begin{align*}
    {  (i) }&{ \leq \left( \frac{T}{\Delta t}\frac{2R}{\Delta x} \frac{C}{\Delta t^2} \sigma^4 \Delta x^2\Delta t^2\right)^{1/2} = C \sigma^2 {\Delta x}^{1/2}\Delta t^{-1/2} } .
\end{align*}
For terms $(ii),(iii)$ and $(iv)$ the expectation can be bounded by
\begin{align*}
   {  C \frac{\sigma^4}{\Delta t ^2} \Delta x^4 \Delta t^2} . 
\end{align*}
Since the only difference among these terms is how many of them we need to consider, it follows that $(iv)$ produces the worst case yielding
\begin{align*}
    { (iv)}& \leq {  \left(\frac{4R^2}{\Delta x ^2} \frac{T^2}{\Delta t^2} C \frac{\sigma^4}{\Delta t ^2} \Delta x^4 \Delta t^2 \right)^{1/2}}
    \\ 
    & = {  C\sigma^2 \Delta x\Delta t^{-1} \ .}
\end{align*}
This is the worst error associated to the noise with respect to time. The estimate for $\norm{\zeta^i_{x,6}}_{L^2(\varepsilon)}$ can be obtained in an analogous way. Indeed, we have by expanding the square and an application of the triangle inequality that
\begin{align*}
    {  \norm{\zeta_{x,6}}_{L^2(\varepsilon)} \leq }& {  \ \sum_{\ell=1, m=-M}^{L,M}\E|\varepsilon^\ell_m \delta_x^+ \varepsilon^\ell_m \mathcal{C}^i_m(\rho^\ell) |^2\Delta x^2 \Delta t^2} 
     \\
     & { + \sum_{\ell_1 \neq \ell_2=1, m=-M}^{L,M}\E|\varepsilon^{\ell^1}_m\varepsilon^{\ell_2}_m\delta_x^+ \varepsilon^{\ell_1}_m \delta_x^+ \varepsilon^{\ell_2}_m\mathcal{C}^i_m(\rho^{\ell_2})\mathcal{C}^i_m(\rho^{\ell_1})|^2\Delta x^2 \Delta t^2 }
     \\
     &{  + \sum_{\ell=1, m_1 \neq m_2=-M}^{L,M}\E|\varepsilon^{\ell}_{m_1}\varepsilon^{\ell}_{m_2}\delta_x^+ \varepsilon^{\ell}_{m_1} \delta_x^+ \varepsilon^{\ell}_{m_2}\mathcal{C}^i_{m_1}(\rho^{\ell})\mathcal{C}^i_{m_2}(\rho^{\ell})|^2\Delta x^2 \Delta t^2 }
     \\
     & { + \sum_{\ell_1 \neq \ell_2=1, m_1 \neq m_2=-M}^{L,M}\E|\varepsilon^{\ell_1}_{m_1}\varepsilon^{\ell_2}_{m_2}\delta_x^+ \varepsilon^{\ell_1}_{m_1} \delta_x^+ \varepsilon^{\ell_2}_{m_2}\mathcal{C}^i_{m_1}(\rho^{\ell_1})\mathcal{C}^i_{m_2}(\rho^{\ell_2})|^2\Delta x^2 \Delta t^2  }
     \\
     & = (i) + (ii) + (iii) + (iv)
\end{align*}
Term $(i)$ can be bounded as
\[
{ \norm{(i)}_{L^2(\varepsilon)}\leq C\sigma^2 \Delta t^{1/2} \Delta x^{-1/2}\ .}
\]
Furthermore, the expectations in terms $(ii), (iii)$ and $(iv)$ can be bounded by
\[
{  C\sigma^4 \Delta t^2 \ .}
\]

As before, since the only difference among these terms is how many of them we need to consider, it follows that $(iv)$ produces the worst error yielding
\begin{align*}
(iv) &{ \leq \left(\frac{4R^2}{\Delta x^2}\frac{T^2}{\Delta t ^2} C \sigma^4 \Delta t^2\right)^{1/2}}
\\
& = {  C \sigma^2 \Delta x^{-1}}
\end{align*}
Thus, we can conclude that 
\[
{  \norm{\mathbf{{b}}_{n,M,L}(i)-\mathbf{\widetilde{b}}_{n,M,L}(i)}_{L^2(\varepsilon)} \leq C \sigma^2(\Delta x^{-1}+\Delta x\Delta t^{-1})}
\]
\end{proof}

\section*{Acknowledgments}
JAC and GER were supported by the Advanced Grant Nonlocal\--CPD (Non\-local PDEs for Complex Particle Dynamics: Phase Transitions, Patterns and Synchronization) of the European Research Council Executive Agency (ERC) under the European Union’s Horizon 2020 research and innovation programme (grant agreement No. 883363).
JAC was also partially supported by EPSRC grant numbers EP/T022132/1 and EP/V051121/1. GER acknowledges the support from the research group 2021 SGR 00087 and the project macroKNIGHTs (PID2022-143012NA-100) funded by the Spanish Ministry of Science and Innovation.
LM was supported by the EPSRC Centre for Doctoral Training in Mathematics of Random Systems: Analysis, Modelling and Simulation (EP/S023925/1).
S. Tang received partial support from the Hellman Faculty Fellowship and the Faculty Early Career Development Awards, funded by the University of California Santa Barbara and the NSF DMS under grant number 2111303 and 2340631. S. Tang extends gratitude to Ben Adcock for valuable discussions on LASSO. Additionally, a portion of this research was conducted during visits by JAC and ST to the Simons Institute for the Theory of Computing. LM wants to thank Ben Hambly and Markus Schmidtchen for their helpful comments and suggestions. 

\printbibliography

\end{document}